\documentclass[11pt]{report}
\usepackage{tikz}
\usepackage{amsmath,amsthm,amssymb}
\usepackage{enumerate}
\newtheorem{theorem}{Theorem}[section]
\newtheorem{lem}[theorem]{Lemma}
\newtheorem{cor}[theorem]{Corollary}
\newtheorem{prop}[theorem]{Proposition}
\theoremstyle{remark}
\newtheorem{rem}[theorem]{Remark}
\newtheorem{ex}[theorem]{Example}
\theoremstyle{definition}
\newtheorem{defn}[theorem]{Definition}

\begin{document}
\begin{center}
\Large{\textbf{Some combinatorial models for reduced expressions in Coxeter groups}}
\end{center}

\begin{center}
by
\end{center}

\begin{center}
\textbf{Hugh Denoncourt}
\end{center}

\begin{center}
B.A., University of New Hampshire, 1997
\end{center}

\begin{center}
A thesis submitted to the
\end{center}

\begin{center}
Faculty of the Graduate School of the
\end{center}

\begin{center}
University of Colorado in partial fulfillment
\end{center}

\begin{center}
of the requirements for the degree of
\end{center}

\begin{center}
Doctor of Philosophy
\end{center}

\begin{center}
Department of Mathematics
\end{center}

\begin{center}
2009
\end{center}
\newpage
\noindent
Denoncourt, Hugh (Ph.D., Mathematics)\\ \\
Some combinatorial models for reduced expressions in Coxeter groups\\ \\
Thesis directed by Prof. Richard M. Green\\ \\

\begin{center}
\textbf{Abstract}
\end{center}

\vspace{0.4cm}

\noindent
Stanley's formula for the number of reduced expressions of a
permutation regarded as a Coxeter group element raises the
question of how to enumerate the reduced expressions of an
arbitrary Coxeter group element.  We provide a framework for
answering this question by constructing combinatorial objects
that represent the inversion set and the reduced expressions
for an arbitrary Coxeter group element.  The framework also
provides a formula for the length of an element formed by
deleting a generator from a Coxeter group element.  Fan and
Hagiwara, et al$.$ showed that for certain Coxeter groups, the
short-braid avoiding elements characterize those elements that
give reduced expressions when any generator is deleted from a
reduced expression.  We provide a characterization that holds
in all Coxeter groups.  Lastly, we give applications to the
freely braided elements introduced by Green and Losonczy,
generalizing some of their results that hold in simply-laced
Coxeter groups to the arbitrary Coxeter group setting.
\newpage
\begin{center}
\textbf{Dedication}
\end{center}

\vspace{1.5cm}

This thesis is dedicated to Mom and Dad.
\newpage
\begin{center}
\textbf{Acknowledgements}
\end{center}

\vspace{1.5cm}

\noindent
I would especially like to thank Richard Green for being my advisor, for
giving me great advice, especially with respect to the craft of writing
research mathematics, but mostly I would like to thank him for his patience.
I would also like to Nathaniel Thiem for being my second reader and all other
members of my committee for serving on the committee.  I would also like to
thank any and all fellow graduate students who have been there for me over
these adventurous years.

\tableofcontents

\chapter{Introduction}
\label{c:introchap}

Let $W$ be an arbitrary Coxeter group with generating set $S$ of involutions.  The geometric representation of $W$ is a canonically determined faithful representation of $W$ on a vector space $V$ of dimension $|S|$.  Associated to the geometric representation is a special subset $\Phi$ of $V$, called the root system of $W$, upon which $W$ acts.  The root system splits into ``positive'' and ``negative'' roots, and it is known that a Coxeter group element $w$ is determined by the set $\Phi(w)$ of positive roots sent negative by $w$.\\ \\
The minimal length representations of Coxeter group elements with respect to the generators in $S$ are called reduced expressions.  The number of generators in a minimal length representation of a Coxeter group element $w$ is called the length of $w$.\\ \\
In this thesis, we study two problems associated with the combinatorics of reduced expressions.  The first problem is to construct correspondences between the set of reduced expressions for an arbitrary Coxeter group element $w$ and sets of combinatorial objects related to the set $\Phi(w)$.  The second problem is to determine the reduction in length caused by deleting a generator from a reduced expression.\\ \\
A general correspondence for reduced expressions of arbitrary Coxeter groups was given by Tits in \cite{tits2}, but that correspondence involves the Coxeter complex instead of the set $\Phi(w)$.  In \cite{balancedtableaux}, Edelman and Greene gave a correspondence between the set of reduced expressions for a special type of element in a type $A$ Coxeter group and what they call ``balanced tableaux''.  These tableaux can be viewed as special labelings of the set $\Phi(w)$.  A more general correspondence was given by Kra\'skiewicz in \cite{reducedweyl}, in which he gave a correspondence between the set of reduced expressions for $w$ and what he calls standard $w$-tableaux.  While more general than the Edelman-Greene correspondence, Kra\'skiewicz's correspondence only applies to Coxeter groups that are finite Weyl groups.\\ \\
For an arbitrary element $w$ of an arbitrary Coxeter group $W$, we introduce incidence structures with ``betweenness'' relations on the set of positive roots $\Phi^+$ and the set $\Phi(w)$.  The main theorem of this thesis, Theorem~\ref{t:maintheorem}, gives several correspondences between the set of reduced expressions for $w$ and sets of combinatorial objects that are naturally compatible with these two incidence structures.\\ \\
Additionally, Theorem~\ref{t:deletiontheorem} uses the incidence structure on $\Phi(w)$ to determine the reduction in length caused by deleting a generator from a reduced expression.  We use Theorem~\ref{t:deletiontheorem} to give a characterization of the fully covering elements of $W$, which are the elements such that deletion of any generator from any reduced expression results in a reduced expression.  Such elements were characterized by Fan for finite Weyl groups in \cite{shortbraid} and by Hagiwara, et al$.$ for FC-finite simply-laced Coxeter groups in \cite{hagiwara}.  Our characterization applies in the setting of arbitrary Coxeter groups.\\ \\
Lastly, we study the freely braided elements introduced by Green and Losonczy for simply-laced Coxeter groups in \cite{fbI}.  We generalize their definition to the setting of arbitrary Coxeter groups and characterize the freely braided elements in terms of statistics determined by the element.\\ \\
This thesis is organized as follows:\\ \\
In Chapter $1$, we give the preliminaries associated to Coxeter groups and the geometric representation for a Coxeter group $W$ acting on a vector space $V$ with basis determined by $S$.  Associated to the geometric representation are important subsets of the vector space $V$ that $W$ acts upon.  These are the root system $\Phi$ and its associated positive system $\Phi^+$ and negative system $\Phi^-$.\\ \\
Chapter $2$ begins with the preliminaries for reduced expressions and their associated root sequences.  In Corollary~\ref{c:deletion}, we show that the length of an element $w$ obtained by deleting a generator from a reduced expression $\textbf{x}$ for $w$ is given by finding the number of roots in the root sequence of $\textbf{x}$ sent negative by a reflection associated to the deletion.  Towards making this calculation in Chapter $5$, we introduce the notion of a local Coxeter system associated to a set $\Lambda$ of roots.  In this thesis, we primarily use the case where $\Lambda$ consists of two roots, which determines local Coxeter systems that we call ``dihedral subsystems''.\\ \\
In Chapter $3$ we study the dihedral subsystems of $W$ by analyzing how the roots associated to a dihedral subsystem are formed.  We introduce sequences of the roots in a dihedral subsystem that make transparent the restrictions on the order in which these roots must occur in a root sequence.  Such restrictions are given by a property of subsets of roots called ``biconvexity''.  Proposition~\ref{p:biconvexstructure} characterizes the biconvex subsets of $\Phi^+$ in terms of their intersections with dihedral subsystems.\\ \\
In Chapter $4$ we construct a discrete incidence structure with betweenness relations for the set of positive roots $\Phi^+$ that is compatible with the sequences of roots introduced in Chapter $3$.  For an arbitrary $w \in W$ we also construct a discrete (and necessarily finite) incidence structure for the set $\Phi(w)$.  We then give correspondences between reduced expressions for $w$ and information placed ``on top'' of these structures, which is summarized in Theorem~\ref{t:maintheorem}.\\ \\
In Chapter $5$, Corollary~\ref{c:localreflect} shows that the sequences of roots in a dihedral subsystem derived in Chapter $3$ can be used to determine which reflections applied to a positive root will send that positive root to a negative root.  This information is naturally encoded in the discrete incidence structure on $\Phi(w)$ introduced in Chapter $4$, so in Theorem~\ref{t:deletiontheorem}, we use the incidence structure to determined the effect of deleting a generator from a reduced expression upon the length of the resulting Coxeter group element.  We then use Theorem~\ref{t:deletiontheorem} to characterize the fully covering elements of a Coxeter group, which are the elements such that deletion of any generator from any reduced expression always results in a reduced expression.  We then introduce the freely braided elements of Green and Losonczy.  Definition~\ref{d:freelybraided} generalizes their definition to the setting of arbitrary Coxeter groups.  Theorem~\ref{t:fbcharacterization} characterizes the elements $w$ that are freely braided in terms of the number of commutation classes of $w$ and the number of contractible long inversion sets of $w$.
\section{Preliminaries}
We define a \emph{Coxeter system} to be a pair $(W,S)$ consisting of a group $W$ and a finite set of generators $S \subseteq W$ subject only to relations of the form
\begin{equation*}
(ss')^{m_{s,s'}} = 1,
\end{equation*}
where $m_{s,s} = 1$ and $m_{s,s'} > 1$ for $s \neq s'$ in $S$.  In case no relation occurs for a pair $s,s'$, we make the convention that $m_{s,s'} = \infty$.  Thus, to specify a Coxeter system $(W,S)$ is to specify a finite set $S$ and a matrix $m$ indexed by $S$ with entries in $\mathbb{Z} \cup \{\infty\}$ satisfying the above conditions on $m$.  We call the matrix $m$ the \emph{Coxeter matrix of $(W,S)$}.  See \cite[Section 5.1]{humph} for details.\\ \\
Let $S^*$ denote the free monoid on $S$.  The elements of $S^*$ will be written as finite sequences.  There is a natural surjective monoid morphism $\phi:S^* \rightarrow W$ given by $\phi(s_1,\ldots,s_n) = s_1 \cdots s_n$.  We say that $\textbf{x} \in S^*$ is an \emph{expression for $w \in W$} if $\phi(\textbf{x}) = w$.  In general, given any finite sequence $\overline{\alpha}$, the \emph{length} $\ell(\overline{\alpha})$ of $\overline{\alpha}$ is the number of entries in the finite sequence.  Thus, for $\textbf{x} = (s_1,\ldots,s_n) \in S^*$, we say that $\ell(\textbf{x}) = n$.  The \emph{length} of an element $w \in W$, denoted $\ell(w)$, is the minimum of the lengths of the sequences in $\phi^{-1}(\{w\})$.  If $w = \phi(\textbf{x})$ and $\ell(\textbf{x}) = \ell(w)$, then we say that $\textbf{x}$ is a \emph{reduced expression} for $w$.  We denote the set of all reduced expressions for $w$ by $\mathcal{R}(w)$.\\ \\
We also sometimes specify a group element $w \in W$ as a product $$w = w_1w_2 \cdots w_k$$ where each $w_i \in W$ is called a \emph{factor of $w$}.  The product $w_1w_2 \cdots w_k$ is also called an \emph{expression for $w$}.\\ \\
For $k \geq 0$, we write $(s,t)_k$ for the length $k$ sequence $(s,t,s,\ldots)$ in $S^*$ that begins with $s$ and alternates between $s$ and $t$.  Similarly, if $u,v \in W$, we write $(uv)_k$ for the alternating product $uvu\cdots$ with $k$ factors.  If $k = 0$, we adopt the convention that $(s,t)_k$ and $(uv)_k$ represent the identity element of $S^*$ and $W$, respectively.\\ \\
Let $V$ be the real vector space with basis $\Delta = \{\alpha_s\,:\,s \in S\}$.  Associated to $(W,S)$ is the symmetric bilinear form $B$ on $V$ determined by $B(\alpha_s,\alpha_t) = -\cos(\pi/m_{s,t})$ for all $s,t \in S$.  We define an action of $W$ on $V$ by requiring that $s(\alpha_t) = \alpha_t - 2B(\alpha_s,\alpha_t)\,\alpha_s$ and extending linearly.  It follows that the action of $s$ on an arbitrary $v \in V$ takes the same form.  It can also be checked that $B$ is invariant under the action of the generators in $S$, and hence under the action of any $w \in W$.  This action gives rise to a canonical representation $\sigma:W \rightarrow GL_n(V)$ called the \emph{geometric representation of $W$}.  For details, see \cite[Section 5.3]{humph}.\\ \\
We call the set $\Phi = \{w(\alpha_s)\,:\,s \in S,w \in W\}$ the \emph{root system} of $(W,S)$, and we call the elements of $\Phi$ \emph{roots}.  Given $s \in S$, we call $\alpha_s$ a \emph{simple root} and $\Delta$ the \emph{simple system} of $(W,S)$.  Let $\Phi^+$ be the set of all $\alpha \in \Phi$ expressible as a nonnegative linear combination of the simple roots.  We call the set $\Phi^+$ the \emph{positive system} of $(W,S)$ and we call the elements of $\Phi^+$ \emph{positive roots}.  Let $\Phi^- = -\Phi^+$.  We call the set $\Phi^-$ the \emph{negative system} of $(W,S)$ and we call the elements of $\Phi^-$ \emph{negative roots}.  It is known (see \cite[Section 5.4]{humph}) that $\Phi = \Phi^+ \cup \Phi^-$ and that the union is disjoint.\\ \\
Let $\alpha, \beta \in \Phi^+$.  Suppose that $a,b$ are nonnegative real numbers such that $a,b$ are not both $0$ and suppose that $a\alpha + b\beta \in \Phi$.  Then it follows easily that $a \alpha + b\beta \in \Phi^+$.  Similarly, a root that is a nonnegative linear combination of negative roots is itself a negative root.  It is a consequence of the $W$-invariance of $B$ that all roots are of unit length (i.e. $\alpha \in \Phi$ implies $B(\alpha,\alpha) = 1$).  Hence, if $\alpha \in \Phi$, then $\pm \alpha$ are the only scalar multiples of $\alpha$ in $\Phi$.  It follows that distinct positive roots are not scalar multiples of one another.  For details, see \cite[Section 5.4]{humph}.\\ \\
Once we specify a Coxeter system $(W,S)$, we obtain the geometric representation of $W$ and hence the associated simple system, positive (and negative) system, and the root system associated to $\sigma$.  The terminology we use is the same as in \cite[Section 1.3]{humph} except that we allow $W$ (and hence $\Phi$) to be infinite.\\ \\
If $\alpha = w(\alpha_s)$ for some $s \in S$, then we write $s_\alpha = wsw^{-1}$ and we call $s_{\alpha}$ a \emph{reflection}.  We denote the set of all reflections by $R$.  The formula for a reflection $s_\alpha$ applied to an arbitrary vector $\lambda \in V$ is given by
\begin{equation} \label{e:reflection}
s_{\alpha}(\lambda) = \lambda - 2B(\lambda,\alpha)\alpha.
\end{equation}
Observe that $s_{\alpha}$ is independent of the choice of $w$ and $s$, and that the action of $s_\alpha$ on $V$ takes the same form as that of a simple reflection.  That is, the action of $s_{\alpha}$ on $V$ sends $\alpha$ to its negative and since $B(\alpha,\alpha) \neq 0$, the orthogonal complement of $\text{span}(\{\alpha\})$ is a subspace of $V$ of codimension $1$ (i.e. a hyperplane) and this subspace is fixed pointwise by $s_{\alpha}$.  Also, we have $s_{\alpha} = s_{-\alpha}$ for any $\alpha \in \Phi$.  See \cite[Section 5.7]{humph} for details.\\ \\
To each $w \in W$ one can associate the set $\Phi(w) = \Phi^+ \cap w^{-1}(\Phi^-)$, which we call the \emph{inversion set of $w$}.  This is the set of positive roots sent to negative roots by $w$.  It is well known (see \cite[Proposition 5.6]{humph}) that $\Phi(w)$ has $\ell(w)$ elements and that $w$ is uniquely determined by $\Phi(w)$.\\ \\
We let $\mathbb{N}_0$ denote the set of all natural numbers containing zero.
\chapter{Root sequences, labelings, and reflection subgroups}
\label{c:braidmovechapt}
\section{Root Sequences and Labelings}
Let $(W,S)$ be a Coxeter system, let $w \in W$ and let $\textbf{x} = (s_1,\ldots,s_n)$ be an expression for $w$ (i.e. $\phi(\textbf{x}) = w$).  We form the sequence $\overline{\theta}(\textbf{x}) = (\theta_1,\ldots, \theta_n)$ given by $\theta_k = s_n \cdots s_{k+1}(\alpha_{s_k})$ for $1 \leq k \leq n$ (with $\theta_n$ understood to be $\alpha_{s_n}$), which we call the \emph{root sequence} of $\textbf{x}$.  It is well-known (cf. \cite[Exercise 5.6(1)]{humph}) that if $\textbf{x}$ is reduced, then the entries of the root sequence are precisely the elements of $\Phi(w)$.\\ \\
We use the following basic results from the theory of Coxeter groups, so we reproduce them here for convenience.
\begin{prop}  \label{p:basiccoxeter}
Let $(W,S)$ be a Coxeter system with root system $\Phi$.  Then:
\begin{enumerate}[(1)]
\item If $\alpha, \beta \in \Phi^+$ and $\beta = w(\alpha)$ for some $w \in W$, then $ws_\alpha w^{-1} = s_\beta$.
\item Let $w \in W$ and $s \in S$.  Then $\ell(ws) = \ell(w) \pm 1$.
\item If $s \in S$, then $s$ sends $\alpha_s$ to $-\alpha_s$, but permutes the remaining positive roots.
\item Let $w \in W$ and $\alpha \in \Phi^+$.  Then $\ell(ws_\alpha) > \ell(w)$ if and only if $w(\alpha) \in \Phi^+$.
\item Let $w = s_1 \cdots s_n$ ($s_i \in S$).  Suppose $s_\alpha$ satisfies $\ell(ws_\alpha) < \ell(w)$.  Then there is an index $i$ for which $ws_\alpha = s_1 \cdots \widehat{s_i} \cdots s_n$ ($s_i$ is omitted).
\item If $w = s_1 \cdots s_n$ and $\ell(w) < n$, then there exist indices $i$ and $j$ such that $$w = s_1 \cdots \widehat{s_i} \cdots \widehat{s_j} \cdots s_n.$$
\item Let $w, w' \in W$.  If $\Phi(w) = \Phi(w')$, then $w = w'$.
\item Let $w \in W$ and let $\textbf{x}$ be a reduced expression for $w$.  Then the entries of $\overline{\theta}(\textbf{x})$ are precisely the elements of $\Phi(w)$.
\item Let $w \in W$ and let $\textbf{x}$ and $\textbf{x}'$ be reduced expressions for $w$.  Suppose $\overline{\theta}(\textbf{x}) = \overline{\theta}(\textbf{x}')$.  Then $\textbf{x} = \textbf{x}'$.
\end{enumerate}
\end{prop}
\begin{proof}
See \cite[Lemma 5.7]{humph} for a proof of statement $(1)$.  See \cite[Proposition 5.2]{humph} for a proof of statement $(2)$.  See \cite[Proposition 5.6a]{humph} for a proof of statement $(3)$. See \cite[Proposition 5.7]{humph} for a proof of statement $(4)$.  See \cite[Theorem 5.8]{humph} for a proof of statement $(5)$.  See \cite[Corollary 5.8]{humph} for a proof of statement $(6)$.  Statement $(7)$ follows from \cite[Proposition 2]{coxeterorderings}.  For statement $(8)$, see Exercise $1$ of \cite[Section 5.6]{humph} or \cite[Lemma 4.3]{bourbakicoxeter}.  Statement $(9)$ can be proven by a straightforward induction argument, and is explicitly mentioned in \cite[Section 1.2]{fbII}.
\end{proof}
\noindent
Statement $(5)$ is called the \emph{strong exchange condition}, while statement $(6)$ is called the \emph{deletion condition}.  It is well known that if $W$ is a group generated by a set $S$ of involutions, then $(W,S)$ is a Coxeter system if and only if $(W,S)$ satisfies the deletion condition if and only if $(W,S)$ satisfies the strong exchange condition (see \cite[Theorem 1.5.1]{bb}, for example).\\ \\
We wish to make our root sequence manipulations precise.  Thus, given two finite sequences $\overline{\alpha} = (\alpha_1,\ldots,\alpha_k)$ and $\overline{\beta} = (\beta_1,\ldots,\beta_{k'})$ of roots we define the multiplication $\overline{\alpha}\overline{\beta}$ by concatenation of sequences.  In other words,
\begin{equation*}
\overline{\alpha}\overline{\beta} = (\alpha_1,\ldots,\alpha_k,\beta_1,\ldots,\beta_{k'}).
\end{equation*}
Also, given $w \in W$, we let $w$ act on a sequence $\overline{\theta} = (\theta_1,\ldots,\theta_k)$ by
\begin{equation*}
w[\overline{\theta}] = (w(\theta_1),\ldots,w(\theta_k)).
\end{equation*}
If we form the free monoid over the set $\Phi$, we find that our multiplication of finite sequences of roots is isomorphic to the multiplication of the free monoid $\Phi^*$.  We do not pursue this except to point out that our sequence multiplication satisfies both the left and right cancelation properties.  The next two lemmas record cancelation properties of our sequence multiplication, which are well known and trivial.
\begin{lem} \label{l:cancelation}
Let $\overline{\alpha}$, $\overline{\beta}$, and $\overline{\gamma}$ be sequences of roots.  If $\overline{\alpha} \overline{\beta} = \overline{\alpha} \overline{\gamma}$, then $\overline{\beta} = \overline{\gamma}$.  Similarly, if $\overline{\alpha} \overline{\gamma} = \overline{\beta} \overline{\gamma}$, then $\overline{\alpha} = \overline{\beta}$.
\end{lem}
\begin{proof}
This is clear.
\end{proof}
\begin{lem} \label{l:likecancelation}
Let $\overline{\alpha},\overline{\beta},\overline{\gamma},\overline{\delta}$ be sequences of roots and suppose $\overline{\alpha} \, \overline{\gamma} = \overline{\beta}  \, \overline{\delta}$.  If $\ell(\overline{\alpha}) = \ell(\overline{\beta})$, then $\overline{\alpha} = \overline{\beta}$ and $\overline{\gamma} = \overline{\delta}$.
\end{lem}
\begin{proof}
This is clear.
\end{proof}
\begin{lem} \label{l:basicrs}
Let $\textbf{a}, \textbf{b} \in S^*$, let $\textbf{x} = \textbf{a} \textbf{b}$, and let $w^{-1} = \phi(\textbf{b})$.  Then $\overline{\theta}(\textbf{x}) = w[\overline{\theta}(\textbf{a})] \, \overline{\theta}(\textbf{b})$.
\end{lem}
\begin{proof}
Starting at the $\ell(\textbf{a}) + 1$ entry of $\overline{\theta}(\textbf{x})$, the calculations are identical to those of $\overline{\theta}(\textbf{b})$.  If $\textbf{a} = (s_1,\ldots,s_m)$ and $\textbf{b} = (s_1',\ldots,s_n')$, then for $i < \ell(\textbf{a}) + 1$, the $i$-th entry of $\overline{\theta}(\textbf{x})$ is $s_n' \cdots s_1' s_m \cdots s_{i+1}(\alpha_{s_i}) = w(s_n \cdots s_{i+1}(\alpha_{s_i}))$, which is $w$ applied to the $i$-th entry of $\overline{\textbf{a}}$.
\end{proof}
\noindent
It will be useful to know what change occurs in a root sequence upon substituting one expression for another in a word where the expressions represent the same element $w \in W$.  For more complex decompositions of a word than what appears in Lemma~\ref{l:basicrs}, we need the following terminology:
\begin{defn} \label{d:factordecomp}
Let $\textbf{x} = \textbf{a}_1 \cdots \textbf{a}_k$ be a product in $S^*$.  If $\overline{\theta}(\textbf{x}) = \overline{\theta_1} \cdots \overline{\theta_k}$ and $$\ell(\overline{\theta_i}) = \ell(\textbf{a}_i)$$ for all $i$ such that $1 \leq i \leq k$, then we say that the equation $\overline{\theta}(\textbf{x}) = \overline{\theta_1} \cdots \overline{\theta_k}$ is \emph{the decomposition of $\overline{\theta}(\textbf{x})$ respecting $\textbf{a}_1 \cdots \textbf{a}_k$}.
\end{defn}
\begin{rem}
Note that the $\overline{\theta_i}$'s appearing in the decomposition are to denote subsequences of the root sequence of $\textbf{x}$, but are not themselves root sequences of the $\textbf{a}_i$.  Rather, the $\overline{\theta_i}$ are given by some $w \in W$ acting on the root sequence of $\textbf{a}_i$, which can be made precise by repeatedly applying Lemma~\ref{l:basicrs} to the factorization of $\textbf{x}$ in Definition~\ref{d:factordecomp}.
\end{rem}
\begin{lem} \label{l:subrs}
Let $\textbf{x} = \textbf{a}\textbf{b}\textbf{c}$ and $\textbf{x}' = \textbf{a}\textbf{b}'\textbf{c}$.  Let $v^{-1} = \phi(\textbf{c})$ and $u^{-1} = \phi(\textbf{b})$.   Suppose $\phi(\textbf{b}) = \phi(\textbf{b}')$ and set $\overline{\theta_2}' = v[\overline{\theta}(\textbf{b}')]$.  If $\overline{\theta}(\textbf{x}) = \overline{\theta_1}\,\overline{\theta_2}\,\overline{\theta_3}$ is the decomposition of $\overline{\theta}(\textbf{x})$ respecting $\textbf{a}\textbf{b}\textbf{c}$, then $\overline{\theta}(\textbf{x}') = \overline{\theta_1}\, \overline{\theta_2}' \,\overline{\theta_3}$ is the decomposition of $\overline{\theta}(\textbf{x}')$ respecting $\textbf{a}\textbf{b}'\textbf{c}$.  Thus, substituting equivalent expressions changes only the corresponding substituted portion of the root sequence.
\end{lem}
\begin{proof}
From Lemma~\ref{l:basicrs}, we have $$\overline{\theta}(\textbf{x}) = vu[\overline{\theta}(\textbf{a})] \, v[\overline{\theta}(\textbf{b})] \, \overline{\theta}(\textbf{c}).$$  Thus, by Lemma~\ref{l:likecancelation} and Definition~\ref{d:factordecomp}, $\overline{\theta_1} = vu[\overline{\theta}(\textbf{a})]$, $\overline{\theta_2} = v [\overline{\theta}(\textbf{b})]$, and $\overline{\theta_3} = \overline{\theta}(\textbf{c})$.  By Lemma~\ref{l:basicrs} again,
\begin{equation*}
\overline{\theta}(\textbf{x}') = vu[\overline{\theta}(\textbf{a})] \, v[\overline{\theta}(\textbf{b}')] \, \overline{\theta}(\textbf{c}),
\end{equation*}
which gives the result.
\end{proof}
\noindent
The root sequence of a reduced expression $\textbf{x}$ is related to the strong exchange condition by the following well known fact, which specifies the reflection that will remove the $k$-th generator of $\textbf{x}$.
\begin{lem} \label{l:reflectdelete}
Suppose that $\textbf{x} = (s_1,\ldots,s_n)$ is a reduced expression for $w$, $\overline{\theta}(\textbf{x})$ is the root sequence for $\textbf{x}$, and $\theta_k$ is the $k$-th root of $\overline{\theta}(\textbf{x})$.  Then $ws_{\theta_k} = s_1 \cdots \widehat{s_k} \cdots s_n$.
\end{lem}
\begin{proof}
This is a straightforward calculation based on Proposition~\ref{p:basiccoxeter}(1), which appears in the proof of \cite[Theorem 5.8]{humph}.
\end{proof}
\noindent
We can state a version of the previous lemma that gives the effect of deleting a generator upon the root sequence of a reduced expression.  It is probably well known, but we have not found it in the literature, so we prove it here.
\begin{lem} \label{l:gendeletion}
Let $(W,S)$ be a Coxeter system and let $w \in W$.  Let $s \in S$, let $\textbf{x} = \textbf{a} \, (s) \, \textbf{b}$ be a reduced expression for $w$, and let $\overline{\theta}(\textbf{x})$ be the root sequence of $\textbf{x}$.  Let $j = \ell(\textbf{a}) + 1$ and let $D_j(\textbf{x}) = \textbf{x}' = \textbf{a} \, \textbf{b}$ denote the expression obtained by deleting the $j$-th letter from $\textbf{x}$.  Let $\overline{\theta}(\textbf{x}) = \overline{\theta_1} (\theta_j) \overline{\theta_2}$ be the decomposition of $\overline{\theta}(\textbf{x})$ respecting $\textbf{a}\,(s)\,\textbf{b}$.  Then $\overline{\theta}(\textbf{x}') = s_{\theta_j} [\overline{\theta_1}] \, \overline{\theta_2}$ is the decomposition of $\overline{\theta}(\textbf{x}')$ respecting $\textbf{a} \, \textbf{b}$.
\end{lem}
\begin{proof}
Let $\overline{\theta}(\textbf{x}') = \overline{\theta_1'} \, \overline{\theta_2'}$ be the decomposition of $\overline{\theta}(\textbf{x}')$ respecting $\textbf{a} \textbf{b}$.  By Definition~\ref{d:factordecomp}, we have $\overline{\theta_2} = \overline{\theta_2'} = \overline{\theta}(\textbf{b})$.  Let $u^{-1} = \phi((s)\textbf{b})$.  By Lemma~\ref{l:reflectdelete}, $(s_{\theta_j} u)^{-1} = \phi(\textbf{b})$.  Thus, by Lemma~\ref{l:basicrs} and Definition~\ref{d:factordecomp}, we have
\begin{equation*}
\overline{\theta_1} = u [\overline{\theta}(\textbf{a})] \text{ and }
\overline{\theta_1'} = s_{\theta_j} u [\overline{\theta}(\textbf{a})].
\end{equation*}
It follows that $\overline{\theta_1'} = s_{\theta_j} [\overline{\theta_1}]$.
\end{proof}
\noindent
The root sequence of an expression is related to the deletion condition by the following result, which asserts that the number of times we must apply the deletion condition to an expression to reach a reduced expression is given by the number of negative roots in the root sequence.
\begin{prop} \label{p:rsdeletion}
Let $(W,S)$ be a Coxeter system.  Let $w \in W$, let $\textbf{x} = (s_1,\ldots,s_n)$ be an expression for $w$, and let $\overline{\theta}(\textbf{x})$ be the root sequence of $\textbf{x}$.   If $d$ is the number of negative entries in $\overline{\theta}(\textbf{x})$, then $\ell(\textbf{x}) = \ell(w) + 2d$.  In particular, $\textbf{x}$ is reduced if and only if $\overline{\theta}(\textbf{x})$ consists only of positive roots.
\end{prop}
\begin{proof}
The proof is by induction on $d$.  If $d = 0$, then the entries in $\overline{\theta}(\textbf{x})$ are precisely the positive roots sent negative by $w$, so $\ell(\textbf{x}) = \ell(w)$.  Let $d > 0$ and let $i$ be the greatest index such that $(s_i, \ldots, s_n)$ is not reduced.  Since $(s_n, \ldots, s_{i+1})$ is a reduced expression but $(s_n, \ldots, s_i)$ is not reduced, we have $s_n \cdots s_{i+1}(\alpha_{s_i}) \in \Phi^-$ by Proposition~\ref{p:basiccoxeter}(4).  Thus if we let $\textbf{a} = (s_1, \ldots, s_{i-1})$, $\textbf{b} = (s_i, \ldots, s_n)$, and we let $\overline{\theta}(\textbf{x}) = \overline{\theta_1} \, \overline{\theta_2}$ be the decomposition of $\overline{\theta}(\textbf{x})$ respecting $\textbf{a} \textbf{b}$, then the first root (and only the first root) of $\overline{\theta_2}$ is negative.  Thus, $\overline{\theta_1}$ contains $d - 1$ negative roots.\\ \\
Since $(s_n)$ is a reduced expression, $i < n$.  Let $u^{-1} = \phi(s_{i+1}, \ldots, s_n)$, so that $(us_i)^{-1} = \phi(\textbf{b})$ and $\ell(u) = n - i$.  By Proposition~\ref{p:basiccoxeter}(2), $\ell(us_i) = \ell(w) \pm 1$, while the choice of $i$ implies $\ell(us_i) = \ell(u) - 1$.  Thus, since the expression $(s_i, \ldots, s_n)$ is not reduced, the deletion condition implies there exists an expression $\textbf{b}' = (s_i', \ldots, s_{n-2}')$ such that $\phi(\textbf{b}') = (us_i)^{-1} = \phi(\textbf{b})$.  Thus $\textbf{x}' = \textbf{a} \textbf{b}'$ is an expression for $w$ such that $\ell(\textbf{x}') = \ell(\textbf{x}) - 2$.  Since $\ell(us_i) = \ell(u) - 1$ and $\ell(\textbf{b}') = n - i - 1$, $\textbf{b}'$ is a reduced expression for $(us_i)^{-1}$.  If $\overline{\theta_2}'$ is the root sequence for $\textbf{b}'$, then $\overline{\theta}(\textbf{x}') = \overline{\theta_1} \, \overline{\theta_2}'$ is the root sequence for $\textbf{x}'$ by Lemma~\ref{l:basicrs}.  The sequence $\overline{\theta_2}'$ has no negative roots since $\textbf{b}'$ is a reduced expression, and $\overline{\theta_1}$ contains $d - 1$ negative roots, so $\overline{\theta}(\textbf{x}')$ contains $d - 1$ negative roots.  The induction hypothesis now implies
\begin{equation*}
\ell(\textbf{x}) - 2 = \ell(\textbf{x}') = \ell(w) + 2(d - 1),
\end{equation*}
which then implies $\ell(\textbf{x}) = \ell(w) + 2d$.
\end{proof}
\begin{cor} \label{c:deletion}
Let $(W,S)$ be a Coxeter system and let $w \in W$.  Let $\textbf{x}$ be a reduced expression for $w$, let $s \in S$, and let $\textbf{x} = \textbf{a} \, (s) \, \textbf{b}$.  Let $j = \ell(\textbf{a}) + 1$ and $\textbf{x}' = D_j(\textbf{x}) = \textbf{a} \, \textbf{b}$.  Let $w' = \phi(\textbf{x}')$ and let $\overline{\theta}(\textbf{x}) = (\theta_1,\ldots,\theta_{\ell(w)})$ be the root sequence of $\textbf{x}$.  If $d$ is the number of roots $\theta_k$, $1 \leq k < j$, such that $s_{\theta_j}(\theta_k) \in \Phi^-$, then $\ell(w') = \ell(w) - 2d - 1$.  In particular, $\textbf{x}'$ is a reduced expression for $w'$ if and only if $s_{\theta_j}(\theta_k) \in \Phi^+$ for every $k$ such that $1 \leq k < j$.
\end{cor}
\begin{proof}
If we let $\overline{\theta}(\textbf{x}) = \overline{\theta_1} (\theta_j) \overline{\theta_2}$ be the decomposition of $\overline{\theta}(\textbf{x})$ respecting $\textbf{a} (s) \textbf{b}$, then by Lemma~\ref{l:gendeletion}, $\overline{\theta}(\textbf{x}') = s_{\theta_j} [\overline{\theta_1}] \overline{\theta_2}$.  Thus the roots in the root sequence $\overline{\theta}(\textbf{x}')$ not already in $\overline{\theta}(\textbf{x})$ are those of the form $s_{\theta_j}(\theta_k)$, where $1 \leq k < j$.  Since $\textbf{x}$ is a reduced expression, the roots in $\overline{\theta}(\textbf{x})$ are positive, so the result follows by Proposition~\ref{p:rsdeletion}.
\end{proof}
\noindent
While we will use root sequences in the sequel, we shall also have some use for labelings of roots.  In Definition~\ref{d:standardencoding}, we introduce a labeling associated to an expression $\textbf{x}$, which carries the exact same information as the root sequence of $\textbf{x}$.  We give correspondences in Chapter $4$ between certain types of labelings of $\Phi^+$ and reduced expressions for a Coxeter group element $w$.  While the correspondences can also be stated in terms of root sequences, we prefer the labeling approach for technical reasons.
\begin{defn} \label{d:labeling}
Let $(W,S)$ be a Coxeter system and and let $\Lambda \subseteq \Phi$.  A function $T:\Lambda \rightarrow \Bbb{N}_0$ is called a \emph{labeling of $\Lambda$}.  The \emph{support of $T$} is the set $\text{supp}(T)$ of all $\lambda \in \Lambda$ such that $T(\lambda) \neq 0$.  We call a labeling \emph{sequential} if $\text{supp}(T)$ is finite and $T(\text{supp}(T)) = \{1,\ldots,|\text{supp}(T)|\}$.
\end{defn}
\begin{rem}
The definition of labeling allows the restriction of a labeling $T$ to its support to be non-injective.  If we view the range of $T$ as the labels of a labeling, then the definition allows the labels to skip numbers and be larger than $|\text{supp}(T)|$.  By contrast, the nonzero labels of a sequential labeling $T$ occur in order (sequentially), and the restriction of a sequential labeling $T$ to its support is necessarily injective.
\end{rem}
\begin{defn} \label{d:standardencoding} \cite[Definition 2.4]{reducedweyl}
Let $\textbf{x}$ be a reduced expression for $w$ and let $$\overline{\theta}(\textbf{x}) = (\theta_1,\ldots,\theta_{\ell(w)})$$ be the root sequence for $\textbf{x}$.  Define $T_{\textbf{x}}:\Phi^+ \rightarrow \mathbb{N}_0$ by $T_{\textbf{x}}(\theta_k) = k$ and $T_{\textbf{x}}(\lambda) = 0$ for $\lambda \not\in \Phi(w)$.  We call this labeling of $\Phi^+$ the \emph{standard encoding of $\textbf{x}$} and we say that \emph{$T_{\textbf{x}}$ encodes $\textbf{x}$}.
\end{defn}
\noindent
The following terminology is quite standard.  In addition to the linear orders on $\Phi(w)$ determined by the root sequence of an expression, we shall have some use for partial orders on $\Phi(w)$ in Chapter $5$.
\begin{defn}
A \emph{partial order} is a binary relation $\leq$ on a set $X$ that is:
\begin{enumerate}[(1)]
\item reflexive:  for all $x \in X$, $x \leq x$;
\item antisymmetric:  for all $x,y \in X$, if $x \leq y$ and $y \leq x$, then $x = y$;
\item transitive:  for all $x,y,z \in X$, if $x \leq y$ and $y \leq z$, then $x \leq z$.\\ \\
We call the pair $(X,\leq)$ a \emph{partially ordered set}.
A \emph{linear order} (or \emph{total order}) \emph{on $X$} is a partial order $\leq$ on $X$ that satisfies the following additional property:
\item For all $x,y \in X$, either $x \leq y$ or $y \leq x$.
\end{enumerate}
\noindent
If properties $(1)-(4)$ are satisfied, then we call the pair $(X,\leq)$ a \emph{linearly ordered set} (or a \emph{totally ordered set}).
\end{defn}
\section{Reflection subgroups and the roots associated to them}
At this point we fix a Coxeter system $(W,S)$ and the root system $\Phi$ associated to the geometric representation of $(W,S)$.  Thus we also fix $\Phi^+$, $\Phi^-$, and $\Delta$ for the remainder of this thesis.\\ \\
Subgroups of $W$ generated by a set of (not necessarily simple) reflections are called \emph{reflection subgroups} and were investigated by Deodhar in \cite{reflectiongen} and Dyer independently in \cite{dyerreflection}.  In both papers it is shown that a reflection subgroup is a Coxeter system in its own right.  Thus, if $R$ is the set of reflections of $(W,S)$, $R' \subseteq R$, and $W' = \langle R' \rangle$, then there exists a (canonically determined) set $S' \subseteq R'$ such that $(W', S')$ is a Coxeter system.  If
\begin{equation*}
R(w) := \{t \in R \ : \ \ell(tw) < \ell(w) \},
\end{equation*}
then the canonical set $S'$ is the set of all $t' \in R$ such that $R(t') \cap W' = \{t'\}$ (see \cite[Section 8.2]{humph} for a summary or the papers \cite{reflectiongen} and \cite{dyerreflection} for details).
\begin{theorem}[\textbf{Deodhar, Dyer}] \label{t:dyerdeodhar}
Let $(W,S)$ be a Coxeter system and $R$ be the set of reflections in $W$.  Let $W'$ be the subgroup of $W$ generated by $R' \subseteq R$.  Then, with $S'$ defined as above,
\begin{enumerate}[(1)]
\item  Every reflection in $R \cap W'$ is of the form $w' s' w'^{-1}$, where $w' \in W'$ and $s' \in S'$;
\item  The pair $(W',S')$ is a Coxeter system.
\end{enumerate}
\end{theorem}
\begin{proof}
See \cite[Theorem 3.3(i)]{dyerreflection} for a proof of statement $(1)$.  See \cite[Theorem]{reflectiongen} or \cite[Theorem 3.3]{dyerreflection} for a proof of statement $(2)$.
\end{proof}
\noindent
Proposition~\ref{p:basiccoxeter}(1) induces a 1--1 correspondence between the set $R$ of reflections in $(W,S)$ and the set $\Phi^+$ of positive roots.  Thus, from an arbitrary subset $\Lambda \subseteq \Phi^+$, we can form the reflection subgroup $(W[\Lambda],S[\Lambda])$ by letting $R' = \{s_\lambda \, : \, \lambda \in \Lambda\}$.  We associate to such an arbitrary $\Lambda$, sets $\Phi[\Lambda]$, $\Phi^+[\Lambda]$, $\Phi^-[\Lambda]$, and $\Delta[\Lambda]$ that are analogs of the root system $\Phi$, the positive system $\Phi^+$, the negative system $\Phi^-$, and the simple system $\Delta$, respectively.  We would like the following properties to hold for the associated root system in $(W[\Lambda],S[\Lambda])$:\\ \\
$(1)$  The positive root analogs are positive roots in the original Coxeter system.  Similarly, the negative root analogs are negative roots in the original Coxeter system.  That is, $\Phi^+[\Lambda] \subseteq \Phi^+$ and  $\Phi^-[\Lambda] \subseteq \Phi^-$. \\ \\
$(2)$  The reflections associated to the simple root analogs form the generating set $S[\Lambda]$ of $(W[\Lambda],S[\Lambda])$.  That is, $S[\Lambda] = \{s_\delta \, : \, \delta \in \Delta[\Lambda]\}$.\\ \\
$(3)$  Every root in $\Phi[\Lambda]$ is either a nonnegative linear combination of roots in $\Delta[\Lambda]$ or a nonpositive linear combination of roots in $\Delta[\Lambda]$.\\ \\
$(4)$  Every root in $\Phi[\Lambda]$ is either a positive root or a negative root, but not both.  In other words, $\Phi[\Lambda] = \Phi^+[\Lambda] \cup \Phi^-[\Lambda]$ and the union is disjoint.\\ \\
$(5)$  Every root in $\Phi[\Lambda]$ can be obtained by applying an element in $W[\Lambda]$ to a root in $\Delta[\Lambda]$.\\ \\
Deodhar establishes properties $(1)$ through $(4)$ in the proof of \cite[Theorem]{reflectiongen}.  We give his constructions in Definitions~\ref{d:localrootsystem} and \ref{d:simplelocal}.  Property $(1)$ is immediate from Definition~\ref{d:localrootsystem}.  Property $(2)$ is given by Lemma~\ref{l:basiclocal}(1).  Property $(3)$ follows from Lemma~\ref{l:basiclocal}(3) and Definition~\ref{d:localrootsystem}(5).  Property $(4)$ is Lemma~\ref{l:basiclocal}(2).  Lastly, Lemma~\ref{l:reflectionform}, gives property $(5)$.\\ \\
Below we give the above mentioned subsets of $\Phi$ that we use from Deodhar's proof of \cite[Theorem]{reflectiongen}.
\begin{defn} \label{d:localrootsystem}
Let $\Lambda \subseteq \Phi^+$ and let $(W', S') = (W,S)[\Lambda]$ be the reflection subgroup generated by $R' = \{ s_\lambda \in R \ : \ \lambda \in \Lambda\}$.  Then we set:
\begin{enumerate}[(1)]
\item  $W[\Lambda] = W'$,
\item  $S[\Lambda] = S'$,
\item $\Phi[\Lambda] = \{ w(\gamma) \ : \ w \in W', \, s_\gamma \in R' \} = \{ w(\gamma) \ : \ w \in W', \gamma \in \Lambda\}$,
\item  $\Phi^+[\Lambda] = \Phi[\Lambda] \cap \Phi^+$,
\item  $\Phi^-[\Lambda] = \Phi[\Lambda] \cap \Phi^-$.
\end{enumerate}
\noindent
We call $\Phi[\Lambda]$, $\Phi^+[\Lambda]$, and $\Phi^-[\Lambda]$, respectively, the \emph{root system, positive system, and negative system}, respectively, of $(W,S)[\Lambda]$.  We call $S[\Lambda]$ the \emph{canonical generators} for $(W,S)[\Lambda]$.
\end{defn}
\noindent
Note that the root system, positive system, and negative system analogs are obtained from roots associated to the reflections present in the reflection subgroup $W'$.  The simple system analog we will use can be similarly obtained from $S[\Lambda]$, but to identify the simple root analogs we will need more specific information regarding such roots.  The relation $\leqq_{\Lambda}$ in the next definition is from \cite[Section 2]{reflectiongen}.
\begin{defn} \label{d:preorder}
Let $\Lambda \subseteq \Phi^+$.  We say that $\mu \in \Phi^+[\Lambda]$ is a \emph{positive summand of $\lambda \in \Phi^+[\Lambda]$} if there exists $a > 0$ and a finite set $\{\gamma_i \in \Lambda\}_{i \in \mathcal{I}}$ of roots such that $\lambda = a\mu + \sum a_{\gamma_i} \gamma_i$, where $a_{\gamma_i} \geq 0$ for all $i \in \mathcal{I}$.  If $\mu \in \Phi^+[\Lambda]$ is a positive summand of $\lambda \in \Phi^+[\Lambda]$, then we write $\mu \leqq_{\Lambda} \lambda$.
\end{defn}
\begin{rem}
It is noted in \cite[Section 2]{reflectiongen} that $\leqq_{\Lambda}$ is a preorder relation (i.e. a reflexive and transitive relation).  Thus there is an equivalence relation $\sim_{\Lambda}$ given by $\lambda \sim_{\Lambda} \mu$ if and only if $\lambda \leqq_{\Lambda} \mu$ and $\mu \leqq_{\Lambda} \lambda$ that gives rise to a partial order on the equivalence classes.  For our purposes, we only need the preorder relation.
\end{rem}
\begin{ex}
To illustrate the preorder $\leqq_{\Lambda}$, we let $(W,S)$ be the Coxeter system of type $B_2$ given by $S = \{s,t\}$ and $m_{s,t} = 4$.  It is known that the positive system is given by $\Phi^+ = \{\alpha_s, \sqrt{2} \alpha_s + \alpha_t, \alpha_s + \sqrt{2} \alpha_t, \alpha_t\}$.  We have $$\alpha_s \leqq_{\Phi^+} \sqrt{2}\alpha_s + \alpha_t$$ by writing $$\sqrt{2} \alpha_s + \alpha_t = (\sqrt{2} \alpha_s) + \alpha_t,$$ which makes $\alpha_s$ a positive summand of $\sqrt{2}\alpha_s + \alpha_t$.  Similarly $$\alpha_s \leqq_{\Phi^+} \alpha_s + \sqrt{2}\alpha_t.$$  Since $$\sqrt{2} \alpha_s + \alpha_t = \sqrt{2}/2 (\alpha_s + \sqrt{2}\alpha_t) + \sqrt{2}/2 \alpha_s,$$ we have $$\sqrt{2} \alpha_s + \alpha_t \leqq_{\Phi^+} \alpha_s + \sqrt{2}\alpha_t.$$  By a similar calculation, we also have $$\alpha_s + \sqrt{2}\alpha_t \leqq_{\Phi^+} \sqrt{2} \alpha_s + \alpha_t.$$  Note that $\alpha_s$ and $\alpha_t$, the simple roots of $(W,S)$, are incomparable and minimal with respect to $\leqq_{\Phi^+}$.
\end{ex}
\begin{defn} \label{d:simplelocal}
Let $$\Delta[\Lambda] = \{\lambda \in \Phi^+[\Lambda] \, : \, (\mu \in \Phi^+[\Lambda] \text{ and } \mu \leqq_\Lambda \lambda) \Rightarrow \lambda = \mu\}.$$  We call the set $\Delta[\Lambda]$ the \emph{simple system of $(W,S)[\Lambda]$}.
\end{defn}
\noindent
The next lemma tells us that our constructions of $\Phi[\Lambda]$, $\Phi^+[\Lambda]$, $\Phi^-[\Lambda]$, and $\Delta[\Lambda]$, respectively, behave like a root system, a positive system, a negative system, and a simple system, respectively.
\begin{lem} \label{l:basiclocal}
Let $\Delta[\Lambda]$ be the simple system of $(W,S)[\Lambda]$.  Then: \\ \\
(1)  $S[\Lambda] = \{s_\delta \, : \, \delta \in \Delta[\Lambda]\}$.  \\
(2)  $\Phi[\Lambda] = \Phi^+[\Lambda] \cup (-\Phi^+[\Lambda]) = \Phi^+[\Lambda] \cup \Phi^-[\Lambda]$, where the unions are disjoint.\\
(3)  Any $\lambda \in \Phi^+[\Lambda]$ can be written as a sum $$\sum_{\delta \in \Delta[\Lambda]} a_\delta \delta,$$ where $a_\delta \geq 0$ for all $\delta \in \Delta[\Lambda]$.
\end{lem}
\begin{proof}
See ``step 4'' of the proof of \cite[Theorem]{reflectiongen}.
\end{proof}
\noindent
For the next lemma, recall that $B(\lambda, \lambda) = 1$ for all $\lambda \in \Phi$.
\begin{lem} \label{l:pmroots}
Let $\alpha, \beta \in \Phi$ and suppose $s_\alpha = s_\beta$.  Then $\alpha = \pm \beta$.
\end{lem}
\begin{proof}
Since $s_\alpha(\beta) = s_\beta(\beta)$, we have $\beta \, - \, 2B(\alpha, \beta)\alpha = \beta \, - \, 2B(\beta,\beta)\beta$.  Thus
\begin{equation*}
-2B(\alpha,\beta) \alpha = -2\beta,
\end{equation*}
so that $\alpha = \mu \beta$ for some scalar $\mu$.  It follows that $$1 = B(\alpha,\alpha) = B(\mu \beta, \mu \beta) = \mu^2.$$  Thus $\mu = \pm 1$, so $\alpha = \pm \beta$.
\end{proof}
\begin{lem} \label{l:reflectionform}
Let $\Lambda \subseteq \Phi^+$ and let $(W',S') = (W,S)[\Lambda]$ be the reflection subgroup generated by $R' = \{s_\lambda \in R \ : \ \lambda \in \Lambda\}$.  Then $$\Phi[\Lambda] = \{ w( \delta ) \ : \ w \in W[\Lambda], \, \delta \in \Delta[\Lambda]\}.$$
\end{lem}
\begin{proof}
\noindent
For any $w \in W[\Lambda]$ and $\delta \in \Delta[\Lambda]$, we have $w(\delta) \in \Phi[\Lambda]$ by Definition~\ref{d:localrootsystem}(3).\\ \\
By Definition~\ref{d:localrootsystem}(3), to prove the converse, we need to show that for any $s_\lambda \in R'$ and $u \in W[\Lambda]$, we have $u(\lambda) = v(\delta)$ for some $\delta \in \Delta[\Lambda]$ and $v \in W[\Lambda]$.  Thus, we let $s_\lambda \in R'$.  By Theorem~\ref{t:dyerdeodhar}(1), $s_\lambda = w s_\mu w^{-1}$, where $w \in W[\Lambda]$ and $s_\mu \in S[\Lambda]$.  By Lemma~\ref{l:basiclocal}(1), $s_\mu = s_\delta$, where $\delta \in \Delta[\Lambda]$.  By Proposition~\ref{p:basiccoxeter}(1), $s_\lambda = s_{w(\delta)}$.  Lemma~\ref{l:pmroots} then implies that $\lambda = \pm w(\delta)$.\\ \\
Thus, if $u \in W[\Lambda]$, then $u(\lambda) = u w(\delta)$ or $u(\lambda) = u w(s_\delta(\delta))$.  In either case, $u(\lambda) = v(\delta)$ for some $v \in W[\Lambda]$ and $\delta \in \Delta[\Lambda]$.
\end{proof}
\noindent
If a positive system of a reflection subgroup contains a simple root $\alpha_i$, then $\alpha_i$ must be simple relative to the root system $\Phi[\Lambda]$.  This is the content of the next lemma.
\begin{lem}  \label{l:simplerootlocal}
Let $\Lambda \subseteq \Phi^+$ and let $\alpha_i \in \Phi^+[\Lambda]$ be a simple root relative to $(W,S)$.  Then $\alpha_i \in \Delta[\Lambda]$.
\end{lem}
\begin{proof}
Suppose $\lambda \in \Phi^+[\Lambda]$ is such that $\alpha_i = a \lambda + \sum a_\delta \delta$, where $a > 0$ and the sum on the right hand side of the equation is a nonnegative linear combination of positive roots.  Then, since $\alpha_i$ is simple, we must have $\lambda = \alpha_i$, $a = 1$, and $a_\delta = 0$ for all $\delta$ in the sum.  By Definition~\ref{d:simplelocal}, $\alpha_i \in \Delta[\Lambda]$.
\end{proof}
\section{Local Coxeter systems}
It was noted in \cite[Remark 3.2]{dyerbruhat} that associated to each pair of distinct positive roots $\alpha$ and $\beta$ there exists a unique maximal dihedral reflection subgroup containing the reflections $s_\alpha$ and $s_\beta$. In general, given an arbitrary set $\Lambda$ of roots, we can form a reflection subgroup that is maximal with respect to the roots that lie in the subspace spanned by $\Lambda$.  The next definition gives a name to these reflection subgroups.  This section is devoted to the useful properties that do not hold for arbitrary reflection subgroups that do hold in these maximal reflection subgroups.
\begin{defn}  \label{d:localsystem}
Let $(W,S)$ be Coxeter system.  Let $\Lambda \subseteq \Phi$.  We call the Coxeter system $(W,S)[\text{span}(\Lambda) \cap \Phi^+]$ the \emph{local Coxeter system of $\Lambda$}, which we denote by $(W,S)_\Lambda$.  Similarly, we set:\\ \\
(1)  $W_\Lambda = W[\text{span}(\Lambda) \cap \Phi^+]$,\\
(2)  $S_\Lambda = S[\text{span}(\Lambda) \cap \Phi^+]$,\\
(3)  $\Phi_\Lambda = \Phi[\text{span}(\Lambda) \cap \Phi^+]$,\\
(4)  $\Phi^+_\Lambda = \Phi^+[\text{span}(\Lambda) \cap \Phi^+]$,\\
(5)  $\Phi^-_\Lambda = \Phi^-[\text{span}(\Lambda) \cap \Phi^+]$,\\
(6)  $\Delta_\Lambda = \Delta[\text{span}(\Lambda) \cap \Phi^+]$.\\ \\
We call $\Phi_\Lambda$ the \emph{local root system of $\Lambda$}.  We call $\Phi^+_\Lambda$ the \emph{local positive system of $\Lambda$}.  We call $\Phi^-_\Lambda$ the \emph{local negative system of $\Lambda$}.  We call $\Delta_\Lambda$ the \emph{local simple system of $\Lambda$}.  If $|\Lambda| = 2$, then we say that $(W,S)_\Lambda$ is a \emph{dihedral subsystem}.
\end{defn}
\begin{rem}
Note that we do not allow arbitrary subsets of reflections in our formation of local Coxeter systems.  A simple example of what we wish to disallow occurs with the Coxeter system $(W,S)$ of type $B_2$ specified by $S = \{s, t\}$ and $m_{s,t} = 4$.  We have
\begin{equation*}
\Phi^+ = \{\alpha_s, \sqrt{2} \alpha_s + \alpha_t, \alpha_s + \sqrt{2}\alpha_t, \alpha_t\}.
\end{equation*}
Using $R' = \{s,tst\}$, the reflection subgroup generated by $R'$ is the subgroup $W' = \{1,s,tst,stst\}$.  Letting $\Lambda = \{\alpha_s, \alpha_s + \sqrt{2}\alpha_t\}$ (the positive roots associated to the reflections in $R'$), the local Coxeter system is $(W,S)$ itself since $\text{span}(\Lambda) \cap \Phi^+ = \Phi^+$.  By Definition~\ref{d:localrootsystem} parts $(3)$ and $(4)$, we can apply the elements of $W'$ to those of $\Lambda$ to check that $\Phi^+[\Lambda] = \Lambda$, which is properly contained in $\Phi^+$.  In the sequel we show that the positive root systems of distinct dihedral subsystems intersect in at most one root.  This example shows that this fact does not hold for the positive root systems of distinct reflection subgroups generated by two reflections.
\end{rem}
\noindent
For local Coxeter systems, parts $(3)$, $(4)$ and $(5)$ of Definition~\ref{d:localrootsystem} can be made more explicit.  Thus one advantage to using local Coxeter systems over arbitrary reflection subgroups is that we know what the root system $\Phi_\Lambda$ is without reference to the reflection subgroup $W'$.
\begin{lem} \label{l:localsystemform}
Let $\Lambda \subseteq \Phi^+$.  Then we have\\ \\
(1)  $\Phi_\Lambda = \mathrm{span}(\Lambda) \cap \Phi$;\\
(2)  $\Phi^+_\Lambda = \mathrm{span}(\Lambda) \cap \Phi^+$;\\
(3)  $\Phi^-_\Lambda = \mathrm{span}(\Lambda) \cap \Phi^-$.
\end{lem}
\begin{proof}
Let $\lambda \in \text{span}(\Lambda) \cap \Phi$.  Since $s_\alpha = s_{-\alpha}$ for all $\alpha \in \Phi$, we have that $s_\lambda$ is one of the generating reflections of $W[\text{span}(\Lambda) \cap \Phi^+]$.  Using $w = 1$ in part $(3)$ of Definition~\ref{d:localrootsystem}, we have $\lambda \in \Phi_\Lambda$, so $\text{span}(\Lambda) \cap \Phi \subseteq \Phi_\Lambda$.  Conversely, each reflection of the form $s_\gamma$ with $\gamma \in \text{span}(\Lambda) \cap \Phi^+$ fixes (set-wise) the space $\text{span}(\Lambda)$ by the reflection formula $(\ref{e:reflection})$, so $\Phi_\Lambda \subseteq \text{span}(\Lambda) \cap \Phi$, which proves the first assertion.  Intersecting both sides of the equation $\Phi_\Lambda =  \text{span}(\Lambda) \cap \Phi$ with $\Phi^+$ and $\Phi^-$, we obtain the stated formulas for $\Phi^+_\Lambda$ and $\Phi^-_\Lambda$ by Definition~\ref{d:localsystem} and Definition~\ref{d:localrootsystem}.
\end{proof}
\begin{lem} \label{l:dihedralsize}
Let $\Lambda = \{\alpha, \beta\} \subseteq \Phi^+$ be such that $\alpha \neq \beta$.  Let $(W,S)_\Lambda$ be the local Coxeter system of $\Lambda$.  Then $|\Delta_\Lambda| = 2$.
\end{lem}
\begin{proof}
See \cite[Remark 3.2]{dyerbruhat}.
\end{proof}
\begin{rem} \label{r:dihedral}
Proposition 4.5.4 of \cite{bb} states that a subgroup generated by distinct reflections $s_{\gamma}$, $s_{\delta}$ is a finite dihedral group if $|B(\gamma,\delta)| < 1$ and an infinite dihedral group if $B(\gamma,\delta) \leq -1$.  This, combined with Lemma~\ref{l:dihedralsize}, shows that the subgroup $W_{\{\alpha,\beta\}}$ of $W$ is a dihedral group.  Thus, the name ``dihedral subsystem'' (as given in Definition~\ref{d:localsystem}) is justified whenever $|\Lambda| = 2$.
\end{rem}
\begin{defn} \label{d:canonicalroot}
Let $\alpha,\beta \in \Phi^+$ be distinct.  If $\Delta_{\{\alpha,\beta\}} = \{\gamma,\delta\}$ and $S_{\{\alpha,\beta\}} = \{s_\gamma,s_\delta\}$ then we call $\gamma$ and $\delta$ the \emph{canonical simple roots}, and $s_\gamma,s_\delta$ the \emph{canonical generators}, for the dihedral Coxeter system $(W,S)_{\{\alpha,\beta\}}$.
\end{defn}
\noindent
The next lemma shows that a dihedral subsystem and its associated root system is uniquely determined by the canonical simple roots.  As a result, any two distinct roots within a dihedral subsystem determine the same dihedral subsystem.
\begin{lem} \label{l:closure}
Let $\alpha,\beta \in \Phi^+$ be distinct and let $\gamma$ and $\delta$ be the canonical simple roots for the local Coxeter system $(W,S)_{\{\alpha,\beta\}}$.  Then:
\begin{enumerate}[(1)]
\item  $(W,S)_{\{\alpha,\beta\}} = (W,S)_{\{\gamma,\delta\}}$.
\item  $\Delta_{\{\alpha,\beta\}} = \Delta_{\{\gamma, \delta\}}$.
\item  $\Phi^+_{\{\alpha,\beta\}} = \Phi^+_{\{\gamma,\delta\}}$.
\item  $\Phi_{\{\alpha,\beta\}} = \Phi_{\{\gamma, \delta\}}$.
\end{enumerate}
\end{lem}
\begin{proof}
Recall that distinct positive roots are not scalar multiples of one another.  Thus, since $\gamma, \delta \in \text{span}(\{\alpha,\beta\})$ and $\text{span}(\{\gamma,\delta\})$ is two-dimensional, we have
\begin{equation*}
\text{span}(\{\alpha,\beta\}) = \text{span}(\{\gamma,\delta\}),
\end{equation*}
so the generating set of reflections is the same for both localizations.  By Theorem~\ref{t:dyerdeodhar}(2), the generating set of reflections uniquely determines the reflection subgroup and its canonical generating set. Therefore we have $(W,S)_{\{\alpha,\beta\}} = (W,S)_{\{\gamma,\delta\}}$.  The remaining equations follow from Definition~\ref{d:localrootsystem}.
\end{proof}
\noindent
Lemmas~\ref{l:preservespan} shows that when an arbitrary $w \in W$ acts on a dihedral subsystem, the result is a dihedral subsystem.  Lemmas~\ref{l:preservesignspan} and \ref{l:wclosure} show that if $w$ does not send the canonical simple roots negative, then $w$ maps positive roots to positive roots and canonical simple roots to canonical simple roots in the resulting subsystem.
\begin{lem} \label{l:preservespan}
Let $w \in W$ and let $\alpha, \beta \in \Phi$ be distinct roots.  Then
\begin{equation*}
w(\Phi_{\{\alpha,\beta\}}) = \Phi_{\{w(\alpha),w(\beta)\}}.
\end{equation*}
\end{lem}
\begin{proof}
Since $W$ acts on $\Phi$ and each $w \in W$ acts as a vector space automorphism on $V$, we have $a\alpha + b\beta \in \Phi$ if and only if $aw(\alpha) + bw(\beta) \in \Phi$.  We also have $$w(\text{span}(\{\alpha,\beta\})) = \text{span}(\{w(\alpha),w(\beta)\})$$ because $w$ acts as a vector space automorphism on $V$.  Thus $$w(\text{span}(\{\alpha,\beta\}) \cap \Phi) = \text{span}(\{w(\alpha),w(\beta)\}) \cap \Phi.$$  By Lemma~\ref{l:localsystemform}, the result follows.
\end{proof}
\begin{lem} \label{l:preservesignspan}
Let $\gamma$ and $\delta$ be the canonical simple roots of $\Phi_{\{\alpha,\beta\}}$ and suppose that $w(\gamma), w(\delta) \in \Phi^+$.  Then
\begin{equation*}
w\left(\Phi^+_{\{\alpha,\beta\}}\right) = \Phi^+_{\{w(\alpha),w(\beta)\}} \text{ and } w\left(\Phi^-_{\{\alpha,\beta\}}\right) = \Phi^-_{\{w(\alpha),w(\beta)\}}.
\end{equation*}
\end{lem}
\begin{proof}
Let $w \in W$ and suppose $\gamma$ and $\delta$ are the canonical simple roots of $\Phi_{\{\alpha,\beta\}}$ and that $w(\gamma), w(\delta) \in \Phi^+$.  Let $\lambda \in \Phi^+_{\{\alpha,\beta\}}$.  We have $\lambda = c\gamma + d\delta$ for some $c,d \geq 0$ by Lemma~\ref{l:basiclocal}(3).  Since $w(\gamma),w(\delta) \in \Phi^+$, we have $w(\lambda) \in \Phi^+$.  Similarly, we have that $\lambda \in \Phi^-_{\{\alpha,\beta\}}$ implies $w(\lambda) \in \Phi^-$.  By Lemma~\ref{l:preservespan}, $\lambda \in \Phi_{\{\alpha,\beta\}}$ implies $w(\lambda) \in \Phi_{\{w(\alpha), w(\beta)\}}$.  Thus, by parts (4) and (5) of Definition~\ref{d:localrootsystem}, this proves $$w\left(\Phi^+_{\{\alpha, \beta\}}\right) \subseteq \Phi^+_{\{w(\alpha),w(\beta)\}}$$
\begin{center}
and
\end{center}
$$w\left(\Phi^-_{\{\alpha,\beta\}}\right) \subseteq \Phi^-_{\{w(\alpha),w(\beta)\}}.$$
To prove the converse, suppose $\lambda \in \Phi^+_{\{w(\alpha),w(\beta)\}}$. By Lemma~\ref{l:basiclocal}(3), we have $\lambda = aw(\alpha) + bw(\beta)$ for some $a,b \in \mathbb{R}$ and $\lambda \in \Phi^+$.  Then we have $w^{-1}(\lambda) \in \Phi_{\{\alpha,\beta\}}$ by Lemma~\ref{l:preservespan}.  Suppose towards a contradiction that $w^{-1}(\lambda) \in \Phi^-_{\{\alpha,\beta\}}$.  By Lemma~\ref{l:basiclocal}, there exist $c,d \in \mathbb{R}$ such that $c,d \leq 0$ and $w^{-1}(\lambda) = c\gamma + d\delta$.  Since $$w\left(\Phi^-_{\{\alpha,\beta\}}\right) \subseteq \Phi^-_{\{w(\alpha),w(\beta)\}},$$ we have $w(w^{-1}(\lambda)) \in \Phi^-$.  This implies $\lambda \in \Phi^-$ and since it was assumed that $\lambda \in \Phi^+$, we have a contradiction.  Thus $w^{-1}(\lambda) \in \Phi^+$.  Now, since $w^{-1}(\lambda) \in \Phi^+_{\{\alpha,\beta\}}$, it follows that $$\lambda \in w\left(\Phi^+_{\{\alpha,\beta\}}\right).$$  By a similar argument, we have $\lambda \in \Phi^-_{\{w(\alpha),w(\beta)\}}$ implies $\lambda \in w\left(\Phi^-_{\{\alpha,\beta\}}\right)$.
\end{proof}
\begin{lem} \label{l:wclosure}
Let $w \in W$ and let $\alpha$ and $\beta$ be distinct positive roots such that $\gamma$ and $\delta$ are the canonical simple roots for $(W,S)_{\{\alpha,\beta\}}$.  Suppose that $w(\gamma), w(\delta) \in \Phi^+$.  Then the canonical simple roots of $(W,S)_{\{w(\alpha),w(\beta)\}}$ are $w(\gamma)$ and $w(\delta)$.
\end{lem}
\begin{proof}
Set $\Lambda = \text{span}(\{\alpha,\beta\}) \cap \Phi^+$ and $\Lambda' = \text{span}(\{w(\alpha),w(\beta)\}) \cap \Phi^+$.  We assumed that $w(\gamma), w(\delta) \in \Phi^+$, so Lemma~\ref{l:preservesignspan} implies that if $\mu \in \Phi^+_{\{w(\alpha),w(\beta)\}}$ then we may write $\mu = w(\lambda)$ for some $\lambda \in \Phi^+_{\{\alpha,\beta\}}$.\\ \\
Suppose that $w(\lambda) \leqq_{\Lambda'} w(\gamma)$.  Then, by Definition~\ref{d:preorder} there exists $a > 0$ and a finite set of coefficients $\{a_\nu\}_{\nu \in \mathcal{I}}$ such that $a_\nu \geq 0$ for all $\nu \in \mathcal{I}$ satisfying
\begin{equation*}
w(\gamma) = aw(\lambda) + \sum_{\nu \in \mathcal{I}} a_\nu w(\nu) = w(a\lambda + \sum_{\nu \in \mathcal{I}} a_\nu \nu),
\end{equation*}
where each $\nu \in \Phi^+_{\{\alpha,\beta\}}$.  Thus we have
\begin{equation*}
\gamma = a\lambda + \sum_{\nu \in \mathcal{I}} a_\nu \nu,
\end{equation*}
so that $\lambda \leqq_\Lambda \gamma$.  Since $\gamma \in \Delta_{\{\alpha,\beta\}}$, we have $\lambda = \gamma$ by Definition~\ref{d:simplelocal}.  Thus $w(\lambda) = w(\gamma)$, and it follows that $w(\gamma) \in \Delta_{\{w(\alpha),w(\beta)\}}$.  The same argument applies to $w(\delta)$ so that $\{w(\delta),w(\gamma)\} \subseteq \Delta_{\{w(\alpha),w(\beta)\}}$.  By Lemma~\ref{l:dihedralsize}, $\left|\Delta_{\{w(\alpha),w(\beta)\}}\right| = 2$, so $$\Delta_{\{w(\alpha),w(\beta)\}} = \{w(\gamma),w(\delta)\}.$$
\end{proof}
\noindent
Recall that the expression $(s_\gamma s_\delta)_k$ has $k$ factors (as opposed to $2k$ factors) and does not end in $s_\delta$ if $k$ is odd.
\begin{lem} \label{l:inversionroots}
Let $\Delta_{\{\alpha,\beta\}} = \{\gamma,\delta\}$.  Then the roots in $\Phi_{\{\alpha,\beta\}}$ can be obtained by applying a product of the form $(s_\gamma s_\delta)_k$ or $(s_\delta s_\gamma)_k$ ($k \geq 0$) to $\gamma$ or to $\delta$.
\end{lem}
\begin{proof}
Let $\lambda \in \Phi_{\{\alpha,\beta\}}$.  By Lemma~\ref{l:reflectionform}, we have that $\lambda = w(\theta)$ for some $w \in W_{\{\alpha,\beta\}}$ and $\theta \in \{\gamma,\delta\}$.  Since $W_{\{\alpha,\beta\}}$ is generated by $s_\gamma$ and $s_\delta$ we may express $w$ as a product whose only factors are $s_\gamma$ and $s_\delta$.  If any repeated factors of $s_\gamma$ or $s_\delta$ occur, we may apply the relations $s_\gamma^2 = 1$ and $s_\delta^2 = 1$ until there are no repeated factors.
\end{proof}
\chapter{Dihedral subsystems}  \label{c:dihedralchapt}
The alternating generators expression of Lemma~\ref{l:inversionroots} gives recurrences for the roots of a dihedral subsystem.  These recurrences also naturally determine doubly infinite sequences, which give a natural total ordering on the positive roots of the subsystem.  We then solve the recurrences explicitly in terms of substitutions into Chebyshev polynomials in Lemma~\ref{l:cheb:alphabeta}.  In Lemmas~\ref{l:biconvexity} and \ref{l:infinitebiconvexity}, as well as Corollaries~\ref{c:biconvexityconverse} and \ref{c:infinitebiconvexityconverse}, the solution is used to characterize when a root in one of the doubly infinite sequences lies in the convex cone spanned by two roots from the sequences.  These characterizations are used in Section $3.4$ to construct a total ordering on the positive roots of a dihedral subsystem.  We will find that ``betweenness'' in the ordering detects when a root lies in the convex cone spanned by two other roots and vice versa.\\ \\
In \cite[Section 2]{dyerhecke}, Dyer defines an ordering on the reflections of a dihedral subsystem that corresponds to our total ordering on the positive roots. He outlines an approach for showing the two orders are compatible, which we provide details for in this chapter.\\ \\
In Chapter $5$, building on what we develop in this chapter, we will give conditions that determine whether $s_\theta(\mu)$ is negative in terms of where $\theta$ and $\mu$ lie in the ordering.  The conditions are made precise in Corollary~\ref{c:localreflect}, which is used to prove Theorem~\ref{t:deletiontheorem}.
\section{Convex cones and biconvexity}
\begin{defn} \label{d:convexspan}
Let $A \subseteq V$.  We call the set $\text{c}(A)$ of all nonnegative linear combinations of elements of $A$ the \emph{convex cone spanned by the set $A$}.  If $A = \{\alpha, \beta\}$, then we call $\text{c}(A)= \{a\alpha + b\beta \, : \, a,b \geq 0\}$ the \emph{convex cone spanned by $\alpha$ and $\beta$}.  If $\gamma = a\alpha + b\beta$ where $a,b > 0$, then we say \emph{$\gamma$ strictly lies in the convex cone spanned by $\alpha$ and $\beta$}.
\end{defn}
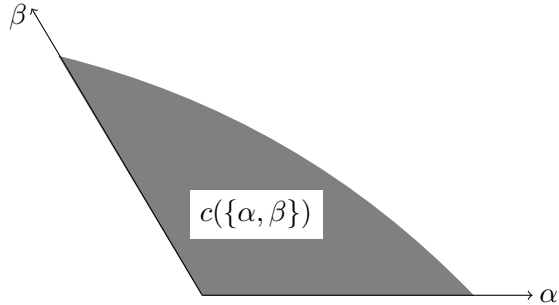
\begin{figure} \label{convexconefig}
\begin{center}
\begin{tikzpicture}[scale=3.6]
\filldraw[fill=gray, draw=gray] (0,0) -- (1,0) arc (44:76:3.2) -- (0,0);
\draw (0.2,0.3) node [fill=white]{$c(\{\alpha,\beta\})$};
\draw [->] (0,0) -- (-0.63,1.06);
\draw (-0.68,1.03) node {$\beta$};
\draw [->] (0,0) -- (1.22,0);
\draw (1.28,0) node {$\alpha$};
\end{tikzpicture}
\caption{The shaded region portrays the convex cone spanned by $\alpha$ and $\beta$.  Included in the convex cone is the indefinite extension into the plane of the shaded region.}
\end{center}
\end{figure}
\noindent
The following lemma is basic, but useful.
\begin{lem} \label{l:convexcone}
Let $\alpha, \beta, \gamma \in V$ be pairwise distinct nonzero vectors such that $\alpha$ is not a scalar multiple of $\beta$.  Suppose $\gamma$ strictly lies in the convex cone spanned by $\alpha$ and $\beta$.  Then $\alpha$ does not lie in the convex cone spanned by $\beta$ and $\gamma$.  Similarly, $\beta$ does not lie in the convex cone spanned by $\alpha$ and $\gamma$.
\end{lem}
\begin{proof}
This is clear.
\end{proof}
\noindent
In \cite[Section 3]{coxeterorderings}, it is shown that given a finite Weyl group $W$, the subsets of positive roots that have the form $\Phi(w)$ for some $w \in W$ are characterized by a property called ``biconvexity''.  Corollary~\ref{c:finitebiconvex}, which is probably folklore, shows that this characterization holds for arbitrary Coxeter groups.
\begin{defn} \label{d:biconvexity}
Let $\Lambda \subset \Phi^+$.  We say that $\Lambda$ is \emph{convex} (or that $\Lambda$ satisfies the \emph{convexity} property) if for every $\alpha, \beta \in \Phi^+$
the following condition holds:
\begin{enumerate}[(1)]
\item If $\alpha, \beta \in \Lambda$ and $\lambda \in \Phi^+$ lies in the convex cone spanned by $\alpha$ and $\beta$, then $\lambda \in \Lambda$.\\ \\
If both $\Lambda$ and $\Phi^+ \setminus \Lambda$ are convex, then we say $\Lambda$ is \emph{biconvex} (or that $\Lambda$ satisfies the \emph{biconvexity} property).  Equivalently, $(1)$ and the following implication holds:
\item If $\alpha, \beta \not\in \Lambda$ and $\lambda \in \Phi^+$ lies in the convex cone spanned by $\alpha$ and $\beta$, then $\lambda \not\in \Lambda$.
\end{enumerate}
\end{defn}
\noindent
Recall that $\Phi(w)$ is the set of positive roots that are sent to a negative root by $w$.
Lemma~\ref{l:inversionbiconvex} is implicit in \cite[Section 3]{coxeterorderings} and stated exactly as we state it here in \cite[Section 2]{fbI} and \cite[Section 2]{nilorbits}.  No proof is given in any of those papers, so we include a proof for the sake of completeness.
\begin{lem} \label{l:inversionbiconvex}
Let $w \in W$.  Then the set $\Phi(w)$ satisfies the biconvexity property.
\end{lem}
\begin{proof}
Recall that $w$ acts as a linear transformation, and that a nonnegative linear combination of positive (respectively, negative) roots is a positive (respectively, negative) root.\\ \\
Let $\alpha, \beta \in \Phi(w)$ and suppose $\lambda$ is a root such that $\lambda = a\alpha + b\beta$ for some $a,b \geq 0$.  By the definition of $\Phi(w)$, $\alpha$ and $\beta$ are positive roots, so $\lambda$ is also a positive root.  Since $w(\alpha)$ and $w(\beta)$ are negative roots, we have that $a w(\alpha) + b w(\beta)$ is also a negative root.  It follows that $\lambda \in \Phi(w)$.\\ \\
Similarly, let $\alpha, \beta \in \Phi^+ \setminus \Phi(w)$ and suppose $\lambda$ is a root such that $\lambda = a\alpha + b\beta$ for some $a,b \geq 0$.  Then $\lambda$ is a positive root since $\alpha$ and $\beta$ are positive roots.  Since $\alpha, \beta \not\in \Phi(w)$, $w(\alpha)$ and $w(\beta)$ are positive roots.  Thus, $w(\lambda)$ is a positive root as well and it follows that $w(\lambda) \not\in \Phi(w)$.
\end{proof}
\section{Chebyshev polynomials of the second kind}
By Lemma~\ref{l:inversionroots}, the roots of a dihedral subsystem can be obtained by applying an alternating product of canonical generators to one of the canonical roots.  We can directly calculate the scalars arising in the linear combinations of the canonical generators by applying the reflection formula (\ref{e:reflection}).  The alternating products give rise to a recurrence that has the same form that the Chebyshev polynomials of the second kind have.  The way in which the alternating products are calculated is independent of the canonical roots themselves, so the Chebyshev polynomials require an evaluation based on the canonical roots to calculate the actual scalars.  In this section, we collect the results about Chebyshev polynomials that we need for later calculations.\\ \\
In what follows, it will be convenient to use sequences indexed by the integers instead of the natural numbers.  A typical recurrence for a sequence of natural numbers expresses a sequence entry in terms of previous entries in the sequence.  For doubly infinite sequences, we also need a ``backward recurrence'', which is a recurrence expressing a sequence entry in terms of entries indexed by larger integers.  In the sequences we use, we can obtain these backward recurrences by rearranging the forward recurrences.\\ \\
The following definition of the Chebyshev polynomials of the second kind agrees with \cite[Definition 1.2]{cheby} except that we shift the indices up by one and extend it to a doubly infinite sequence.
\begin{defn}  \label{d:shifted}
Let $x \in [-1,1]$ and write $x = \cos \theta$ for some $\theta \in [0,\pi]$.  Then for $n \in \mathbb{Z}$, we define
\begin{equation*}
U_n(x) = \frac{\sin (n\theta)}{\sin \theta}.
\end{equation*}
The functions in this doubly infinite sequence are known as the \emph{Chebyshev polynomials of the second kind}.
\end{defn}
\begin{rem}
The discontinuities that occur at $\theta = 0$ and $\theta = \pi$ are removable.  Since the functions $U_n(x)$ turn out to be polynomials, they are continuous and hence, the discontinuities are not present when $U_n(x)$ is viewed as a polynomial.  Thus, $U_n(1) = n$ and $U_n(-1) = (-1)^{n+1} n$ for all $n \in \mathbb{Z}$.
\end{rem}
\noindent
The following lemma is our reason for shifting the indices in the standard definition.
\begin{lem} \label{e:cheb:symmetric}
For any $n \in \mathbb{N}$, we have $U_{-n}(x) = - U_n(x)$.
\end{lem}
\begin{proof}
We have $$U_{-n}(x) = \frac{\sin(-n\theta)}{\sin \theta} = -\frac{\sin(n\theta)}{\sin \theta} = -U_n(x).$$
\end{proof}
\begin{lem} \label{l:cheb}
Let $\{U_n(x)\}_{n \in \mathbb{Z}}$ be the Chebyshev polynomials of the second kind.  Then $U_n(x)$ satisfies the initial conditions
\begin{equation} \label{e:cheb:initial}
U_0(x) = 0, U_1(x) = 1
\end{equation}
and the recurrence
\begin{equation} \label{e:cheb:recurrence}
U_{n+2}(x) = 2x U_{n+1}(x) - U_n(x) \ \ \ \ \ (n \in \mathbb{Z}).
\end{equation}
The associated backward recurrence is given by
\begin{equation} \label{e:cheb:backrecurrence}
U_n(x) = 2x U_{n+1}(x) - U_{n+2}(x).
\end{equation}
\end{lem}
\begin{proof}
Equation (\ref{e:cheb:recurrence}) is \cite[(1.6a)]{cheby}, which is obtained by applying the identity
\begin{equation*}
\sin((n+1)z) + \sin((n-1)z) = 2\cos (z) \sin(nz)
\end{equation*}
to the definition.  We obtain the backward recurrence (\ref{e:cheb:backrecurrence}) by rearranging (\ref{e:cheb:recurrence}).
\end{proof}
\noindent
Although the defining representation for the Chebyshev polynomials of the second kind applies only to the interval $[-1,1]$, the recurrence relations imply that the resulting sequence of functions is a sequence of polynomials with integral coefficients.  Thus, we may extend the domain of $U_n(x)$ to $\mathbb{R}$ (or even $\mathbb{C}$).  In what follows, we will assume that for each function $U_n(x)$, the domain and codomain are both $\mathbb{R}$.  Note that if we obtain an identity using the trigonometric definition, then the identity applies to the extended domain because two polynomials need only agree on a large enough finite set in order for them to agree everywhere.  Thus, if two polynomials agree on $[-1,1]$, an infinite set, then they have the same coefficients and must agree on their extension to $\mathbb{R}$.\\ \\
We will see that $B(\gamma,\delta) \leq -1$ whenever $s_\gamma$ and $s_\delta$ generate an infinite dihedral subsystem.  Thus the condition $a \geq 1$ in the lemmas that follow will be applied when we make substitutions of the form $x = a$, where $a = -B(\gamma,\delta)$.
\begin{lem} \label{l:chebincreasing}
Let $a \geq 1$ be a fixed real number and let $n \geq 1$.  Then the sequence of real numbers $\{U_n(a)\}_{n \geq 1}$ is positive and strictly increasing.
\end{lem}
\begin{proof}
Since $U_1(a) = 1$ and $U_2(x) = 2a$ where $x \geq 1$, we have $U_2(a) > U_1(a)$.  Suppose that $U_{n+1}(a) > U_n(a)$.  Then, by the recurrence (\ref{e:cheb:recurrence}), we have $$U_{n+2}(a) = 2aU_{n+1}(a) - U_n(a) > 2aU_{n+1}(a) - U_{n+1}(a).$$
Thus $U_{n+2}(a) > (2a - 1)U_{n+1}(a) \geq U_{n+1}(a)$ since we assumed $a \geq 1$.  From $U_1(a) > 0$, we get that the sequence $\{U_n(a)\}_{n \geq 1}$ is positive and strictly increasing by induction.
\end{proof}
\begin{lem} \label{l:chebratio}
Let $a \geq 1$ be a fixed real number and for $n \geq 1$, form the sequence of ratios $$r_n = \frac{U_{n+1}(a)}{U_n(a)}.$$  Then we have $r_n$ is a decreasing sequence of positive real numbers such that $r_n \geq 1$ for all $n \geq 1$.
\end{lem}
\begin{proof}
First note that by Lemma~\ref{l:chebincreasing}, the denominator of $r_n$ is nonzero.
For the base case, we prove that $r_1 > r_2$.  By the recurrence relations and initial conditions, we have $U_1(a) = 1$, $U_2(a) = 2a$, and $U_3(a) = 4a^2 - 1$.  Thus $r_1 = 2a$ and $r_2 = 2a - \frac{1}{2a}$.  Since $a \geq 1$, we have $r_1 > r_2 \geq 1$, as desired.  For the inductive step, if we divide both sides of (\ref{e:cheb:recurrence}) by $U_{n+1}(a)$ we have $$\frac{U_{n+2}(a)}{U_{n+1}(a)} = 2a - \frac{U_n(a)}{U_{n+1}(a)}.$$  Thus we have $r_{n+1} = 2a - \frac{1}{r_n}$.  It follows that $r_{n+2} = 2a - \frac{1}{r_{n+1}}$.  If we assume that $r_n > r_{n+1}$, then $r_{n+1} - r_{n+2} = {1 \over {r_{n+1}}} - {1 \over {r_n}}$, so that $r_{n+1} - r_{n+2} > 0$.  Thus, the sequence is decreasing by induction.  Also, if $r_n \geq 1$, then $\frac{1}{r_n} \leq 1$ so that $r_{n+1} = 2a - \frac{1}{r_n} \geq 1$ follows from the assumption that $a \geq 1$.  Thus, by induction, $r_n \geq 1$ for all $n \geq 1$.
\end{proof}
\noindent
The previous lemma is useful because of the next lemma, which will be used to determine when a root in an infinite dihedral subsystem lies in the convex cone spanned by two other roots in the subsystem.\\ \\
We interpret the next lemma in the extended real number system. The natural ordering on the extended reals $\mathbb{R} \cup \{-\infty,+\infty\}$ is determined by the condition that $$-\infty < x < \infty$$ for every $x \in \mathbb{R}$. We interpret a fraction of the form $\frac{a}{0}$, where $a > 0$, as $+\infty$.  Such an interpretation is not standard, but our coefficients are nonnegative (so that we are approaching $0$ from the right) and the inferences we use are consistent with this choice.  In particular, for the hypothesis $\frac{c_1}{d_1} < \frac{c_3}{d_3} < \frac{c_2}{d_2}$ (or its reverse), only one of the fractions can have $0$ as a denominator in order for the inequalities to be strict.  Thus, if $\frac{a}{b} < \frac{c}{d}$ where the denominator $d$ is $0$, we still have $ad - bc < 0$.
\begin{lem} \label{l:conelemma}
Suppose that $\{\gamma,\delta\}$ forms a basis for a two-dimensional subspace of $V$.  Let $\alpha_1 = c_1\gamma + d_1\delta$, $\alpha_2 = c_2\gamma + d_2 \delta$, and $\alpha_3 = c_3 \gamma + d_3 \delta$, where $c_i, d_i \geq 0$ and $(c_i,d_i) \neq (0,0)$ for each $i \in \{1,2,3\}$.  Suppose that either $\frac{c_1}{d_1} < \frac{c_3}{d_3} < \frac{c_2}{d_2}$ or $\frac{c_1}{d_1} > \frac{c_3}{d_3} > \frac{c_2}{d_2}$. Then there exist positive scalars $a,b > 0$ satisfying $a\alpha_1 + b\alpha_2 = \alpha_3$.
\end{lem}
\begin{proof}
The given fractional inequalities imply inequalities involving determinants:
\begin{equation*}
\frac{c_1}{d_1} < \frac{c_3}{d_3} < \frac{c_2}{d_2} \Rightarrow c_1d_2 - d_1c_2 < 0 \text{, } c_1d_3 - d_1c_3 < 0 \text{, and } c_3d_2 - d_3c_2 < 0.
\end{equation*}
\begin{equation*}
\frac{c_1}{d_1} > \frac{c_3}{d_3} > \frac{c_2}{d_2} \Rightarrow c_1d_2 - d_1c_2 > 0 \text{, } c_1d_3 - d_1c_3 > 0 \text{, and } c_3d_2 - d_3c_2 > 0.
\end{equation*}
The solution to the vector equation $a\alpha_1 + b\alpha_2 = \alpha_3$ is given by the matrix equation
\begin{equation*}
\begin{bmatrix}
c_1 & c_2\\
d_1 & d_2
\end{bmatrix}
\begin{bmatrix}
a \\
b
\end{bmatrix}
=
\begin{bmatrix}
c_3 \\
d_3
\end{bmatrix}.
\end{equation*}
Since $c_1d_2 - d_1 c_2 \neq 0$ in either case, the matrix in the matrix equation is invertible.
By Cramer's rule, the solution to this matrix equation is given by
\begin{equation*}
a = \frac{
\begin{vmatrix}
c_3 & c_2 \\
d_3 & d_2
\end{vmatrix}}{
\begin{vmatrix}
c_1 & c_2 \\
d_1 & d_2
\end{vmatrix}} \text{ and } b = \frac{
\begin{vmatrix}
c_1 & c_3 \\
d_1 & d_3
\end{vmatrix}}{
\begin{vmatrix}
c_1 & c_2 \\
d_1 & d_2
\end{vmatrix}}.
\end{equation*}
If $\frac{c_1}{d_1} < \frac{c_3}{d_3} < \frac{c_2}{d_2}$ then the inequalities imply all of the above determinants are negative.  If instead, $\frac{c_1}{d_1} > \frac{c_3}{d_3} > \frac{c_2}{d_2}$, then the inequalities imply all of the above determinants are positive.  In either case, $a,b > 0$.
\end{proof}
\begin{lem} \label{l:infinitedistinct}
Let $a \geq 1$ be a fixed real number.  Then, for $n \in \mathbb{Z}$ the real numbers $U_n(a)$ are pairwise distinct, where $U_n(a) > 0$ if $n \geq 1$, $U_n(a) < 0$ if $n \leq -1$, and $U_n(a) = 0$ if $n = 0$.
\end{lem}
\begin{proof}
By Lemma~\ref{l:chebincreasing}, for $n \geq 1$, we have that $\{U_n(a)\}_{n \geq 1}$ is a strictly increasing sequence of positive real numbers.  Thus, the numbers present in the sequence are pairwise distinct.  Since $U_0(a) = 0$ and $U_{-n}(a) = -U_n(a)$, the result follows.
\end{proof}
\noindent
With the exception of the shift of indices, our proof of the next lemma is identical to the one given for \cite[(2.1)]{udrea} and is included for completeness.
\begin{lem} \label{l:cheb:biconvexity}
Let $n,i,j \in \mathbb{Z}$ and $a \in \mathbb{R}$.  Then we have
\begin{equation} \label{e:cheb:biconvexity}
U_i(a) U_{n+i+j}(a) + U_j(a) U_n(a) = U_{i+j}(a) U_{n+i}(a).
\end{equation}
\end{lem}
\begin{proof}
First suppose $a \in [-1,1]$.  Let $\theta \in [0,\pi]$ be such that $a = \cos \theta$.  Then
\begin{equation*}
\begin{split}
U_i(a) U_{n + i + j}(a) + U_j(a) U_n(a) ={}& \frac{\sin (i\theta)}{\sin \theta} \frac{\sin ((n+i+j) \theta)}{\sin \theta} + \frac{\sin (j\theta)}{\sin \theta} \frac{\sin(n\theta)}{\sin \theta} \\
={}& \frac{\cos((n + j)\theta) - \cos((n + 2i + j)\theta)}{2\sin^2 \theta} \\ &+ \frac{\cos((n - j)\theta) - \cos((n + j)\theta)}{2 \sin^2 \theta}\\
={}& \frac{\cos((n - j)\theta) - \cos((n + 2i + j)\theta)}{2\sin^2 \theta} \\
={}& \frac{\sin((i + j)\theta)}{\sin \theta} \frac{\sin((n + i)\theta)}{\sin \theta}\\
={}& U_{i+j}(x) U_{n+i}(a).
\end{split}
\end{equation*}
The first and last equations are by Definition~\ref{d:shifted}, while the second and fourth equations use the basic identity
\begin{equation*}
\sin(x)\sin(y) = \frac{\cos(x - y) - \cos(x + y)}{2}.
\end{equation*}
Since the left hand side and the right hand side are both polynomials for fixed values of $n,i,j \in \mathbb{Z}$, the fact that the equality holds for all $a \in [-1,1]$ implies that the equality holds for all $a \in \mathbb{R}$.
\end{proof}
\noindent
We will see that $B(\gamma,\delta) = -\cos(\pi/m)$ for some $m < \infty$ whenever $s_\gamma$ and $s_\delta$ generate a finite dihedral subsystem.  Thus, there are two substitutions into the Chebyshev polynomials of the second kind that are of importance to us.  These are substitutions of the form $x = a$, where $a = \cos(\pi/m)$ for $2 \leq m < \infty$ and substitutions of the form $x = a$ where $a \geq 1$.
\begin{lem} \label{l:distinctpairs}
Let $a = \cos(\pi/m)$, where $m \geq 2$.  For $-m \leq n \leq m - 1$, the ordered pairs $$(U_n(a), U_{n+1}(a))$$ are pairwise distinct as ordered pairs of real numbers.
\end{lem}
\begin{proof}
If $a = \cos(\pi/m)$, then $$U_n(a) = \frac{\sin(\frac{n\pi}{m})}{\sin(\frac{\pi}{m})} \text{ and } U_{n+1}(a) = \frac{\sin(\frac{(n+1)\pi}{m})}{\sin(\frac{\pi}{m})}.$$
Suppose there exist $n,n'$ such that $-m \leq n,n' \leq m - 1$ and $U_n(a) = U_{n'}(a)$ and $U_{n+1}(a) = U_{n'+1}(a)$.
Since $$\sin\left(\frac{(n+1)\pi}{m}\right) = \sin\left(\frac{n\pi}{m}\right)\cos\left(\frac{\pi}{m}\right) + \sin\left(\frac{\pi}{m}\right)\cos\left(\frac{n\pi}{m}\right),$$
and since we are assuming $U_{n+1}(a) = U_{n'+1}(a)$, we have
\begin{eqnarray*}
\begin{split}
\sin\left(\frac{n\pi}{m}\right)\cos\left(\frac{\pi}{m}\right) + \sin\left(\frac{\pi}{m}\right)\cos\left(\frac{n\pi}{m}\right) ={}& \sin\left(\frac{n'\pi}{m}\right)\cos\left(\frac{\pi}{m}\right)\\
 &+ \sin\left(\frac{\pi}{m}\right)\cos\left(\frac{n'\pi}{m}\right).
\end{split}
\end{eqnarray*}
Since we are assuming $U_n(a) = U_{n'}(a)$, and hence that $$\sin(n\pi/m) = \sin(n'\pi/m),$$ this last equation reduces to
$$\cos\left(\frac{n\pi}{m}\right) = \cos\left(\frac{n'\pi}{m}\right).$$
As $\frac{n\pi}{m}$, $\frac{n'\pi}{m} \in [-\pi, \pi)$ and the two angles agree on both sine and cosine, we must have $\frac{n\pi}{m} = \frac{n'\pi}{m}$.  Thus $n = n'$ and the ordered pairs $(U_n(a),U_{n+1}(a))$ are pairwise distinct in the given range.
\end{proof}
\section{Sequences associated to dihedral Coxeter systems}
In this section we show that we can generate the roots of a dihedral subsystem using a recurrence very much like the one given by the Chebyshev polynomials.  Though we only gave results for evaluating Chebyshev polynomials at $x = \cos(\pi/m)$ and $x \geq 1$, those will be the only values at which we evaluate the Chebyshev polynomials.  Specifically, we will be plugging in the value $-B(\gamma, \delta)$, which by a theorem of Dyer can only take on certain values.  The next theorem is translated to the situation of dihedral subsystems.
\begin{theorem}[Dyer] \label{t:dyertheorem}
Let $(W,S)_{\{\alpha,\beta\}}$ be a dihedral subsystem of $(W,S)$.  Then, if $\gamma$ and $\delta$ are the canonical simple roots of $(W,S)_{\{\alpha,\beta\}}$, we must have
$$B(\gamma, \delta) \in (-\infty,-1] \cup \{-\cos(\pi/n) \, : \, n \in \mathbb{N}, n \geq 2\}.$$
\end{theorem}
\begin{proof}
See \cite[Theorem 4.4]{dyerreflection}.
\end{proof}
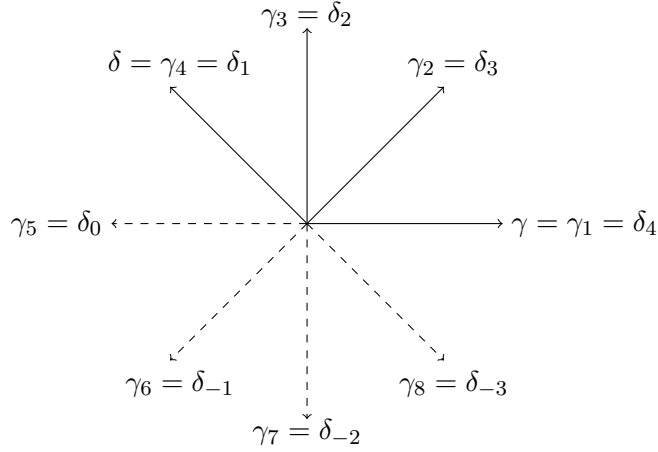
\begin{figure} \label{finitesequencefig}
\begin{center}
\begin{tikzpicture}[scale=2.6]
\draw [->] (0,0) -- (1,0);
\draw (1.42,0) node {$\gamma = \gamma_1 = \delta_4$};
\draw [->] (0,0) -- (0.7,0.7);
\draw (0.75,0.82) node {$\gamma_2 = \delta_3$};
\draw [->] (0,0) -- (0,1);
\draw (0,1.07) node {$\gamma_3 = \delta_2$};
\draw [->] (0,0) -- (-0.7,0.7);
\draw (-0.65,0.82) node {$\delta = \gamma_4 = \delta_1$};
\draw [dashed, ->] (0,0) -- (-1,0);
\draw (-1.28, 0) node {$\gamma_5 = \delta_0$};
\draw [dashed, ->] (0,0) -- (0.7,-0.7);
\draw (0.75,-0.82) node {$\gamma_8 = \delta_{-3}$};
\draw [dashed, ->] (0,0) -- (0,-1);
\draw (0,-1.07) node {$\gamma_7 = \delta_{-2}$};
\draw [dashed, ->] (0,0) -- (-0.7,-0.7);
\draw (-0.65,-0.82) node {$\gamma_6 = \delta_{-1}$};
\end{tikzpicture}
\caption{The sequence of roots in a dihedral subsystem pictured will form what we call a ``local root sequence'' for the dihedral subsytem.  This figure depicts the case when $|s_\gamma s_\delta| = 4$.}
\end{center}
\end{figure}
\begin{defn} \label{d:alphasequence}
Let $\alpha,\beta \in \Phi^{+}$ and let $\Delta_{\{\alpha,\beta\}} = \{\gamma,\delta\}$ be the simple system of $(W,S)_{\{\alpha,\beta\}}$.  Let $\overline{\gamma} = \{\gamma_i\}_{i \in \mathbb{Z}}$ and $\overline{\delta} = \{\delta_i\}_{i \in \mathbb{Z}}$ be the doubly infinite sequences of roots defined by the initial conditions
\begin{equation} \label{e:initial}
\gamma_1 = \gamma,\ \delta_1 = \delta
\end{equation}
and the recurrences
\begin{equation}  \label{e:recurrence}
\gamma_{i+1} = s_{\gamma}(\delta_i),\ \delta_{i+1} = s_{\delta}(\gamma_i).
\end{equation}
The associated backward recurrences are given by
\begin{equation} \label{e:backrecurrence}
\gamma_i = s_\delta (\delta_{i+1}),\ \delta_i = s_\gamma (\gamma_{i+1}).
\end{equation}
We call $\overline{\gamma}$ the \emph{$\gamma$-sequence} of $\Phi_{\{\alpha,\beta\}}$ and $\overline{\delta}$ the \emph{$\delta$-sequence} of $\Phi_{\{\alpha,\beta\}}$.  Any such sequence is called a \emph{local root sequence}.
\end{defn}
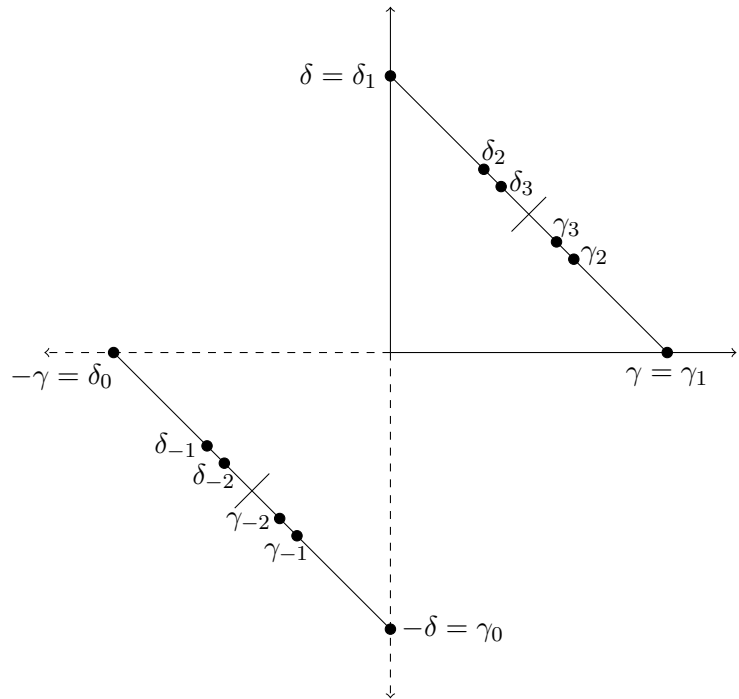
\begin{figure} \label{infiniterootsysfig}
\begin{center}
\begin{tikzpicture}[scale=4.6]
\draw [->] (0,0) -- (1,0);
\draw [->] (0,0) -- (0,1);
\draw [dashed, ->] (0,0) -- (-1,0);
\draw [dashed, ->] (0,0) -- (0,-1);
\filldraw           (0.8,0) circle (0.015)
                    (0,0.8) circle (0.015)
                    (-0.8,0) circle (0.015)
                    (0,-0.8) circle (0.015)
                    (0.53,0.27) circle (0.015)
                    (0.48,0.32) circle (0.015)
                    (0.27,0.53) circle (0.015)
                    (0.32,0.48) circle (0.015)
                    (-0.53,-0.27) circle (0.015)
                    (-0.48,-0.32) circle (0.015)
                    (-0.27,-0.53) circle (0.015)
                    (-0.32,-0.48) circle (0.015);
\draw (0.8,0) -- (0,0.8);
\draw (-0.8,0) -- (0,-0.8);
\draw (0.35,0.35) -- (0.45,0.45);
\draw (-0.35,-0.35) -- (-0.45,-0.45);
\draw (-0.15,0.8) node {$\delta = \delta_1$};
\draw (0.8,-0.07) node {$\gamma = \gamma_1$};
\draw (0.18,-0.8) node {$-\delta = \gamma_0$};
\draw (-0.95,-0.07) node {$-\gamma = \delta_0$};
\draw (0.59, 0.28) node {$\gamma_2$};
\draw (0.51, 0.36) node {$\gamma_3$};
\draw (0.3, 0.58) node {$\delta_2$};
\draw (0.38, 0.49) node {$\delta_3$};
\draw (-0.62, -0.27) node {$\delta_{-1}$};
\draw (-0.51, -0.36) node {$\delta_{-2}$};
\draw (-0.3, -0.58) node {$\gamma_{-1}$};
\draw (-0.41, -0.49) node {$\gamma_{-2}$};
\end{tikzpicture}
\caption{The geometric analogies break down somewhat for infinite dihedral subsystems.  This figure treats the simple roots as perpendicular (they are not perpendicular in the geometric representation).  The non-simple roots are then projected so that their coefficients add to $1$.  The picture does faithfully represent the property of a root lying in the convex cone of two roots in the subsystem.}
\end{center}
\end{figure}
\begin{rem}
The entries of the $\overline{\gamma}$- and $\overline{\delta}$-sequences are calculated by multiplying alternating factors of $s_\gamma$ and $s_\delta$ and applying the result to either $\gamma$ or $\delta$.  In particular, each $\gamma_i$ for $i > 0$ has $s_\gamma$ as a leftmost factor, and similarly $\delta_i$ has $s_\delta$ as a leftmost factor.
\end{rem}
\begin{ex}
Let $(W,S)$ be the Coxeter system of type $A_3$, with generating set $S = \{s,t,u\}$ and relations given by $m_{s,t} = m_{t,u} = 3$ and $m_{s,u} = 2$.  Let $\alpha = \alpha_s + \alpha_t + \alpha_u$ and $\beta = \alpha_s$.   Then $\gamma = \alpha_t + \alpha_u$ and $\delta = \alpha_s$ are canonical simple roots for the local system $(W,S)_{\{\alpha,\beta\}}$.  We have
\begin{equation*}
\overline{\gamma} = (\ldots,-(\alpha_s + \alpha_t + \alpha_u),-\alpha_s,\alpha_t + \alpha_u, \alpha_s+\alpha_t+\alpha_u,\alpha_s,-(\alpha_t + \alpha_u),\ldots),
\end{equation*}
where  the displayed root $\alpha_t+\alpha_u$ is meant to represent $\gamma_1$.  Also,
\begin{equation*}
\overline{\delta} = (\ldots,-\alpha_s,-(\alpha_s+\alpha_t+\alpha_u),-(\alpha_t+\alpha_u),\alpha_s,\alpha_s+\alpha_t+\alpha_u,
\alpha_t+\alpha_u,-\alpha_s,\ldots),
\end{equation*}
where the displayed root $\alpha_s$ is meant to represent $\delta_1$.
\end{ex}
\noindent
The recurrences given in Definition~\ref{d:alphasequence} for the $\overline{\gamma}$- and $\overline{\beta}$-sequences can, in a certain sense, be ``solved'' explicitly.  Our first ``solution'' expresses the roots as an alternating product applied to either $\gamma$ or $\delta$.\\ \\
Recall that if $u,v \in W$, then $(uv)_k$ denotes the alternating product $uvuv\cdots$ of $u$ and $v$ beginning with $u$ and having $k$ factors.  Thus, if $k$ is odd, then $(uv)_k$ ends with the factor $u$; otherwise, $(uv)_k$ ends with the factor $v$.
\begin{lem} \label{l:alphabetasolution}
Let $\alpha,\beta \in \Phi^{+}$ and let $\Delta_{\{\alpha,\beta\}} = \{\gamma,\delta\}$ be the simple system of $(W,S)_{\{\alpha,\beta\}}$.  Let $\overline{\gamma}$ and $\overline{\delta}$ be the associated local root sequences.  Then, for $k \geq 1$, there exist $\theta_k,\theta_k' \in \{\gamma,\delta\}$ satisfying
\begin{eqnarray*}
(s_\gamma s_\delta)_k(\theta_k) = \gamma_{k + 1}\\
(s_\gamma s_\delta)_k(\theta_k') = \delta_{-k + 1}\\
(s_\delta s_\gamma)_k(\theta_k') = \delta_{k + 1}\\
(s_\delta s_\gamma)_k(\theta_k) = \gamma_{-k + 1},
\end{eqnarray*}
where $\theta_k \neq \theta_k'$.
\end{lem}
\begin{proof}
For the base case of $k = 1$, we apply the recurrences (\ref{e:recurrence}) and (\ref{e:backrecurrence}) once to the initial conditions to obtain $\gamma_2 = s_\gamma(\delta)$, $\delta_0 = s_\gamma(\gamma)$, $\delta_2 = s_\delta(\gamma)$, and $\gamma_0 = s_\delta(\delta)$.  Thus the equations are satisfied if $\theta_1 = \delta$ and $\theta_1' = \gamma$.  For the inductive step, we have:
\begin{eqnarray*}
\gamma_{k + 1} = s_\gamma(\delta_k) = s_\gamma(s_\delta s_\gamma)_{k-1}(\theta_{k-1}') = (s_\gamma s_\delta)_k(\theta_{k-1}')\\
\delta_{-k + 1} = s_\gamma(\gamma_{-(k-1) + 1}) = s_\gamma(s_\delta s_\gamma)_{k-1}(\theta_{k-1}) = (s_\gamma s_\delta)_k(\theta_{k-1})\\
\delta_{k + 1} = s_\delta(\gamma_k) = s_\delta(s_\gamma s_\delta)_{k-1}(\theta_{k-1}) = (s_\delta s_\gamma)_k(\theta_{k-1})\\
\gamma_{-k + 1} = s_\delta(\delta_{-(k-1) + 1}) = s_\delta(s_\gamma s_\delta)_{k-1}(\theta_{k-1}') = (s_\delta s_\gamma)_k (\theta_{k-1}'),
\end{eqnarray*}
where the first equation in each line follows from (\ref{e:recurrence}) or (\ref{e:backrecurrence}), and the second equation in each line follows from the inductive hypothesis.  We also have that $\theta_{k-1}, \theta_{k-1}' \in \{\gamma, \delta\}$ and $\theta_{k-1} \neq \theta_{k-1}'$ by the inductive hypothesis.  The result follows from the last equation of each line by using $\theta_k = \theta_{k-1}'$ and $\theta_k' = \theta_{k-1}$.
\end{proof}
\noindent
The next lemma provides much of the motivation for introducing the $\overline{\gamma}$- and $\overline{\delta}$-sequences.
\begin{lem} \label{l:rootsinsequence}
Let $(W,S)_{\{\alpha,\beta\}}$ be a dihedral subsystem with canonical generators $s_\gamma$ and $s_\delta$.  Then every root in $\Phi_{\{\alpha,\beta\}}$ is in either the $\overline{\gamma}$-sequence or the $\overline{\delta}$-sequence.
\end{lem}
\begin{proof}
By Lemma~\ref{l:inversionroots}, every root in $\Phi_{\{\alpha,\beta\}}$ can be obtained as a (possibly empty) alternating product of $s_\gamma$ and $s_\delta$ applied to either $\gamma$ or to $\delta$.  In particular, $\gamma$ and $\delta$ can be obtained using the empty product applied to $\gamma$ or $\delta$.  For an alternating product of fixed length $k \geq 1$, the product either begins with $s_\gamma$ or with $s_\delta$, and is applied to either $\gamma$ or to $\delta$.  These four cases are precisely the four cases given in Lemma~\ref{l:alphabetasolution}.
\end{proof}
\noindent
The next lemma gives another ``solution'' to the recurrences (\ref{e:recurrence}) and (\ref{e:backrecurrence}).  This solution gives an explicit description of the scalars of any root in a dihedral subsystem expressed as a linear combination of the canonical simple roots.
\begin{lem}  \label{l:cheb:alphabeta}
Let $\alpha, \beta \in \Phi$ and $\Delta_{\{\alpha,\beta\}} = \{\gamma, \delta\}$.  Let $a = -B(\gamma,\delta)$. Then we have
\begin{equation}  \label{e:chebcoefficients}
\begin{split}
\gamma_k &= U_k(a) \gamma + U_{k-1}(a) \delta \text{ and}\\
\delta_k &= U_{k-1}(a) \gamma + U_k(a) \delta,
\end{split}
\end{equation}
for all $k \in \mathbb{Z}$.
\end{lem}
\begin{proof}
For $k = 1$, this is just the initial conditions given by (\ref{e:cheb:initial}) since $U_0(a) = 0$ and $U_1(a) = 1$.  For $k = 0$, we have $$\gamma_0 = s_\delta (\delta_1) = -\delta = U_0(a) \gamma + U_{-1}(a) \delta$$
\begin{center}
and
\end{center}
$$\delta_0 = s_\gamma (\gamma_1) = -\gamma = U_{-1}(a) \gamma + U_0(a) \delta.$$  Suppose $k > 1$.  Then, by induction,
\begin{equation*}
\begin{split}
\gamma_k
         &= s_\gamma (\delta_{k-1}) \\
         &= s_\gamma(U_{k-2}(a) \gamma + U_{k-1}(a) \delta) \\
         &= -U_{k-2}(a) \gamma + U_{k-1}(a) (\delta - 2B(\gamma,\delta)\gamma) \\
         &= (2a U_{k-1}(a) - U_{k-2}(a)) \gamma + U_{k-1}(a) \delta \\
         &= U_k(a) \gamma + U_{k-1}(a) \delta.
\end{split}
\end{equation*}
The equation $\delta_k = U_{k-1}(a) \gamma + U_k(a) \delta$ can be obtained by reversing the roles of $\gamma$ and $\delta$ in the above proof.  For $k < 0$, we apply the backward recurrences and use induction, proving the equations hold for $\gamma_k,\delta_k$ assuming they hold for $\gamma_{k+1}$,$\delta_{k+1}$:
\begin{equation*}
\begin{split}
\gamma_k
    &= s_\delta(\delta_{k+1})\\
    &= s_\delta(U_k(a) \gamma + U_{k+1}(a)\delta)\\
    &= U_k(a)(\gamma - 2B(\gamma,\delta)\delta) - U_{k+1}(a) \delta\\
    &= U_k(a)\gamma + (2aU_k(a) - U_{k+1}(a))\delta\\
    &= U_k(a)\gamma + U_{k-1}(a)\delta.
\end{split}
\end{equation*}
The last equation applies the backward recurrence (\ref{e:cheb:backrecurrence}) given for the Chebyshev polynomials using $n=k-1$.  The equation for $\delta_k$ with $k < 0$ can be obtained by reversing the roles of $\gamma$ and $\delta$ in the above proof.
\end{proof}
\begin{lem}  \label{l:sequencepositive}
Let $\alpha, \beta \in \Phi^+$ and $\Delta_{\{\alpha,\beta\}} = \{\gamma, \delta\}$ be the simple system of $(W,S)_{\{\alpha,\beta\}}$.  If $m = |s_\gamma s_\delta|$, then the roots of the form $\gamma_i$, $1 \leq i \leq m$ ($i < m$ if $m = \infty$), are pairwise distinct and positive.  Similarly, the roots of the form $\delta_i$, $1 \leq i \leq m$ ($i < m$ if $m = \infty$), are pairwise distinct and positive.
\end{lem}
\begin{proof}
Suppose $m$ is finite.  Then $B(\gamma,\delta) > -1$, so by Theorem~\ref{t:dyertheorem} we have $$B(\gamma, \delta) = -\cos(\pi/m')$$ for some $m' \geq 2$.  Since $m'$ determines the order of $s_\gamma s_\delta$, $m' = m$.  If $\gamma_k = \gamma_{k'}$ for $1 \leq k,k' \leq m$, then by Lemma~\ref{l:cheb:alphabeta}, we have $$U_k(a) \gamma + U_{k-1}(a) \delta = U_{k'}(a)\gamma + U_{k'-1}(a)\delta,$$ where $a = \cos(\pi/m)$.  By Lemma~\ref{l:distinctpairs}, we must have $k = k'$.  Since the scalars in $\delta_k$ can be obtained from the $\gamma_k$ scalars by interchanging $\gamma$ and $\delta$, the same argument can be applied to the roots $\delta_i$, $1 \leq i \leq m$.\\ \\
For all $i$ satisfying $0 \leq i \leq m$, we have $\sin(\frac{i\pi}{m}) \geq 0$.  Since $a = \cos(\pi/m)$, we have that the linear combination $U_i(a) \gamma + U_{i-1}(a)\delta$ has nonnegative coefficients for $1 \leq i \leq m$ by Definition~\ref{d:shifted}.  Thus $\gamma_i$ is a positive root for all $i$ satisfying $1 \leq i \leq m$.  The same argument applies to $\delta_i$ where $1 \leq i \leq m$.\\ \\
If $m$ is infinite, then $B(\gamma, \delta) \leq -1$ by Theorem~\ref{t:dyertheorem}.  By Lemma~\ref{l:infinitedistinct}, for $a = -B(\gamma, \delta)$, we have $U_k(a) = U_{k'}(a)$ only if $k = k'$.  Also, for $k \geq 0$, $U_k(a) \geq 0$ by Lemma~\ref{l:infinitedistinct}, so that the linear combination $U_i(a) \gamma + U_{i-1}(a) \delta$ has nonnegative coefficients for $i \geq 1$.  Thus, for $i \geq 1$, the $\gamma_i$ are pairwise distinct and positive.  Similarly, the $\delta_i$ are pairwise distinct and positive for all $i \geq 1$.
\end{proof}
\begin{lem}  \label{l:sequencenegative}
Let $\alpha, \beta \in \Phi^+$ and $\Delta_{\{\alpha,\beta\}} = \{\gamma, \delta\}$ be the simple system of $(W,S)_{\{\alpha,\beta\}}$.  Then, for $i \geq 0$, we have $\gamma_{-i} = -\delta_{i+1}$ and $\delta_{-i} = -\gamma_{i+1}$.
\end{lem}
\begin{proof}
First note that $\gamma_0 = s_\delta (\delta_1) = -\delta$ and $\delta_0 = s_\gamma (\gamma_1) = -\gamma$ by the backward recurrences of (\ref{e:backrecurrence}).  Next, for $i > 0$, we have
\begin{equation*}
\begin{split}
\gamma_{-i}
    &= s_\delta (\delta_{-i+1})\\
    &= s_\delta (-\gamma_i)\\
    &= -\delta_{i+1}.
\end{split}
\end{equation*}
The second equality follows from induction while the last equality follows from the definition of the $\delta$-sequence. A similar computation yields the result that $\delta_{-i} = -\gamma_{i+1}$.
\end{proof}
\begin{cor} \label{c:sequencenegative}
Let $\alpha, \beta \in \Phi^+$ and $\Delta_{\{\alpha,\beta\}} = \{\gamma, \delta\}$ be the simple system of $(W,S)_{\{\alpha,\beta\}}$.  If $m = |s_\gamma s_\delta|$, then  the roots of the form $\gamma_{-i}$, $0 \leq i \leq m-1$ (with $i < m$ if $m = \infty$), are negative and pairwise distinct.  Similarly, the roots of the form $\delta_{-i}$, $0 \leq i \leq m-1$ (with $i < m$ if $m = \infty$), are negative and pairwise distinct.
\end{cor}
\begin{proof}
The formulas for $\gamma_{-i}$ and $\delta_{-i}$ given by Lemma~\ref{l:sequencenegative} and Lemma~\ref{l:sequencepositive} imply that the roots of the form $\gamma_{-i}$, for $0 \leq i \leq m-1$ are negative and pairwise distinct, and similarly for the roots of the form $\delta_{-i}$, $ \leq i \leq m-1$.
\end{proof}
\begin{cor}  \label{c:periodic}
Let $\gamma$, $\delta$ be the canonical simple roots for a dihedral subsystem and suppose $m = |s_\gamma s_\delta|$ is finite.  Then, the $\gamma$- and $\delta$-sequences are periodic with period $2m$.
\end{cor}
\begin{proof}
It follows from the defining recurrence for the $\gamma$- and $\delta$-sequence that $\gamma_{i+2} = s_\gamma s_\delta \gamma_i$ and $\delta_{i+2} = s_\delta s_\gamma \delta_i$ for all $i$.  Since $$(s_\gamma s_\delta)^m = (s_\delta s_\gamma)^m = 1,$$ we have $\gamma_{i + 2m} = \gamma_i$ and $\delta_{i + 2m} = \delta_i$.  Lemma~\ref{l:sequencepositive} and Corollary~\ref{c:sequencenegative} imply that the $2m$ roots $\gamma_{-(m-1)},\ldots,\gamma_m$ are pairwise distinct, as are the $2m$ roots $\delta_{-(m-1)},\ldots,\delta_m$, so the result follows.
\end{proof}
\begin{lem} \label{l:sequencerelation}
Let $\gamma$, $\delta$ be the canonical simple roots for a local Coxeter system and let $\gamma \in \Phi^+(\gamma,\delta)$. If $m = |s_\gamma s_\delta|$ is finite, then $\gamma_i = \delta_{m + 1 - i}$ and $\delta_i = \gamma_{m + 1 - i}$.
\end{lem}
\begin{proof}
Let $a = -B(\gamma,\delta) = \cos(\pi/m)$.  We have $$\gamma_i = U_i(a)\gamma + U_{i-1}(a)\delta \text{ and }$$
$$\delta_{m + 1 - i} = U_{m - i}(a) \gamma + U_{m - (i - 1)}(a)\delta,$$
by Lemma~\ref{l:cheb:alphabeta}.  Since $$\sin \left(\frac{(m - i)\pi}{m} \right) = \sin \left(\pi - \frac{i\pi}{m} \right) = \sin \left(\frac{i\pi}{m} \right)$$
for any $i \in \mathbb{Z}$, we have $U_{m-i}(a) = U_i(a)$ and $U_{m - (i - 1)}(a) = U_{i-1}(a)$.  Thus we have $\gamma_i = \delta_{m + 1 - i}$ for all $i \in \mathbb{Z}$.  By the same calculation with $\gamma$ and $\delta$ interchanged, we have $\delta_i = \gamma_{m + 1 - i}$ for all $i \in \mathbb{Z}$.
\end{proof}
\begin{lem}  \label{l:infinitedisjoint}
Let $\Phi^+_{\{\alpha,\beta\}}$ be an infinite dihedral subsystem with canonical simple roots $\gamma$ and $\delta$.  Then, no entry in the $\overline{\gamma}$-sequence is an entry of the $\overline{\delta}$-sequence.
\end{lem}
\begin{proof}
Let $a = -B(\gamma,\delta)$.  Since $\Phi^+_{\{\alpha,\beta\}}$ is infinite, we have $a \geq 1$ by Theorem~\ref{t:dyertheorem}.  By Lemma~\ref{l:chebincreasing}, $\{U_n(a)\}_{n \geq 1}$ is a strictly increasing sequence of real numbers.  By Lemma~\ref{l:cheb:alphabeta}, for any $i,j \geq 1$, we have $$\gamma_i = U_i(a) \gamma + U_{i-1}(a)\delta$$
\begin{center}
and
\end{center}
$$\delta_j = U_{j-1}(a)\gamma + U_j(a)\delta.$$  Thus, the $\gamma$ coefficient is strictly larger than the $\delta$ coefficient for $\gamma_i$, whereas the $\gamma$ coefficient is strictly smaller than the $\delta$ coefficient for $\delta_j$.  It follows that $\gamma_i \neq \delta_j$.  By Lemma~\ref{l:sequencenegative}, this implies that $\gamma_i \neq \delta_j$ for any $i,j \leq 0$ as well.  In the case that $i \leq 0$ but $j \geq 1$, $\gamma_i$ is negative by Corollary~\ref{c:sequencenegative}, whereas $\delta_j$ is positive by Lemma~\ref{l:sequencepositive}, so $\gamma_i \neq \delta_j$.  The same reasoning applies to the case $i \geq 1$ and $j \leq 0$.
\end{proof}
\begin{prop} \label{p:localdisjoint}
Let $\gamma$, $\delta$ be the canonical simple roots for a dihedral subsystem $\Phi_{\{\gamma,\delta\}}$.  Let $\overline{\gamma}$ and $\overline{\delta}$ be the associated local root sequences.  Let $m = |s_\gamma s_\delta|$. If $m$ is finite, then $$\Phi^+_{\{\alpha,\beta\}} = \{\gamma_1,\ldots,\gamma_m\} = \{\delta_1,\ldots,\delta_m\}.$$  If $m$ is infinite, then $$\Phi^+_{\{\alpha,\beta\}} = \{\gamma_1,\gamma_2,\ldots\} \cup \{\delta_1,\delta_2,\ldots\},$$
and the union is disjoint.
\end{prop}
\begin{proof}
By Lemma~\ref{l:rootsinsequence}, every root of $\Phi^+_{\{\alpha,\beta\}}$ is an entry of the $\overline{\gamma}$-sequence or of the $\overline{\delta}$-sequence.\\ \\
Suppose that $m$ is finite.  By Lemma~\ref{l:sequencerelation}, every root of the $\overline{\delta}$-sequence is an entry of the $\overline{\gamma}$-sequence, so that every root of $\Phi^+_{\{\alpha,\beta\}}$ is an entry of the $\overline{\gamma}$-sequence.  By Corollary~\ref{c:periodic}, the $\overline{\gamma}$-sequence is periodic with period $2m$, so every root of $\Phi_{\{\alpha,\beta\}}$ is in the set $\{\gamma_{-m-1},\ldots,\gamma_0,\gamma_1,\ldots,\gamma_m\}$.  Thus, Lemma~\ref{l:sequencepositive} and Corollary~\ref{c:sequencenegative} imply that $\Phi^+_{\{\alpha,\beta\}} = \{\gamma_1,\ldots,\gamma_m\}$.\\ \\
Suppose that $m$ is infinite.  By Lemma~\ref{l:sequencepositive} and Corollary~\ref{c:sequencenegative}, the roots $\gamma_i$ and $\delta_i$ are positive if $i \geq 1$ and negative if $i \leq 0$.  It follows that every root of $\Phi^+_{\{\alpha,\beta\}}$ is in the set $\{\gamma_1,\gamma_2,\ldots\} \cup \{\delta_1,\delta_2,\ldots\}$.  The union is disjoint by Lemma~\ref{l:infinitedisjoint}.
\end{proof}
\noindent
If we think of the indices of the $\overline{\gamma}$- or $\overline{\delta}$-sequence as providing a total order to the roots in the sequences, then the next lemma says that given any three roots in a local root sequence that are close enough together in the sequence, the ``middle root'' lies in the convex cone spanned by the ``outer roots''.
\begin{lem} \label{l:biconvexity}
Let $d_1,d_2 \geq 1$ and $n \in \mathbb{Z}$.  Let $\gamma, \delta$ be the canonical simple roots for a local Coxeter system and $m = |s_\gamma s_\delta|$ (where $m$ is possibly infinite).  Suppose $d_1 + d_2 < m$.  Then $\gamma_{n + d_1}$ lies in the convex cone spanned by $\gamma_n$ and $\gamma_{n + d_1 + d_2}$.  Similarly, $\delta_{n + d_1}$ lies in the convex cone spanned by $\delta_n$ and $\delta_{n + d_1 + d_2}$.  In particular, if $i,j,k \in \mathbb{N}^+$ and $1 \leq i < j < k \leq m$, then $\gamma_j$ lies in the convex cone spanned by $\gamma_i$ and $\gamma_k$.  Similarly, $\delta_j$ lies in the convex cone spanned by $\delta_i$ and $\delta_k$.
\end{lem}
\begin{proof}
By Lemma~\ref{l:cheb:alphabeta}, we have $\gamma_{n + d_1} = U_{n + d_1}(a) \gamma + U_{n + d_1 - 1}(a) \delta$, where $a = -B(\gamma,\delta)$.  Let $x = U_{d_1}(a)$ and $y = U_{d_2}(a)$. Then, by Lemma~\ref{l:cheb:biconvexity}, we have
\begin{equation*}
\begin{split}
U_{d_1 + d_2}(a) \gamma_{n + d_1} ={}& U_{d_1 + d_2}(a) U_{n + d_1}(a)\gamma + U_{d_1 + d_2}(a) U_{n + d_1 - 1}(a) \delta \\
={}& [x U_{n+d_1+d_2}(a) + y U_n(a)] \gamma\\
&+ [x U_{n+d_1+d_2-1}(a) + y U_{n-1}(a)] \delta\\
={}& x (U_{n+d_1+d_2}(a) \gamma + U_{n+d_1+d_2-1}(a) \delta)\\
&+ y(U_n(a) \gamma + U_{n-1}(a) \delta) \\
={}& x \gamma_{n+d_1+d_2} + y \gamma_n.
\end{split}
\end{equation*}
Since $1 \leq d_1,d_2 < m$, we have $x,y > 0$.  Since $2 \leq d_1 + d_2 < m$, if $m$ is finite then $\sin\left(\frac{(d_1+d_2)\pi}{m}\right) > 0$ implies $U_{d_1+d_2}(a) > 0$.  If $m$ is infinite then Lemma~\ref{l:infinitedistinct} implies $U_{d_1+d_2}(a) > 0$.  The result follows by solving for $\gamma_{n + d_1}$.
\end{proof}
\noindent
Recall by Definition~\ref{d:alphasequence} that the sequences $\overline{\gamma}$ and $\overline{\delta}$ associated to a local dihedral Coxeter system with $\gamma$ and $\delta$ as canonical roots are called local root sequences.
\begin{cor}  \label{c:biconvexityconverse}
Let $\gamma, \delta$ be the canonical simple roots for a dihedral Coxeter system and let $\overline{\gamma}$ be an associated local root sequence.  Let $m = |s_\gamma s_\delta|$ (where $m$ is possibly infinite).  Let $i,j,k \in \mathbb{N}^+$, where $1 \leq i < j \leq m$ and $1 \leq k \leq m$.  If $\gamma_k$ lies in the convex cone spanned by $\gamma_i$ and $\gamma_j$, then $i < k < j$.
\end{cor}
\begin{proof}
Suppose that $k < i < j$.  Then by Lemma~\ref{l:biconvexity}, $\gamma_i$ lies in the convex cone spanned by $\gamma_k$ and $\gamma_j$, which (together with the hypothesis that $\gamma_k$ lies in the convex cone spanned by $\gamma_i$ and $\gamma_j$) contradicts Lemma~\ref{l:convexcone}.  Similarly, if $i < j < k$, then $\gamma_j$ lies in the convex cone spanned by $\gamma_i$ and $\gamma_k$, contradicting Lemma~\ref{l:convexcone}.
\end{proof}
\noindent
The previous two results apply to infinite dihedral subsystems whenever the roots involved lie within a single local root sequence.  However, it is possible for two roots in an infinite dihedral subsystem to lie in distinct local root sequences.  The next two lemmas show that in this case, results similar to the previous two lemmas hold.
\begin{lem} \label{l:infinitebiconvexity}
Let $\Phi^+_{\{\alpha,\beta\}}$ be an infinite dihedral subsystem with canonical simple roots $\gamma$ and $\delta$. Let $\overline{\gamma}$ and $\overline{\delta}$ be the associated local root sequences.
Suppose $1 \leq i < j$ and $k \geq 1$.  Then:
\begin{enumerate}[(1)]
\item  the root $\gamma_j$ lies in the convex cone spanned by $\gamma_i$ and $\delta_k$;
\item  the root $\delta_j$ lies in the convex cone spanned by $\gamma_k$ and $\delta_i$.
\end{enumerate}
\end{lem}
\begin{proof}
Let $a = -B(\gamma,\delta)$.  Since $\Phi^+_{\{\alpha,\beta\}}$ is infinite, we have $a \geq 1$ by Theorem~\ref{t:dyertheorem}.  Let $r_1 = \frac{U_{k-1}(a)}{U_k(a)}$, $r_2 = \frac{U_j(a)}{U_{j-1}(a)}$, and $r_3 = \frac{U_i(a)}{U_{i-1}(a)}$.  By Lemma~\ref{l:chebincreasing}, we have $r_1 < 1$ and $r_2, r_3 > 1$.  By Lemma~\ref{l:chebratio}, we have $r_3 > r_2$.  Since each $r_i$ gives the ratio of the $\gamma$ coefficient to the $\delta$ coefficient, and $r_1 < r_2 < r_3$, the first assertion follows from Lemma~\ref{l:conelemma}.  The second assertion is proven by interchanging $\gamma$ and $\delta$ in the above argument.
\end{proof}
\begin{cor}  \label{c:infinitebiconvexityconverse}
Let $\gamma, \delta$ be the canonical simple roots for a dihedral Coxeter system.  Suppose $|s_\gamma s_\delta|$ is infinite and let $\overline{\gamma}$ and $\overline{\delta}$ be the associated local root sequences.  Suppose $i,j,k \geq 1$, $i \neq j$, and that $\gamma_j$ lies in the convex cone spanned by $\gamma_i$ and $\delta_k$.  Then $i < j$.  Similarly, if $j \neq k$ and $\delta_j$ lies in the convex cone spanned by $\gamma_i$ and $\delta_k$, then $k < j$.
\end{cor}
\begin{proof}
Suppose $i > j$ and $\gamma_j$ lies in the convex cone spanned by $\gamma_i$ and $\delta_k$.  Then by Lemma~\ref{l:infinitebiconvexity}, $\gamma_i$ lies in the convex cone spanned by $\gamma_j$ and $\delta_k$, contradicting Lemma~\ref{l:convexcone}.  Similarly, if $\delta_j$ lies in the convex cone spanned by $\gamma_i$ and $\delta_k$, we get a contradiction if we assume $k > j$.
\end{proof}
\section{Inversion sets}
The sets that contain all the positive roots of a dihedral subsystem play a special role in the constructions of Chapter $4$, so we give these sets a name.  In \cite{fbI}, Green and Losonczy referred to any set of positive roots of the form $\{\alpha, \alpha + \beta, \beta\}$ as an \emph{inversion triple}.  We view our definition of inversion set as a generalization of this notion.
\begin{defn} \label{d:inversionmset}
We call any subset of $\Phi^+$ of the form $\Phi^+_{\{\alpha,\beta\}}$, where $\alpha, \beta \in \Phi$ are (possibly negative) roots that are not scalar multiples of one another, an \emph{inversion set}.  The set of all inversion sets is denoted by $\text{Inv}(\Phi^+)$.  If $m = \left|\Phi^+_{\{\alpha,\beta\}}\right|$, then we also say that $\Phi^+_{\{\alpha,\beta\}}$ is an \emph{inversion $m$-set}.
\end{defn}
\begin{lem}  \label{l:inversion2d}
Let $\Psi$ be an inversion set.  Then $\text{span}(\Psi)$ is a two-dimensional subspace of $V$.  Furthermore, given two distinct roots $\alpha, \beta \in \Psi$, we have $\text{span}(\{\alpha,\beta\}) = \text{span}(\Psi)$.
\end{lem}
\begin{proof}
Since $\Psi$ is an inversion set, $\Psi = \Phi^+_{\{\alpha,\beta\}}$ for some $\alpha, \beta \in \Phi$.  Let $\gamma, \delta$ be the canonical simple roots of $(W,S)_{\{\alpha,\beta\}}$.  By Lemma~\ref{l:inversionroots}, every root in $\Phi^+_{\{\alpha,\beta\}}$ can be obtained as an alternating product of $s_\gamma$ and $s_\delta$ applied to $\gamma$ or $\delta$.  By the reflection formula (\ref{e:reflection}), $s_\gamma$ and $s_\delta$ preserve (set-wise) the set $\text{span}(\{\gamma,\delta\})$.  Thus every root in $\Psi$ is in $\text{span}(\{\gamma,\delta\})$ so that $\text{span}(\Psi)$ is a two-dimensional subspace of $V$, which proves the first assertion.\\ \\
If $\alpha,\beta \in \Psi$ are distinct, then since roots in $\Psi$ are positive, we have that $\alpha$ and $\beta$ are not scalar multiples of one another.  Thus $\text{span}(\{\alpha,\beta\})$ is a two-dimensional subspace of $\text{span}(\Psi)$, which is two-dimensional, so $$\text{span}(\{\alpha,\beta\}) = \text{span}(\Psi).$$
\end{proof}
\begin{lem} \label{l:oneroot}
Let $\Psi$ and $\Upsilon$ be inversion sets with $\Psi \neq \Upsilon$.  Then $\Psi$ and $\Upsilon$ intersect in at most one root.
\end{lem}
\begin{proof}
Suppose $\Psi$ and $\Upsilon$ intersect in distinct positive roots $\alpha$ and $\beta$.  We may then form the dihedral subsystem $(W,S)_{\{\alpha,\beta\}}$.  We have $\Psi = \Phi^+_{\{\alpha', \beta'\}}$ for some $\alpha', \beta' \in \Phi$, by Definition~\ref{d:inversionmset} and hypothesis.  Also, we have $\alpha, \beta \in \Psi$, by assumption.  It follows that $\text{span}(\{\alpha,\beta\}) = \text{span}(\{\alpha',\beta'\})$, since $\text{span}(\{\alpha,\beta\})$ is a $2$-dimensional subspace of $\text{span}(\{\alpha',\beta'\})$, which itself is a two-dimensional space by Lemma~\ref{l:inversion2d}.  Similarly, $\Upsilon = \Phi^+_{\{\alpha'',\beta''\}}$ for some $\alpha'', \beta'' \in \Phi$ and $\text{span}(\{\alpha,\beta\}) = \text{span}(\{\alpha'',\beta''\})$.  Intersecting the span equations with $\Phi^+$ gives $\Psi = \text{span}(\{\alpha,\beta\}) \cap \Phi^+ = \Upsilon$ by Lemma~\ref{l:localsystemform}.
\end{proof}
\begin{cor} \label{c:inversionline}
Let $\alpha, \beta \in \Phi^+$ be distinct positive roots.  There exists a unique inversion set $\Psi$ containing $\alpha$ and $\beta$.
\end{cor}
\begin{proof}
By Definition~\ref{d:inversionmset} and Lemma~\ref{l:localsystemform}, $\Psi = \Phi^+_{\{\alpha,\beta\}}$ is an inversion set containing $\alpha$ and $\beta$.  The uniqueness of $\Psi$ follows from Lemma~\ref{l:oneroot}.
\end{proof}
\noindent
The indices of the $\overline{\gamma}$- and $\overline{\delta}$-sequences make a natural candidate for an order on the roots of an inversion set.  However, in the case that the inversion set is infinite, the situation is not as straightforward.  Though the next definition is made from the point of view of the $\overline{\gamma}$-sequence, either canonical simple root can be ``the $\gamma$ root''.
\begin{defn} \label{d:totalinversion}
Let $\Psi = \Phi^+_{\{\alpha,\beta\}}$ be an inversion $m$-set with canonical simple roots $\gamma$ and $\delta$.  Then we define a total ordering $\leq_{\Psi,\gamma}$ on $\Psi$ as follows:
\begin{enumerate}[(1)]
\item If $m$ is finite, and $\alpha = \gamma_i$, and $\beta = \gamma_j$, where $1 \leq i,j \leq m$, then we say that $\alpha \leq_{\Psi,\gamma} \beta$ if $i \leq j$.
\item If $m$ is infinite then we say $\alpha \leq_{\Psi,\gamma} \beta$ if one of the following holds:
\begin{enumerate}[(a)]
\item  $\alpha = \gamma_i$ and $\beta = \gamma_j$, where $i,j \geq 1$ and $i \leq j$;
\item  $\alpha = \delta_i$ and $\beta = \delta_j$, where $i,j \geq 1$ and $i \geq j$;
\item  $\alpha = \gamma_i$ and $\beta = \delta_j$, where $i,j \geq 1$.
\end{enumerate}
\end{enumerate}
\end{defn}
\begin{rem} \label{r:graphicline}
We can graphically depict the total ordering of $\Psi$ with respect to $\gamma$ on a line by placing the roots of $\Psi$ on a line so that $\alpha$ is to the left of $\beta$ if and only if $\alpha \leq_{\Psi,\gamma} \beta$.  If $\Psi$ is finite, then the typical total ordering of $\Psi$ with respect to $\gamma$ is depicted below:

\vspace{0.1in}

\begin{tikzpicture}[scale=2.6] \label{finitesystemfig}
\draw (-1,0) -- (0.2,0);
\draw (0.6,0) -- (1,0);
\draw (0.42,0) node {$\cdots$};
\filldraw           (-1,0) circle (0.02)
                    (-0.5,0) circle (0.02)
                    (1,0) circle (0.02);
\draw (-1, -0.1) node {$\gamma = \gamma_1$};
\draw (-0.5, -0.1) node {$\gamma_2$};
\draw (1, -0.1) node {$\gamma_m = \delta$};
\end{tikzpicture}

\vspace{0.1in}

\noindent
If $\Psi$ is infinite, then the order begins with all the roots from $\overline{\gamma}$ and ends with the roots in $\overline{\delta}$.  Thus, there are two endpoints, $\gamma$ and $\delta$, having infinitely many roots between them.  We depict this situation by placing a bar between the two infinite sequences.  Thus, the three cases given in Definition~\ref{d:totalinversion} for infinite inversion sets are translated pictorially as:  the two roots are to the left of the bar (both roots are in $\overline{\gamma}$), one root is to the left of the bar and one root is to the right of the bar (one root is in $\overline{\gamma}$, one root is in $\overline{\delta}$), and the two roots are to the right of the bar (both roots are in $\overline{\delta}$).

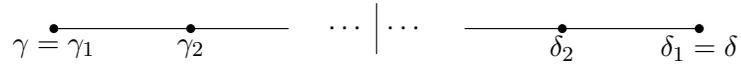
\vspace{0.1in}

\begin{tikzpicture}[scale=2.6] \label{infinitesystemfig}
\draw (-1,0) -- (0.2,0);
\draw (1.1,0) -- (2.3,0);
\draw (0.5,0) node {$\cdots$};
\draw (0.8,0) node {$\cdots$};
\filldraw           (-1,0) circle (0.02)
                    (-0.3,0) circle (0.02)
                    (2.3,0) circle (0.02)
                    (1.6,0) circle (0.02);
\draw (-1, -0.1) node {$\gamma = \gamma_1$};
\draw (-0.3, -0.1) node {$\gamma_2$};
\draw (2.3, -0.1) node {$\delta_1 = \delta$};
\draw (1.6,-0.1) node {$\delta_2$};
\draw (0.65,0.13) -- (0.65,-0.13);
\end{tikzpicture}
\end{rem}
\noindent
We use interval notation in the standard way.  That is, given any totally ordered set $(X,\leq)$, we have $[a,b] = \{x \, : \, a \leq x \leq b\}$.
\begin{prop}  \label{p:biconvexstructure}
Let $\Lambda \subseteq \Phi^+$ be any set of positive roots.  Then $\Lambda$ is a biconvex set of positive roots if and only if for every inversion set $\Psi$ with local simple system $\Delta$, we have one of the following:
\begin{enumerate}[(1)]
\item $\Lambda \cap \Psi = \emptyset$, so that $\Lambda \cap \Delta = \emptyset$.
\item There exists $\gamma \in \Lambda \cap \Delta$ such that $\Lambda \cap \Psi = \{\gamma_1,\gamma_2,\ldots\}$, the set of positive roots in the sequence $\overline{\gamma}$.
\item There exists $\gamma \in \Lambda \cap \Delta$ such that for some $\lambda \in \Psi$, $\Lambda \cap \Psi = [\gamma, \lambda]$ in the total order $(\Psi,\leq_{\Psi,\gamma})$.
\end{enumerate}
\end{prop}
\begin{proof}
We first show that if statement $(1)$, $(2)$, or $(3)$ holds for any inversion set $\Psi$, then $\Lambda$ is a convex set of positive roots.  Suppose that $\alpha, \beta \in \Lambda$ are distinct and $a,b \geq 0$.  We must show that $a\alpha + b\beta \in \Phi^+$ implies $a\alpha + b\beta \in \Lambda$.  Thus we let $\Psi = \Phi^+_{\{\alpha,\beta\}}$ and suppose that $a\alpha + b\beta \in \Phi^+$.  Since $\Lambda \cap \Psi$ is nonempty, $(2)$ or $(3)$ holds.  We consider the following two cases:
\begin{enumerate}[(A)]
\item The roots $\alpha$ and $\beta$ are both in the $\overline{\gamma}$-sequence.
\item The set $\Psi$ is infinite and (without loss of generality) $\alpha$ is contained in the $\overline{\gamma}$-sequence while $\beta$ is in the $\overline{\delta}$-sequence.
\end{enumerate}
\noindent
Suppose $\alpha$ and $\beta$ are in the $\overline{\gamma}$-sequence and that $\alpha = \gamma_i$ and $\beta = \gamma_j$ for some $i,j \geq 1$.  We assume without loss of generality that $i < j$.  If $a\alpha + b\beta$ is in the $\overline{\delta}$-sequence but not the $\overline{\gamma}$-sequence, then Corollary~\ref{c:infinitebiconvexityconverse} implies that $\beta$ is in the convex cone spanned by $\alpha$ and $a\alpha + b\beta$, contradicting Lemma~\ref{l:convexcone}.  Thus $a\alpha + b\beta = \gamma_k$, where $i < k < j$ by Corollary~\ref{c:biconvexityconverse}.  If $(2)$ holds, then $a\alpha + b\beta \in \Lambda$ since $\Lambda$ contains all the roots of the $\overline{\gamma}$-sequence.  If $(3)$ holds, then $\alpha,\beta \in [\gamma,\lambda]$ and we have $\alpha \leq_{\Psi,\gamma} a\alpha + b\beta \leq_{\Psi,\gamma} \beta$, so that $a\alpha + b\beta \in [\gamma,\lambda]$ and hence $a\alpha + b\beta \in \Lambda$.\\ \\
Suppose $\Psi$ is infinite, $\alpha = \gamma_i$, and $\beta = \delta_j$.  If $a\alpha + b\beta = \gamma_k$ for some $k \geq 1$, then Corollary~\ref{c:infinitebiconvexityconverse} implies that $k > i$.  By Definition~\ref{d:totalinversion}, we have $\alpha \leq_{\Psi,\gamma} a\alpha + b\beta \leq_{\Psi,\gamma} \beta$, and hence $a\alpha + b\beta \in \Lambda$.  Similarly, if $a\alpha + b\beta = \delta_k$ for some $k \geq 1$, then Corollary~\ref{c:infinitebiconvexityconverse} implies that $k > j$.  Then we have $a\alpha + b\beta \in [\gamma,\lambda]$ since $\alpha, \beta \in [\gamma,\lambda] \subseteq \Lambda$.\\ \\
To show that $\Lambda$ is biconvex, we note that if $\Lambda$ satisfies $(1)$, $(2)$, or $(3)$ for every inversion set $\Psi$, then so does $\Phi^+ \setminus \Lambda$.  (This follows from Definition~\ref{d:totalinversion} or Remark~\ref{r:graphicline}.)  Thus, we have that $\Phi^+ \setminus \Lambda$ is convex, and hence $\Lambda$ is biconvex by Definition~\ref{d:biconvexity}.\\ \\
Turning to the converse, we suppose that $\Lambda$ is a biconvex set of positive roots.  Let $\Psi$ be an inversion set with local simple system $\Delta = \{\gamma,\delta\}$ and suppose $\Lambda \cap \Psi \neq \emptyset$.  If $\gamma,\delta \not\in \Lambda \cap \Psi$, then there exists $\lambda \in \Lambda \cap \Psi$ such that $\lambda = c\gamma + d\delta$.  Since $\gamma,\delta \in \Psi$, we have $\gamma, \delta \not\in \Lambda$, contradicting the biconvexity of $\Lambda$.  Thus we may assume there exists $\gamma \in \Lambda \cap \Delta$.\\ \\
Suppose there exists a $\lambda'$ in the $\overline{\gamma}$-sequence such that $\lambda' \not\in \Lambda \cap \Psi$.  Let $i$ be the smallest index such that $\gamma_i \not\in \Lambda \cap \Psi$ and set $\lambda = \gamma_{i-1}$.  Since $\gamma_i \in \Psi$, $\gamma_i \not\in \Lambda$.  If there exists $j \in \mathbb{N}^+$ such that $i < j \leq |s_\gamma s_\delta|$ and $\gamma_j \in \Lambda \cap \Psi$, we have that $\gamma_i$ lies in the convex cone spanned by $\gamma_{i-1}$ and $\gamma_j$, by Lemma~\ref{l:biconvexity}.  This contradicts the biconvexity of $\Lambda$.  Similarly, if $\Psi$ is infinite and there is a root $\mu$ in the $\overline{\delta}$-sequence, then by Lemma~\ref{l:infinitebiconvexity}, $\gamma_i$ lies in the convex cone spanned by $\gamma_{i-1}$ and $\mu$.  Thus, in this case, it follows that $\Lambda \cap \Psi = [\gamma,\lambda] = [\gamma,\gamma_{i-1}]$.\\ \\
Suppose there is no positive root $\lambda'$ in the $\overline{\gamma}$-sequence such that $\lambda' \not\in \Lambda \cap \Psi$.  Then if $\Psi$ is finite, we have that $\Lambda \cap \Psi = \Psi = [\gamma,\delta]$ by Proposition~\ref{p:localdisjoint}.  If instead $\Psi$ is infinite, then every root of the $\overline{\gamma}$-sequence is in $\Lambda \cap \Psi$.  If there are no roots of the $\overline{\delta}$-sequence in $\Lambda \cap \Psi$, then Condition $(2)$ is satisfied.\\ \\
Suppose every positive root of $\overline{\gamma}$ lies in $\Lambda \cap \Psi$ and that there are roots of $\Lambda \cap \Psi$ in the $\overline{\delta}$-sequence.  If $\Lambda \cap \Psi = \Psi$, then $\Lambda \cap \Psi = [\gamma,\delta]$ so that statement $(3)$ is satisfied.  Thus we suppose there exists a $\lambda'$ in the $\overline{\delta}$-sequence such that $\lambda' \not\in \Lambda \cap \Psi$.  Let $i$ be the smallest index such that $\delta_i \in \Lambda \cap \Psi$.  If $i > 1$, set $\lambda = \delta_i$ and note that $\delta_{i-1} \not\in \Lambda \cap \Psi$.  If there exists $j > i$ such that $\delta_j \not\in \Lambda \cap \Psi$, then Lemma~\ref{l:biconvexity} implies that $\delta_i$ lies in the convex cone of $\delta_{i-1}$ and $\delta_j$.  Since $\delta_{i-1},\delta_i,\delta_j \in \Psi$, this would imply that $\delta_{i-1} \not\in \Lambda$, $\delta_i \in \Lambda$, and $\delta_j \not\in \Lambda$, which contradicts the biconvexity of $\Lambda$.  Thus, if $i > 1$, Condition $(3)$ is satisfied with $\Lambda \cap \Psi = [\gamma,\lambda] = [\gamma,\delta_i]$.  Lastly, if $i = 1$, then $\gamma,\delta \in \Lambda \cap \Psi$.  By Lemma~\ref{l:infinitebiconvexity}, $\lambda'$ lies in the convex cone spanned by $\delta$ and $\gamma$, which contradicts the biconvexity of $\Lambda$ because $\lambda' \not\in \Lambda$.
\end{proof}
\begin{cor}  \label{c:biconvexstructure}
Let $w \in W$ and let $\Psi$ be an inversion set with canonical simple roots $\gamma$ and $\delta$.  Let $\overline{\gamma}$ and $\overline{\delta}$ be the associated local root sequences.  If $m = |s_\gamma s_\delta|$, then at least one of the three following statements holds:
\begin{enumerate}[(1)]
\item For some $k$ satisfying $1 \leq k \leq m$ ($k < m$ if $m = \infty$), we have $$\Phi(w) \cap \Psi = [\gamma_1, \gamma_k]$$ in the total order $(\Psi,\leq_{\Psi,\gamma})$.
\item For some $k$ satisfying $1 \leq k \leq m$ ($k < m$ if $m = \infty$), we have $$\Phi(w) \cap \Psi = [\delta_1, \delta_k]$$ in the total order $(\Psi,\leq_{\Psi,\delta})$.
\item We have $\Phi(w) \cap \Psi = \emptyset$.
\end{enumerate}
\end{cor}
\begin{proof}
By Lemma~\ref{l:inversionbiconvex}, $\Phi(w)$ is a biconvex set of roots so we may apply Proposition~\ref{p:biconvexstructure}.  Since $\Phi(w)$ is finite we can eliminate Condition $(2)$ of Proposition~\ref{p:biconvexstructure}.  Thus, if $\Phi(w) \cap \Psi$ is nonempty, we have
$$\Phi(w) \cap \Psi = \left[\gamma, \lambda\right]$$ for some $\lambda \in \Phi(w) \cap \Psi$ in the total ordering $(\Psi,\leq_{\Psi,\gamma})$ or $$\Phi(w) \cap \Psi = \left[\delta, \lambda\right]$$ for some $\lambda \in \Phi(w) \cap \Psi$ in the total ordering $(\Psi,\leq_{\Psi,\delta})$. Since $\Phi(w)$ is finite, in the first case we must have $\lambda = \gamma_k$ for some $k \leq m$.  Similarly, in the second case we must have $\lambda = \delta_k$ for some $k \leq m$.
\end{proof}
\section{Consequences for standard encodings of reduced expressions}
The following definition has many variants.  In \cite[Definition 2.5]{reducedweyl}, we find the same definition applied to the situation where $a, b = 1$ in the definition given below.  In \cite[Section 5.2]{bb}, the conditions are meant to apply to a total ordering on $\Phi^{+}$.  Our variant is introduced to characterize the labelings of $\Phi^+$ that encode reduced expressions for some $w \in W$.
\begin{defn} \label{d:standard}
Let $T:\Phi^+ \rightarrow \Bbb{N}_0$ be a labeling of $\Phi^+$.  We say that $T$ is a \emph{standard labeling of $\Phi^+$} if for any pair of distinct positive roots $\alpha$ and $\beta$ the following implication holds:  \\ \\
If $\gamma$ lies in the convex cone spanned by $\alpha$ and $\beta$ then we have either
$$T(\alpha) \leq T(\gamma) \leq T(\beta)$$
\begin{center}
or
\end{center}
$$T(\beta) \geq T(\gamma) \geq T(\alpha).$$
\end{defn}
\begin{rem}
Recall by Definition~\ref{d:labeling} that $T:\Phi^+ \rightarrow \mathbb{N}$ is sequential if $T(\text{supp}(T)) = \{1,\ldots,|\text{supp}(T)|\}$.  If $T$ is both sequential and standard, the inequalities in Definition~\ref{d:standard} are strict whenever $\alpha, \beta, \gamma \in \text{supp}(T)$, i.e. whenever $T(\alpha),T(\beta),T(\gamma) \neq 0$.
\end{rem}
\begin{lem} \label{l:standard}
Let $(W,S)$ be a Coxeter system.  Let $w \in W$.  If $\textbf{x}$ is any reduced expression for $w$ and $T_{
\textbf{x}}:\Phi^+ \rightarrow \Bbb{N}$ is the standard encoding of $\textbf{x}$, then $T_{\textbf{x}}$ is a standard labeling.
\end{lem}
\begin{proof}
Let $\alpha, \beta, a\alpha + b\beta \in \Phi^+$, where $a,b > 0$ and $\alpha \neq \beta$.  Since $\Phi(w)$ is biconvex and $\alpha$ and $\beta$ can be interchanged in all proofs of standardness, there are only four cases to consider:
\begin{enumerate}[(1)]
\item $\alpha \in \Phi(w)$, $a\alpha + b\beta \in \Phi(w)$, $\beta \in \Phi(w)$;
\item $\alpha \in \Phi(w)$, $a\alpha + b\beta \in \Phi(w)$, $\beta \not\in \Phi(w)$;
\item $\alpha \in \Phi(w)$, $a\alpha + b\beta \not\in \Phi(w)$, $\beta \not\in \Phi(w)$;
\item $\alpha \not\in \Phi(w)$, $a\alpha + b\beta \not \in \Phi(w)$, $\beta \not\in \Phi(w)$.
\end{enumerate}
\noindent
Let $\textbf{x} = (s_1,\ldots,s_n)$.  By Definition~\ref{d:standardencoding}, $T_{\textbf{x}}(\lambda) = 0$ for $\lambda \not\in \Phi(w)$, so cases $(3)$ and $(4)$ satisfy $T_{\textbf{x}}(\beta) \leq T_{\textbf{x}}(a\alpha + b\beta) \leq T_{\textbf{x}}(\alpha)$ since the first two labels in the inequality are zero in both cases.\\ \\
Thus, we first suppose $\alpha, \beta, a\alpha + b\beta \in \Phi(w)$, where $a, b > 0$.  Assume without loss of generality that $T_{\textbf{x}}(\alpha) < T_{\textbf{x}}(\beta)$.  Towards a contradiction, we must consider the subcase where $T_{\textbf{x}}(a\alpha +  b\beta)$ is larger than both $T_{\textbf{x}}(\alpha)$ and $T_{\textbf{x}}(\beta)$ and the subcase where $T_{\textbf{x}}(a\alpha + b\beta)$ is smaller than both $T_{\textbf{x}}(\alpha)$ and $T_{\textbf{x}}(\beta)$.  First suppose $T_{\textbf{x}}(\alpha) = k$, $T_{\textbf{x}}(\beta) = k' > k$, and $T_{\textbf{x}}(a\alpha + b\beta) > k'$.  Then we can form the reduced expression $\textbf{x}' = (s_{k'+1},\ldots,s_n)$ for some $w' \in W$.  The root sequence for $\textbf{x}'$ contains neither $\alpha$ nor $\beta$, but it does contain $a\alpha + b\beta$.  By Lemma~\ref{l:inversionbiconvex}, this contradicts the biconvexity of $\Phi(\phi(\textbf{x}'))$.  Now suppose $T_{\textbf{x}}(a\alpha + b\beta) < k$. Then we can form the reduced expression $\textbf{x}' = (s_k,\ldots,s_n)$ for some $w' \in W$.  The root sequence of $\textbf{x}'$ contains $\alpha$ and $\beta$, but not $\lambda \alpha + \mu \beta$, a contradiction.\\ \\
We next suppose that $\alpha, a\alpha + b\beta \in \Phi(w)$ and $\beta \not\in \Phi(w)$.  Suppose that $T_{\textbf{x}}(a\alpha + b\beta) = k$ and $k > T_{\textbf{x}}(\alpha)$.  Then the expression $\textbf{x}' = (s_k,\ldots,s_n)$ is a reduced expression for some $w' \in W$.  The root sequence for $\textbf{x}'$ does not contain $\alpha$ or $\beta$ but it does contain $a\alpha +  b\beta$, contradicting the biconvexity of $\Phi(w')$.  Thus $T_{\textbf{x}}(\alpha) > T_{\textbf{x}}(a\alpha + b\beta)$.  Since $T_{\textbf{x}}(\beta) = 0$, the standardness property is satisfied for $\alpha,a\alpha + b\beta, \beta \in \Phi^+$.  As all cases are exhausted, $T_{\textbf{x}}$ is standard.
\end{proof}
\noindent
The following useful consequence of the biconvexity of $\Phi(w)$ is noted in \cite[Section 2]{fbI} and \cite[Section 2]{nilorbits} for simply laced Coxeter systems (that is, Coxeter systems such that for all $i,j \in I$, $m_{ij} = 2$ or $m_{ij} = 3$).
\begin{cor} \label{c:rootseqrem}
Let $\alpha,\beta,\lambda \in \Phi^+$ and suppose $\lambda$ lies in the convex cone spanned by $\alpha$ and $\beta$.
Let $w \in W$ and let $\textbf{x}$ be a reduced expression for $w$.
\begin{enumerate}[(1)]
\item If $\alpha,\beta,\lambda \in \Phi(w)$ then $\lambda$ occurs between $\alpha$ and $\beta$ in the root sequence  of $\textbf{x}$.
\item If $\alpha, \lambda \in \Phi(w)$, but $\beta \not\in \Phi(w)$, then $\lambda$ must occur before $\alpha$ in the root sequence of $\textbf{x}$.
\end{enumerate}
\end{cor}
\begin{proof}
Let $\overline{\theta}(\textbf{x}) = (\theta_1,\ldots,\theta_{\ell(w)})$ be the root sequence of $\textbf{x}$.  Let $\theta_i = \alpha$ and $\theta_j = \lambda$.  By Definition~\ref{d:standardencoding}, the standard encoding $T_{\textbf{x}}$ of $\textbf{x}$ satisfies $T_{\textbf{x}}(\theta_n) = n$ for all $n$ satisfying $1 \leq n \leq \ell(w)$.\\ \\
Suppose $\alpha, \beta, \lambda \in \Phi(w)$ and let $\theta_k = \beta$.  Then by Lemma~\ref{l:standard} and Definition~\ref{d:standard}, either $i < j < k$ or $k < j < i$, which proves the first assertion.\\ \\
Suppose $\alpha, \lambda \in \Phi(w)$, but $\beta \not\in \Phi(w)$.  Then Lemma~\ref{l:standard} and Definition~\ref{d:standard} imply that $T_{\textbf{x}}(\beta) = 0 < T_{\textbf{x}}(\lambda) < T_{\textbf{x}}(\alpha)$, which proves that $j < i$.  The second assertion follows.
\end{proof}
\begin{lem} \label{l:wlabel}
Let $T:\Phi^+ \rightarrow \mathbb{N}$ be a sequential standard labeling of $\Phi^+$ such that $\text{supp}(T)$ is finite.  Let $\gamma, \delta$ be the canonical simple roots for an inversion set $\Psi$, let $m = |s_\gamma s_\delta|$, and let $\overline{\gamma}$ and $\overline{\delta}$ be the associated local root sequences.  We have:
\begin{enumerate}[(1)]
\item If $T(\gamma) = 0$ and $T(\delta) = 0$, then $T(\lambda) = 0$ for all $\lambda \in \Psi$.\\
\item If $T(\gamma) \neq 0$ and $T(\delta) = 0$, then
there exists $k < m$ such that $$T(\gamma_1) > \cdots > T(\gamma_k)$$ and $T(\gamma_i) = 0$ for $i > k$.\\
\item If $T(\gamma) \neq 0$ and $T(\delta) \neq 0$ then $m$ is finite and we have either $$T(\gamma_1) < \cdots < T(\gamma_m)$$
\begin{center}
or
\end{center}
$$T(\gamma_1) > \cdots > T(\gamma_m).$$
\end{enumerate}
\end{lem}
\begin{proof}
Since $T$ is sequential with finite support, all the inequalities in Definition~\ref{d:standard} are strict when they apply to roots in $\text{supp}(T)$.\\ \\
First suppose $T(\gamma) = 0$ and $T(\delta) = 0$.  Then by Lemmas~\ref{l:biconvexity} and \ref{l:infinitebiconvexity} and Definition~\ref{d:standard}, we have $T(\gamma_i) = T(\delta_j) = 0$ for any $i,j \in \mathbb{N}^+$.  Thus by Proposition~\ref{p:biconvexstructure}, for any $\lambda \in \Psi$, $T(\lambda) = 0$.\\ \\
For case $(2)$, we suppose that $T(\gamma) \neq 0$ and $T(\delta) = 0$.  Since $\text{supp}(T)$ is finite, and since $\delta \not\in \text{supp}(T)$, there must exist a smallest index $k \geq 1$ such that $\gamma_k \in \text{supp}(T)$ and $\gamma_{k+1} \not\in \text{supp}(T)$.  Suppose towards a contradiction that $\gamma_{k + j} \in \text{supp}(T)$ for some $j > 1$.   We have $T(\gamma_1) > T(\gamma_{k+1}) = 0$ because $\gamma_1 \in \text{supp}(T)$.  By Lemma~\ref{l:biconvexity} and Definition~\ref{d:standard}, either $$T(\gamma_1) \leq T(\gamma_{k+1}) \leq T(\gamma_{k+j})$$
\begin{center}
or
\end{center}
$$T(\gamma_1) \geq T(\gamma_{k+1}) \geq T(\gamma_{k+j}),$$ contradicting the assumption that $T(\gamma_1),T(\gamma_{k+j}) \neq 0$ and $T(\gamma_{k+1}) = 0$.\\ \\
By Lemma~\ref{l:biconvexity} and Definition~\ref{d:standard}, for any $i$ satisfying $1 < i < k$, we have $\gamma_i \in \text{supp}(T)$.  Thus, for $1 \leq i \leq k$, we have $\gamma_i \in \text{supp}(T)$, and for $k < i \leq m$ (with $i$ finite if $m = \infty$), we have $\gamma_i \not\in \text{supp}(T)$.\\ \\
Let $i$ be such that $1 < i \leq k$.  Then, $\gamma_i$ is a positive linear combination of $\gamma$ and $\delta$ and $T(\delta) = 0$, so by Definition~\ref{d:standard}, we have
$$T(\gamma_1) > T(\gamma_i) > T(\delta) = 0$$
for all $i$ such that $1 < i \leq k$.  Now, by Lemma~\ref{l:biconvexity} and Definition~\ref{d:standard}, we have $T(\gamma_1) > T(\gamma_i) > T(\gamma_{i+1})$ for any $i$ such that $2 \leq i \leq k$.  This forces $T(\gamma_1) > T(\gamma_2) > \cdots > T(\gamma_k)$.\\ \\
Turning to case $(3)$, suppose $\gamma, \delta \in \text{supp}(T)$.  Assume by way of contradiction that $m = \infty$.  Then for each $i > 1$, by Lemma~\ref{l:infinitebiconvexity} and Definition~\ref{d:standard}, $T(\gamma_i)$ lies between $T(\gamma)$ and $T(\delta)$.  This implies that $\text{supp}(T)$ is infinite, contradicting the hypotheses.
Thus $m$ is finite.\\ \\
If $T(\gamma) < T(\delta)$, then Lemma~\ref{l:biconvexity} and Definition~\ref{d:standard} imply that $$T(\gamma_1) < \cdots < T(\gamma_m) = T(\delta).$$  Similarly, if $T(\gamma) > T(\delta)$, then $T(\gamma_1) > \cdots > T(\gamma_m)$, as desired.
\end{proof}
\chapter{Correspondences with reduced expressions for a Coxeter group element}   \label{theoremchapt}
The goal of this chapter is to establish a correspondence between subsets of $\Phi^+$ and the elements of $W$, as well a correspondence between standard labelings of $\Phi^+$ with support equal to $\Phi(w)$ and the reduced expressions for $w$.  In addition to these correspondences, we introduce an incidence structure that faithfully represents the local dihedral subsystem structure present in $\Phi^+$ and $\Phi(w)$.
\section{A reduced expression correspondence with labelings}
In \cite{reducedweyl}, Kra\'skiewicz gave a 1--1 correspondence between reduced expressions for an element $w$ of a crystallographic Weyl group and standard $w$-tableaux.  In adopting the labeling terminology of Fomin et al. \cite{balancedlabel}, we will decouple any structure associated to $w$ and its inversion set $\Phi(w)$ from structure imposed by a reduced expression.  In our terminology, a standard $w$-tableau is called a standard sequential labeling.  Kra\'skiewicz's notion of encoding a reduced expression is the same as ours.  However, we have generalized the meaning of standard labeling so that the correspondence applies in the expanded context of arbitrary Coxeter groups.  We now recall Kra\'skiewicz's Theorem and his definition of ``standard'' translated into our terminology for comparison purposes.
\begin{defn} \cite[Definition 2.5]{reducedweyl}  \label{d:krakdef}
We say that a labeling $T$ is \emph{standard} if for every $\gamma \in \Phi^+$ and every decomposition $\gamma = \alpha + \beta$ of $\gamma$ into a sum of two positive roots, $T(\gamma)$ is between $T(\alpha)$ and $T(\beta)$ (i.e. $T(\alpha) \geq T(\gamma) \geq T(\beta)$ or $T(\alpha) \leq T(\gamma) \leq T(\beta)$).
\end{defn}
\begin{theorem}[\textbf{Kra\'skiewicz}] \label{t:krakthm}
Let $W$ be the Weyl group of a crystallographic root system $\Phi$.  Let $w \in W$ and let $T$ be a sequential labeling of $\Phi^+$ such that $\text{supp}(T) = \Phi(w)$.  Then $w$ encodes a reduced expression if and only if $T$ is standard.
\end{theorem}
\begin{proof}
See \cite[Theorem 2.6]{reducedweyl}.
\end{proof}
\begin{lem} \label{l:decomp}
Let $(W,S)$ be a Coxeter system with root system $\Phi$.  Let $\gamma \in \Phi^+$.  Suppose there do not exist $\alpha, \beta \in \Phi^+$, $\alpha \neq \beta$, and $a,b > 0$ such that $\gamma = a \alpha + b \beta$.  Then $\gamma$ is a simple root.
\end{lem}
\begin{proof}
Suppose there do not exist $\alpha, \beta \in \Phi^+$, $\alpha \neq \beta$, and $a,b > 0$ satisfying $\gamma = a \alpha + b \beta$.  Let $\delta \in \Phi^+$ be distinct from $\gamma$.  By the reflection formula (\ref{e:reflection}), $s_\gamma (\delta) = \delta - 2B(\delta, \gamma) \gamma$.  Suppose towards a contradiction that $s_\gamma (\delta) \in \Phi^-$.  Then $\delta \in \Phi^+$ implies $\delta \neq s_\gamma(\delta)$.  Since $\gamma,\delta \in \Phi^+$ and $s_\gamma (\delta) \in \Phi^-$, we have $-2B(\delta,\gamma) < 0$ by the reflection formula (\ref{e:reflection}).  Thus $$\gamma = \frac{1}{2B(\delta,\gamma)} \delta + \frac{1}{2B(\delta,\gamma)}(-s_\gamma(\delta)),$$ contradicting the assumption that no such decomposition exists.  It follows that $s_\gamma (\delta) \in \Phi^+$ for each $\delta \in \Phi^+$ distinct from $\gamma$, so that $\Phi(s_\gamma) = \{\gamma\}$.  However, this implies that $\ell(s_\gamma) = 1$, so $\gamma$ is a simple root.
\end{proof}
\noindent
The proof of the next proposition is structurally identical to the proof of Theorem~\ref{t:krakthm}, but the details with respect to Lemma~\ref{l:decomp} and calculations verifying that a labeling is standard in the proof of Proposition~\ref{p:betweenthm} are different.  Our primary contribution to this proposition is in giving the characterization that applies for the weakened hypothesis on $(W,S)$.  For purposes of coherence, we include the full proof.\\ \\
Recall from Definition~\ref{d:labeling} that the support of $T$, denoted $\text{supp}(T)$, is the set of elements not mapped to zero.  Also recall that a sequential labeling has finite support and satisfies $T(\text{supp}(T)) = \{1,\ldots,|\text{supp}(T)|\}$.
\begin{prop} \label{p:betweenthm}
Let $(W,S)$ be a Coxeter system.  Let $T:\Phi^+ \rightarrow \mathbb{N}$ be a labeling of $\Phi^+$.  Then $T$ encodes a reduced expression $\textbf{x}$ for some $w \in W$ if and only if $T$ is a sequential standard labeling of $\Phi^+$.
\end{prop}
\begin{proof}
Let $A = \text{supp}(T)$.
If $T:\Phi(w) \rightarrow \mathbb{N}$ is the standard encoding of a reduced expression $\textbf{x}$ for some $w \in W$, then by Definition~\ref{d:standardencoding}, $A = \Phi(w)$ and $T(A) = \{1,\ldots,|A|\}$, so $T$ is a sequential labeling.  Also, $T$ is a standard labeling by Lemma~\ref{l:standard}.\\ \\
We prove the converse by induction on $|A|$, which is finite since $T$ is assumed to be a sequential labeling.  For the base case, if $A = \emptyset$, then $T(\lambda) = 0$ for all $\lambda \in \Phi^+$, so $T$ is the standard encoding of the empty word, which is a reduced expression for the identity.  Now let $T$ be a sequential standard labeling of $\Phi^+$ such that $|A| = k$, $k > 0$.  The induction hypothesis states that if $T':\Phi^+ \rightarrow \mathbb{N}$ is a sequential standard labeling of $\Phi^+$ such that $\left|\text{supp}(T')\right| < k$, then $T'$ encodes a reduced expression $\textbf{x}'$ for some $w' \in W$.\\ \\
Let $\gamma = T^{-1}(k)$.  If $\gamma$ is not a simple root then $\gamma = a \alpha + b \beta$ for some $\alpha, \beta \in \Phi^+$, $\alpha \neq \beta$,  by Lemma~\ref{l:decomp}.  Since $k$ is the largest value in the image of $T$ and since $T$ is a bijection when restricted to its support, we have that $T(\alpha) < k$ and $T(\beta) < k$.  This contradicts the assumption that $T$ is standard.  Thus, $\gamma$ must be simple.  We define a labeling $T':\Phi^+ \rightarrow \mathbb{N}$ by
\begin{equation*}
T'(\lambda) = \begin{cases}
T(s_\gamma(\lambda)) & \text{if $\lambda \neq \gamma$};\\
0 & \text{if $\lambda = \gamma$}.\\
\end{cases}
\end{equation*}
\noindent
To show that $T'$ is standard, there are two cases to consider:  one of the positive roots in a decomposition is $\gamma$ or neither of the positive roots in a decomposition is $\gamma$.\\ \\
We first suppose that $a,b > 0$, $\alpha, \beta, a \alpha + b\beta \in \Phi^+$ are such that $\alpha \neq \beta$, and neither $\alpha$ nor $\beta$ is $\gamma$.  Then $s_\gamma (\alpha)$, $s_\gamma (\beta)$, and $a s_\gamma(\alpha) + b s_\gamma(\beta)$ are all positive roots by Proposition~\ref{p:basiccoxeter}(3).  Thus $$T(s_\gamma(\alpha)) \leq T(a s_\gamma(\alpha) + b s_\gamma(\beta)) \leq T(s_\gamma(\beta))$$
\begin{center}
if and only if
\end{center}
$$T'(\alpha) \leq T'(a \alpha + b \beta) \leq T'(\beta)$$ (and similarly for the reverse inequalities).  Thus, in this case, $T'$ satisfies one of the inequalities required for $T'$ to be standard.\\ \\
Now suppose that one of the roots in a decomposition is $\gamma$.  That is, suppose $a,b > 0$, $\alpha, \beta, a\alpha + b\beta \in \Phi^+$, $\alpha \neq \beta$, and without loss of generality suppose $\alpha = \gamma$.  Since $T(\alpha)$ is the maximum value of $T(\Phi^+)$ and $T$ is a standard labeling, we must have $T(\alpha) \geq T(a\alpha + b\beta) \geq T(\beta)$.  Since $\alpha$ is a simple root, Proposition~\ref{p:basiccoxeter}(3) implies that $as_\alpha(\alpha) + b s_\alpha(\beta)$ and $s_\alpha(\beta)$ are positive roots.  From $as_\alpha(\alpha) + bs_\alpha(\beta) = -a \alpha + bs_\alpha(\beta)$, it follows that $s_\alpha(\beta) = \frac{1}{b} (-a \alpha + b s_\alpha(\beta)) + \frac{a}{b} \alpha$ is a nonnegative linear combination of $as_\alpha(\alpha) + bs_\alpha(\beta)$ and $\alpha$.  Since $T$ is standard and $T(\alpha)$ is the maximum of $T(\Phi^+)$, we have $T(s_\alpha(a \alpha + b \beta)) \leq T(s_\alpha(\beta)) \leq T(\alpha)$.  This implies $T'(\alpha) \leq T'(a\alpha + b\beta) \leq T'(\beta)$, because $T'(\alpha) = 0$.  Thus, in both cases, one of the inequalities required for $T'$ to be standard is satisfied.  Thus $T'$ is standard.\\ \\
Since $T$ is sequential, and $s_\gamma$ is a bijection, and the root with the highest label in $T$ is labeled $0$ in $T'$, it follows that $T'$ is sequential.  Thus, the inductive hypothesis implies that $T'$ is the standard encoding for a reduced expression $\textbf{x}'$ for some $w' \in W$.  Since $\gamma$ is simple, we may assume $\gamma = \alpha_i$ for some $i \in I$.  We let $\textbf{x} = \textbf{x}' \, i$ and $w = w's_\gamma$.  Then $T'(\gamma) = 0$ and $T'$ is the standard encoding for $\textbf{x}'$, so it follows from Definition~\ref{d:standardencoding} that $\gamma \not\in \Phi(w')$ and $w'(\gamma)$ is therefore a positive root.  By Proposition~\ref{p:basiccoxeter} parts (2) and (4), $\ell(w s_\gamma) = \ell(w) + 1$, so $\textbf{x}$ is a reduced expression for $w$.  Furthermore, we can apply Lemma~\ref{l:gendeletion} (with $\textbf{b}$ being the empty string) to get $\overline{\theta}(\textbf{x}) = s_\gamma [\overline{\theta}(\textbf{x}')] (\theta_i)$.  From the construction of $T'$ we get that $T$ is the standard encoding for $\textbf{x}$.
\end{proof}
\noindent
The next corollary is used to construct the correspondence of Proposition~\ref{p:correspondence1}.
\begin{cor} \label{c:labelcharacterize}
Let $(W,S)$ be a Coxeter system and let $w \in W$.  Let $T:\Phi^+ \rightarrow \mathbb{N}$ be a labeling of $\Phi^+$.  Then $T$ encodes a reduced expression $\textbf{x}$ for $w$ if and only if $T$ is a sequential standard labeling of $\Phi^+$ such that $\text{supp}(T) = \Phi(w)$.
\end{cor}
\begin{proof}
If $T$ encodes a reduced expression $\textbf{x}$ for $w$ then $\text{supp}(T) = \Phi(w)$ by Definition~\ref{d:standardencoding}.  By Proposition~\ref{p:betweenthm}, $T$ is a sequential standard labeling of $\Phi^+$.\\ \\
Conversely, if $T$ is a standard sequential labeling of $\Phi^+$, then Proposition~\ref{p:betweenthm} implies that $T$ encodes a reduced expression $\textbf{x}$ for some $w \in W$.  Since we are also assuming that $\text{supp}(T) = \Phi(w)$, Definition~\ref{d:standardencoding} implies that $\Phi(w) = \Phi(\phi(\textbf{x}))$.  By Proposition~\ref{p:basiccoxeter}(7), $w = \phi(\textbf{x})$ so that $\textbf{x}$ is a reduced expression for $w$.
\end{proof}
\noindent
We regard the next corollary as fundamental to our work in Chapter 5, though it might be folklore.  In the paper \cite{coxeterorderings}, Bj{\"{o}}rner gives a characterization of subsets of $\Phi^+$ that take the form $\Phi(w)$.  The characterization is that the set must be finite and biconvex (see \cite[Proposition 3]{coxeterorderings}), so long as $W$ is a finite Coxeter group.  In \cite[Section 2]{nilorbits}, Fan and Stembridge assert that the same characterization holds in the setting of arbitrary simply-laced Coxeter groups.  Since there is a natural correspondence between reflections and positive roots, it is also interesting to characterize the corresponding subsets of reflections.  In \cite[Lemma 2.11]{dyerhecke}, Dyer characterizes such subsets as the initial sections of reflection orders.  He then notes (see \cite[Remark 2.12]{dyerhecke}) that through the natural correspondence between reflections and positive roots, the initial sections of reflection orders correspond to biconvex sets but that the converse is open.\\ \\
The next corollary asserts that sets of the form $\Phi(w)$ can be characterized as the finite biconvex sets of positive roots in the generalized setting of arbitrary Coxeter groups.  Via Dyer's \cite[Lemma 2.11]{dyerhecke}, it then follows that finite biconvex sets of positive roots correspond to initial sections of reflection orders.  Whether infinite biconvex sets of positive roots (in a necessarily infinite Coxeter group) correspond to initial sections of reflection orders is open.
\begin{cor}  \label{c:finitebiconvex}
Let $A \subseteq \Phi^+$.  Then $A = \Phi(w)$ for some $w \in W$ if and only if $A$ is finite and biconvex.
\end{cor}
\begin{proof}
If $A = \Phi(w)$, then $A$ is finite and biconvex by Lemma~\ref{l:inversionbiconvex}.  By \cite[Proposition 5.2.1]{bb}, there exists a total ordering $<$ on $\Phi^+$ such that if $a,b > 0$ and $\alpha, \beta, a\alpha + b\beta \in \Phi^+$ then $\alpha < a\alpha + b\beta < \beta$ or $\beta < a\alpha + b\beta < \alpha$. We form $T:\Phi^+ \rightarrow \mathbb{N}$ by setting $T(\lambda) = 0$ for $\lambda \not\in A$ and $T(\alpha) = i$ if $\alpha$ is the $i$-th highest root in $A$ relative to the total ordering $<$.  It follows that $T$ is standard, so that by Proposition~\ref{p:betweenthm} there exists a reduced expression $\textbf{x}$ for some $w \in W$ such that $T$ is the standard encoding of $\textbf{x}$.  By Definition~\ref{d:standardencoding}, $A = \Phi(w)$.
\end{proof}
\begin{prop} \label{p:correspondence1}
Let $(W,S)$ be a Coxeter system and let $w \in W$.  Then there is a 1--1 correspondence between the set $\mathcal{R}(w)$ of reduced expressions for $w$ and the set $\textrm{Lab}(w)$ of all sequential standard labelings $T$ such that $\text{supp}(T) = \Phi(w)$.  The correspondence is given by $\textbf{x} \longleftrightarrow T_{\textbf{x}}$.
\end{prop}
\begin{proof}
Let $\textbf{x}$ be a reduced expression for $w$.  By Corollary~\ref{c:labelcharacterize}, the standard encoding $T_{\textbf{x}}$ satisfies $\text{supp}(T_{\textbf{x}}) = \Phi(w)$ and is a sequential standard labeling of $\Phi^+$.  Thus we may define $\mathcal{F}:\mathcal{R}(w) \rightarrow \textrm{Lab}(w)$ by $\mathcal{F}(\textbf{x}) = T_{\textbf{x}}$.\\ \\
Conversely, Corollary~\ref{c:labelcharacterize} also shows that given any sequential standard labeling $T$ of $\Phi^+$ such that $\text{supp}(T) = \Phi(w)$, there exists a reduced expression $\textbf{x}$ for $w$ such that $T = T_{\textbf{x}}$.  By Definition~\ref{d:standardencoding} and Proposition~\ref{p:basiccoxeter}(9), the reduced expression $\textbf{x}$ is unique.  Thus we may define $\mathcal{G}:\mathrm{Lab}(w) \rightarrow \mathcal{R}(w)$ by $\mathcal{G}(T_{\textbf{x}}) = \textbf{x}$.\\ \\
Then we have $\mathcal{G}(\mathcal{F}(\textbf{x})) = \mathcal{G}(T_{\textbf{x}}) = \textbf{x}$, by the definitions of $\mathcal{F}$ and $\mathcal{G}$.\\ \\
Similarly, by the definitions of $\mathcal{F}$ and $\mathcal{G}$, we have $\mathcal{F}(\mathcal{G}(T_{\textbf{x}})) = \mathcal{F}(\textbf{x}) = T_{\textbf{x}}$, so that $\mathcal{F}$ and $\mathcal{G}$ are inverse to each other.
\end{proof}
\section{Combinatorial Models}
In this section, for any $w \in W$, we give a combinatorial model of the set $\Phi(w)$ that faithfully represents the local root structure.  We give a related combinatorial representation of the reduced expressions for $w$.  The terminology we introduce is very similar to that of hypergraphs or incidence structures from graph theory, but we need some additional features.
\begin{defn} \label{d:segmentstructure}
A \emph{segment structure} is a triple $\mathcal{I} = (P,\mathcal{L}, B)$, where $P$ is a set of elements called \emph{points}, $\mathcal{L}$ is a set of non-empty subsets $L$ of $P$ called \emph{lines} satisfying $|L| \geq 2$, and $B$ is a ternary relation on $P$ satisfying:
\begin{enumerate}[(B1)]
\item  If $B(p_1,p_2,p_3)$ holds, then $p_1$, $p_2$, and $p_3$ are pairwise distinct.
\item If $B(p_1,p_2,p_3)$ holds, then there exists $L \in \mathcal{L}$ such that $p_1,p_2,p_3 \in L$.
\item If $p_1,p_2,p_3 \in L$ for some $L \in \mathcal{L}$ and $p_1,p_2,p_3$ are pairwise distinct, then there exists a permutation $$\sigma:\{p_1,p_2,p_3\} \rightarrow \{p_1,p_2,p_3\}$$ such that $B(\sigma(p_1),\sigma(p_2),\sigma(p_3))$ holds.
\item If $B(p_1, p_2, p_3)$ holds, then $B(p_3,p_2,p_1)$ holds, and for any permutation $$\sigma:\{p_1,p_2,p_3\} \rightarrow \{p_1,p_2,p_3\}$$ other than the identity or the transposition $(p_1 \ p_3)$, the ternary relation $B(\sigma(p_1),\sigma(p_2),\sigma(p_3))$ does not hold.
\end{enumerate}
\noindent
If $B(p_1,p_2,p_3)$ holds, then we say that \emph{$p_2$ lies between $p_1$ and $p_3$}.  If $p \in L$ and there do not exist $p', p'' \in L$ such that $B(p', p, p'')$, then we say $p$ is an \emph{endpoint of $L$}.  Otherwise, $p$ is an \emph{intermediate point of $L$}.  If for all $L \in \mathcal{L}$, $p$ is an endpoint of $L$, then we say $p$ is an \emph{endpoint of $\mathcal{I}$}.
\end{defn}
\begin{rem}  \label{rem:lines}
Observe that the lines $L \in \mathcal{L}$ may be infinite.  In the case that $L$ has fewer than three elements, no betweenness relations are satisfied, so every point in $L$ is an endpoint of $L$.
\end{rem}
\noindent
Pictorially, we represent a segment structure with points and (not necessarily straight) edges, similar to the drawings of graphs in graph theory.  However, we prefer to have any collinear points to appear collinear and for the betweenness relations to be respected in the pictorial representation.
\begin{ex} \label{e:segmentex}
Let $A_3$ be the Coxeter system with set generating set given by $S = \{s,t,u\}$ and Coxeter relations $m_{s,t} = m_{t,u} = 3$ and $m_{s,u} = 2$.  As an abstract group, $W(A_3)$ is isomorphic to the symmetric group $S_4$.  It is known that
\begin{equation*}
\Phi^+ = \{ \alpha_s,\alpha_t,\alpha_u, \alpha_s + \alpha_t, \alpha_t + \alpha_u, \alpha_s + \alpha_t + \alpha_u \}.  \end{equation*}
We let $\mathcal{I} = (P, \mathcal{L}, B)$ be the segment structure obtained by letting $P = \Phi^+$, and $\mathcal{L} = \{ \Phi^+_{\{\alpha,\beta\}} \ : \ \alpha, \beta \in \Phi^+ \text{ and } \alpha \neq \beta \}$.  If $\gamma$, $\delta$ are the canonical simple roots for some inversion set $\Phi^+_{\alpha, \beta}$, and $\gamma_i,\gamma_j, \gamma_k \in \Phi^+$ are entries in the associated $\overline{\gamma}$-sequence, then we set $B(\gamma_i,\gamma_j,\gamma_k)$ if and only if we have either $i < j < k$ or $k < j < i$.  The pictorial representation is given below.

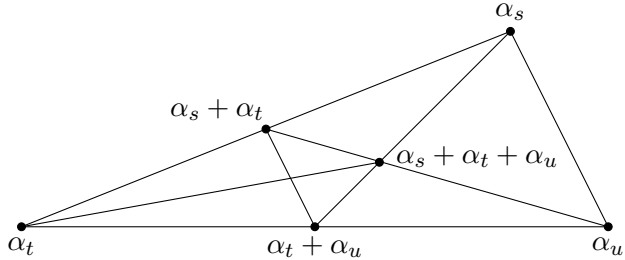
\vspace{0.1cm}

\begin{tikzpicture}[scale=2.6] \label{Aexamplefig}
\filldraw           (3,0) circle (0.02)
                    (1.5,0) circle (0.02)
                    (1.25,0.5) circle (0.02)
                    (1.83,0.33) circle (0.02)
                    (0,0) circle (0.02)
                    (2.5,1) circle (0.02);
\draw (0,0) -- (3,0);
\draw (3,0) -- (2.5,1);
\draw (2.5,1) -- (0,0);
\draw (1.5,0) -- (2.5,1);
\draw (1.25,0.5) -- (3,0);
\draw (0,0) -- (1.83,0.33);
\draw (1.25,0.5) -- (1.5,0);
\draw (0, -0.1) node {$\alpha_t$};
\draw (1.5, -0.1) node {$\alpha_t + \alpha_u$};
\draw (3, -0.1) node {$\alpha_u$};
\draw (1,0.6) node {$\alpha_s + \alpha_t$};
\draw (2.5,1.1) node {$\alpha_s$};
\draw (2.33,0.36) node {$\alpha_s + \alpha_t + \alpha_u$};
\end{tikzpicture}

\vspace{0.1cm}

\noindent
Observe that $\alpha_s$, $\alpha_t$, and $\alpha_u$ satisfy the definition for an endpoint of $\mathcal{I}$.  On the other hand, $\alpha_s + \alpha_t$ is an endpoint of the line $$L = \{\alpha_s + \alpha_t, \alpha_s + \alpha_t + \alpha_u, \alpha_u\},$$ but is an intermediate point of the line $$L = \{\alpha_s, \alpha_s + \alpha_t, \alpha_t\},$$ and hence is not an endpoint of $\mathcal{I}$.
\end{ex}
\noindent
In the next definition we give a ternary relation on $\Phi^+$ that is intended to capture the notion of ``betweenness'' for inversion sets.
\begin{defn} \label{d:betweenness}
Let $\lambda,\mu,\nu \in \Phi^+$ be pairwise distinct positive roots.  If there exists an inversion set $\Psi$ with a canonical simple root $\gamma$ such that either $\lambda <_{\Psi,\gamma} \mu <_{\Psi,\gamma} \nu$ or $\nu <_{\Psi,\gamma} \mu <_{\Psi,\gamma} \lambda$, then we say that \emph{$\mu$ is between $\lambda$ and $\nu$}.  We denote this ternary relation by $B_W(\lambda, \mu,\nu)$ and say that \emph{$B_W(\lambda,\mu,\nu)$ holds} if such a $\Psi$ exists.  If $\lambda, \mu, \nu$ are not pairwise distinct positive roots or if there does not exist such an inversion set $\Psi$, then we say that \emph{$B_W(\lambda,\mu,\nu)$ does not hold}.
\end{defn}
\begin{lem} \label{l:inversionsegment}
Let $(W,S)$ be a Coxeter system and let $\Phi^+$ be the associated positive system.  Let $\mathcal{W} = (\Phi^+,\text{Inv}(\Phi^+), B_W)$, where $B_W$ is the betweenness relation of Definition~\ref{d:betweenness} and $\text{Inv}(\Phi^+)$ is the set of all inversion sets contained in $\Phi^+$.  Then $\mathcal{W}$ is a segment structure.
\end{lem}
\begin{proof}
By definition, an inversion set consists of at least two positive roots, so all of the lines are nonempty.  If $B_W(\lambda, \mu, \nu)$ holds, then there exists an inversion set $\Psi$ containing $\lambda, \mu, \nu$, which are distinct by Definition~\ref{d:betweenness}.  Thus $B_W$ satisfies properties $(B1)$ and $(B2)$.\\ \\
Suppose $\lambda, \mu, \nu \in \Psi = \Phi^+_{\{\alpha,\beta\}}$ and that $\gamma$ is a canonical simple root for $\Phi^+_{\{\alpha,\beta\}}$.  Then, since $\leq_{\Psi,\gamma}$ is a total ordering, there exists a permutation $\sigma$ of $\{\lambda,\mu,\nu\}$ such that $\sigma(\lambda) <_{\Psi,\lambda} \sigma(\mu) <_{\Psi,\gamma} \sigma(\nu)$, so $B_W(\sigma(\lambda),\sigma(\mu),\sigma(\nu))$ holds, proving that $B_W$ satisfies property $(B3)$.\\ \\
Suppose $B_W(\lambda,\mu,\nu)$ holds.  Then there exists an inversion set $\Psi$ such that either $\lambda <_{\Psi,\gamma} \mu <_{\Psi,\gamma} \nu$ or $\nu <_{\Psi,\gamma} \mu <_{\Psi,\gamma} \lambda$.  Thus $B_W(\nu,\mu,\lambda)$ holds also.  Since $<_{\Psi,\gamma}$ is a total order, no other permutation of $\lambda, \mu, \nu$ is consistent with $<_{\Psi,\gamma}$.  By Lemma~\ref{l:oneroot}, if $\Upsilon$ is another inversion set containing $\lambda$, it cannot contain $\mu$ or $\nu$, and hence no other betweenness relations for $\lambda,\mu,\nu$ are possible.  This proves that $B_W$ satisfies property $(B4)$ and hence $\mathcal{W}$ is a segment structure.
\end{proof}
\begin{defn}  \label{d:standardsegment}
Let $(W,S)$ be a Coxeter system and let $\Phi^+$ be the associated positive system.  Let $\mathcal{W}$ be the segment structure of Lemma~\ref{l:inversionsegment}.  We call $\mathcal{W}$ the \emph{standard segment structure of $W$}.
\end{defn}
\begin{defn} \label{d:incidencelabel}
Let $\mathcal{I} = (P,\mathcal{L},B)$ be a segment structure and $T:P \rightarrow \mathbb{N}_0$.  Then we call $T$ an \emph{$\mathcal{I}$-labeling}.  We denote by $\text{supp}(T)$ the set of all $p$ such that $T(p) \neq 0$.  We say that $T$ is a \emph{sequential $\mathcal{I}$-labeling} if $T(\text{supp}(T)) = \{1,\ldots,|\text{supp}(T)|\}$.  Suppose that whenever $B(p_1,p_2,p_3)$ holds, either $T(p_1) \leq T(p_2) \leq T(p_3)$ or $T(p_3) \leq T(p_2) \leq T(p_1)$.  Then we call $T$ a \emph{standard $\mathcal{I}$-labeling}.
\end{defn}
\begin{lem}  \label{l:betweencone}
Let $\mathcal{W}$ be the standard segment structure of $W$ and let $\lambda,\mu,\nu \in \Phi^+$ be distinct positive roots.  Then $B_W(\lambda, \mu, \nu)$ holds if and only if $\mu$ lies in the convex cone spanned by $\lambda$ and $\nu$.
\end{lem}
\begin{proof}
Suppose $B_W(\lambda,\mu,\nu)$ holds.  Then there exists an inversion set $\Psi$ and canonical simple root $\gamma$ of $\Psi$ such that either $\lambda <_{\Psi,\gamma} \mu <_{\Psi,\gamma} \nu$ or $\nu <_{\Psi,\gamma} \mu <_{\Psi,\gamma} \lambda$.  If $\Psi$ is finite, then Definition~\ref{d:totalinversion} and Lemma~\ref{l:biconvexity} imply that $\mu$ lies in the convex cone spanned by $\lambda$ and $\nu$.  If $\Psi$ is infinite, then Definition~\ref{d:totalinversion}, Lemma~\ref{l:biconvexity}, and Lemma~\ref{l:infinitebiconvexity} imply that $\mu$ lies in the convex cone spanned by $\lambda$ and $\nu$.\\ \\
Conversely, suppose that $\mu$ lies in the convex cone spanned by $\lambda$ and $\nu$.  Then we may form the inversion set $\Psi = \Phi^+_{\{\lambda,\nu\}} = \text{span}(\{\lambda,\nu\}) \cap \Phi^+$.  Since $\mu \in \Phi^+_{\{\lambda,\nu\}}$, we have that $\lambda,\mu,\nu$ are comparable relative to $<_{\Psi,\gamma}$, where $\gamma$ is a canonical simple root of $\Psi$, since $<_{\Psi,\gamma}$ is a total order.  Lemma~\ref{l:convexcone} shows that the only orders consistent with Lemma~\ref{l:biconvexity} and Lemma~\ref{l:infinitebiconvexity} are $\lambda <_{\Psi,\gamma} \mu <_{\Psi,\gamma} \nu$ and $\nu <_{\Psi,\gamma} \mu <_{\Psi,\gamma} \lambda$.  Thus $B_W(\lambda, \mu, \nu)$ holds.
\end{proof}
\begin{prop} \label{p:workhorse}
Let $\mathcal{W} = (\Phi^+,\text{Inv}(\Phi^+),B_W)$ be the standard segment structure of $W$.  The function $T:\Phi^+ \rightarrow \mathbb{N}_0$ is a standard $\mathcal{W}$-labeling if and only if $T$ is a standard labeling of $\Phi^+$.
\end{prop}
\begin{proof}
Suppose $\alpha,\beta, a\alpha + b\alpha \in \Phi^+$, where $a,b > 0$ and $\alpha \neq \beta$.  By Lemma~\ref{l:betweencone}, $B_W(\alpha, a\alpha + b\beta, \beta)$ holds.  Thus, if $T$ is a standard $\mathcal{W}$-labeling, then $T(\alpha) \leq T(a\alpha + b\beta) \leq T(\beta)$ or $T(\beta) \leq T(a\alpha + b\beta) \leq T(\alpha)$.  Since $\alpha, a\alpha + b\beta, \beta$ were arbitrary with $a,b > 0$ and $\alpha \neq \beta$, we have that $T$ is a standard labeling of $\Phi^+$ by Definition~\ref{d:standard}.\\ \\
Conversely, suppose $T$ is a standard labeling of $\Phi^+$.  If $B_W(\lambda,\mu,\nu)$ holds, we have $\mu = a\lambda + b\nu$, where $a,b > 0$ and $\lambda \neq \nu$ by Lemma~\ref{l:betweencone}.  Thus, either $T(\lambda) \leq T(\mu) \leq T(\nu)$ or $T(\nu) \leq T(\mu) \leq T(\lambda)$.  By Definition~\ref{d:incidencelabel}, $T$ is a standard $\mathcal{W}$-labeling.
\end{proof}
\begin{cor} \label{c:workhorse}
Let $\mathcal{W} = (\Phi^+,\text{Inv}(\Phi^+),B_W)$ be the standard segment structure of $W$.  Let $w \in W$ and let $\mathrm{Lab}(\mathcal{W},w)$ denote the set of all sequential standard $\mathcal{W}$-labelings such that $\text{supp}(T) = \Phi(w)$.  Let $\mathrm{Lab}(w)$ denote the set of all sequential standard labelings of $\Phi^+$ such that $\text{supp}(T) = \Phi(w)$.  Then $\mathrm{Lab}(w) = \mathrm{Lab}(\mathcal{W},w)$.
\end{cor}
\begin{proof}
Both $\mathrm{Lab}(w)$ and $\mathrm{Lab}(\mathcal{W},w)$ consist of functions $T:\Phi^+ \rightarrow \mathbb{N}_0$ such that $\text{supp}(T) = \{1,\ldots,|\Phi(w)|\}$ (since the labelings are required to be sequential with support $\Phi(w)$).  Thus Proposition~\ref{p:workhorse} implies that $T \in \mathrm{Lab}(w)$ if and only if $T \in \mathrm{Lab}(\mathcal{W},w)$.
\end{proof}
\noindent
Note that the labelings of $\Phi^+$ and the $\mathcal{W}$-labelings of the previous corollary are defined in different ways.  The usage we have in mind for $\mathcal{W}$ is that we compute the geometrical structure in advance, then the $\mathcal{W}$-labelings are placed on top of this structure.  In this way, the reduced expression correspondence of the next corollary involves only combinatorial properties of a fixed discrete structure instead of the messy details of $\Phi^+$ and convex cones.
\begin{cor} \label{c:reducedtheorem}
Let $(W,S)$ be a Coxeter system and let $w \in W$.  Let
\begin{equation*}
\mathcal{W} = (\Phi^+, \text{Inv}(W), B_W)
\end{equation*}
be the standard segment structure of $\Phi^+$. Then there is a 1--1 correspondence between the set $\mathcal{R}(w)$ of all reduced expressions for $w$, and the set $\mathrm{Lab}(\mathcal{W},w)$ of all sequential standard $\mathcal{W}$-labelings such that $\text{supp}(T) = \Phi(w)$.  The correspondence is given by $$\textbf{x} \longleftrightarrow T_{\textbf{x}}.$$
\end{cor}
\begin{proof}
This follows directly from Proposition~\ref{p:correspondence1} and Corollary~\ref{c:workhorse}.
\end{proof}
\begin{ex} \label{ex:mainexample}
The Coxeter matrix of the Coxeter group of type $B_3$ is determined by the entries $m_{s,t} = 4$, $m_{s,u} = 2$, and $m_{t,u} = 3$.  The element $w = \phi(u, s, t, s, t, u)$ has four reduced expressions.  We have $$\Phi(w) = \{\alpha_u, \alpha_t + \alpha_u, \alpha_s + \sqrt{2} \alpha_t + \sqrt{2} \alpha_u, \sqrt{2} \alpha_s + \alpha_t + \alpha_u, \alpha_s, \sqrt{2} \alpha_s + 2 \alpha_t + \alpha_u \}.$$  The positive roots of $\Phi^+$ not in $\Phi(w)$ are $\alpha_t$, $\alpha_s + \sqrt{2} \alpha_t$, and $\sqrt{2} \alpha_s + \alpha_t$.  To avoid clutter in our explanations and figures, we associate names to each root in $\Phi^+$.  The first figure below depicts the standard segment structure of $\Phi^+$ and the chart to its right gives the names we use for each root in $\Phi^+$.  Thus, ``root $F$'' refers to the root $\sqrt{2}\alpha_s + 2\alpha_t + \alpha_u$ and ``the $B-E$ line'' refers to the line $$\{\alpha_s,\sqrt{2}\alpha_s + \alpha_t + \alpha_u,\alpha_s + \sqrt{2}\alpha_t + \sqrt{2}\alpha_u,\alpha_t + \alpha_u\}.$$
\noindent
Let $\mathcal{W}$ be the standard segment structure of $\Phi^+$.  The four figures following the first one are the sequential standard $\mathcal{W}$-labelings $T$ such that $\text{supp}(T) = \Phi(w)$, which we explain after listing the figures.  In all of the diagrams below, we omit the lines containing exactly two roots.  \\ \\
\begin{tikzpicture}[scale=2.5] \label{Btablefig}
\filldraw           (3,0) circle (0.02)
                    (0.75,0.75) circle (0.02)
                    (1.5,1.5) circle (0.02)
                    (0,0) circle (0.02)
                    (2.11,0.88) circle (0.02)
                    (2.38,0.62) circle (0.02)
                    (1.67,0.44) circle (0.02)
                    (1.33,.56) circle (0.02)
                    (1.63,0.68) circle (0.02);
\draw (0,0) -- (1.5,1.5);
\draw (1.5,1.5) -- (3,0);
\draw (0.75,0.75) -- (3,0);
\draw (0,0) -- (2.11,0.88);
\draw (1.5,1.5) -- (1.67,0.44);
\draw (0.75,0.75) -- (2.38,0.62);
\draw (0,0) -- (2.38,0.62);
\draw (-.1,-.05) node {$A$};
\draw (3.05,-.1) node {$B$};
\draw (1.67,0.32) node {$C$};
\draw (1.35,0.44) node {$D$};
\draw (0.63,0.76) node {$E$};
\draw (1.69,0.79) node {$F$};
\draw (2.48,0.63) node {$G$};
\draw (2.23,0.92) node {$H$};
\draw (1.39,1.52) node {$I$};
\node [right=1cm,text width=8cm] at (2.45,1.2)
{
\begin{tabular}{cl} Name & Root\\A & $\alpha_u$\\B & $\alpha_s$\\
C & $\sqrt{2}\alpha_s + \alpha_t + \alpha_u$\\D & $\alpha_s + \sqrt{2} \alpha_t + \sqrt{2}\alpha_u$\\E & $\alpha_t + \alpha_u$\\F & $\sqrt{2}\alpha_s + 2\alpha_t + \alpha_u$\\
G & $\sqrt{2} \alpha_s + \alpha_t$\\H & $\alpha_s + \sqrt{2} \alpha_t$\\
I & $\alpha_t$
\end{tabular}
};

\end{tikzpicture}
\\
The $\mathcal{W}$-labelings corresponding to the reduced expressions $(u,s,t,s,t,u)$ and $(s,u,t,s,t,u)$, respectively are:\\ \\
\begin{tikzpicture}[scale=1.8] \label{Bexample1fig}
\filldraw           (3,0) circle (0.02)
                    (0.75,0.75) circle (0.02)
                    (1.5,1.5) circle (0.02)
                    (0,0) circle (0.02)
                    (2.11,0.88) circle (0.02)
                    (2.38,0.62) circle (0.02)
                    (1.67,0.44) circle (0.02)
                    (1.33,.56) circle (0.02)
                    (1.63,0.68) circle (0.02);
\draw (0,0) -- (1.5,1.5);
\draw (1.5,1.5) -- (3,0);
\draw (0.75,0.75) -- (3,0);
\draw (0,0) -- (2.11,0.88);
\draw (1.5,1.5) -- (1.67,0.44);
\draw (0.75,0.75) -- (2.38,0.62);
\draw (0,0) -- (2.38,0.62);
\draw (-.1,-.05) node {$6$};
\draw (3.05,-.1) node {$2$};
\draw (1.67,0.32) node {$3$};
\draw (1.35,0.44) node {$4$};
\draw (0.63,0.76) node {$5$};
\draw (1.69,0.79) node {$1$};
\draw (2.48,0.63) node {$0$};
\draw (2.23,0.92) node {$0$};
\draw (1.39,1.52) node {$0$};

\filldraw           (6.5,0) circle (0.02)
                    (4.25,0.75) circle (0.02)
                    (5.0,1.5) circle (0.02)
                    (3.5,0) circle (0.02)
                    (5.61,0.88) circle (0.02)
                    (5.88,0.62) circle (0.02)
                    (5.17,0.44) circle (0.02)
                    (4.83,.56) circle (0.02)
                    (5.13,0.68) circle (0.02);
\draw (3.5,0) -- (5.0,1.5);
\draw (5.0,1.5) -- (6.5,0);
\draw (4.25,0.75) -- (6.5,0);
\draw (3.5,0) -- (5.61,0.88);
\draw (5.0,1.5) -- (5.17,0.44);
\draw (4.25,0.75) -- (5.88,0.62);
\draw (3.5,0) -- (5.88,0.62);
\draw (3.4,-.05) node {$6$};
\draw (6.55,-.1) node {$1$};
\draw (5.17,0.32) node {$3$};
\draw (4.85,0.44) node {$4$};
\draw (4.13,0.76) node {$5$};
\draw (5.19,0.79) node {$2$};
\draw (5.98,0.63) node {$0$};
\draw (5.73,0.92) node {$0$};
\draw (4.89,1.52) node {$0$};
\end{tikzpicture}
\\
The $\mathcal{W}$-labelings corresponding to the reduced expressions $(u,t,s,t,s,u)$ and $(u,t,s,t,u,s)$, respectively are:\\ \\
\begin{tikzpicture}[scale=1.8] \label{Bexample2fig}
\filldraw           (3,0) circle (0.02)
                    (0.75,0.75) circle (0.02)
                    (1.5,1.5) circle (0.02)
                    (0,0) circle (0.02)
                    (2.11,0.88) circle (0.02)
                    (2.38,0.62) circle (0.02)
                    (1.67,0.44) circle (0.02)
                    (1.33,.56) circle (0.02)
                    (1.63,0.68) circle (0.02);
\draw (0,0) -- (1.5,1.5);
\draw (1.5,1.5) -- (3,0);
\draw (0.75,0.75) -- (3,0);
\draw (0,0) -- (2.11,0.88);
\draw (1.5,1.5) -- (1.67,0.44);
\draw (0.75,0.75) -- (2.38,0.62);
\draw (0,0) -- (2.38,0.62);
\draw (-.1,-.05) node {$6$};
\draw (3.05,-.1) node {$5$};
\draw (1.67,0.32) node {$4$};
\draw (1.35,0.44) node {$3$};
\draw (0.63,0.76) node {$2$};
\draw (1.69,0.79) node {$1$};
\draw (2.48,0.63) node {$0$};
\draw (2.23,0.92) node {$0$};
\draw (1.39,1.52) node {$0$};

\filldraw           (6.5,0) circle (0.02)
                    (4.25,0.75) circle (0.02)
                    (5.0,1.5) circle (0.02)
                    (3.5,0) circle (0.02)
                    (5.61,0.88) circle (0.02)
                    (5.88,0.62) circle (0.02)
                    (5.17,0.44) circle (0.02)
                    (4.83,.56) circle (0.02)
                    (5.13,0.68) circle (0.02);
\draw (3.5,0) -- (5.0,1.5);
\draw (5.0,1.5) -- (6.5,0);
\draw (4.25,0.75) -- (6.5,0);
\draw (3.5,0) -- (5.61,0.88);
\draw (5.0,1.5) -- (5.17,0.44);
\draw (4.25,0.75) -- (5.88,0.62);
\draw (3.5,0) -- (5.88,0.62);
\draw (3.4,-.05) node {$5$};
\draw (6.55,-.1) node {$6$};
\draw (5.17,0.32) node {$4$};
\draw (4.85,0.44) node {$3$};
\draw (4.13,0.76) node {$2$};
\draw (5.19,0.79) node {$1$};
\draw (5.98,0.63) node {$0$};
\draw (5.73,0.92) node {$0$};
\draw (4.89,1.52) node {$0$};
\end{tikzpicture}
\\
It is easy to check that the four $\mathcal{W}$-labelings above are indeed sequential and standard and have $\Phi(w)$ as their support.  For the converse, we make an ad hoc argument for this specific example.\\ \\
Suppose that $T$ is a sequential standard $\mathcal{W}$-labeling with $\Phi(w)$ as support.  The roots not in $\Phi(w)$ must then be labeled $0$.  Thus, the roots $G$, $H$, and $I$ are labeled $0$ by $T$.\\ \\
Since $T$ is standard and the roots $G$, $H$, and $I$ are labeled $0$ by $T$, the labels on the $I-A$, $H-A$, and $G-A$ lines increase towards root $A$.  Thus, root $A$ necessarily has a higher label than any other root in $\Phi(w)$, except possibly root $B$, which does not lie on the $I-A$, $H-A$, or $G-A$ line.\\ \\
On the lines $G-E$, $H-A$, and $I-C$, and in the direction away from the points labeled $0$, root $F$ is the first root with a nonzero label.  Thus root $F$ must have a lower label than all roots lying on one of those three lines.  Since root $B$ is the only root not lying on those three lines, root $F$ necessarily has a lower label than any root in $\Phi(w)$, except possibly root $B$.\\ \\
The remaining roots of $\Phi(w)$ lie on the $B-E$ line, so that the remaining nonzero labels must lie on that line.  By Definition~\ref{d:standardsegment}, the lines of $\mathcal{W}$ are inversion sets of $\Phi^+$.
By Proposition~\ref{p:workhorse} and Lemma~\ref{l:wlabel}(3), root $E$ must have either the highest label or the lowest label on the $B-E$ line for $T$ to be a standard $\mathcal{W}$-labeling.  Furthermore, once this choice is made, the order of the remaining labels on the $B-E$ line are determined by Lemma~\ref{l:wlabel}(3).  Thus, there are eight cases that exhaust all possibilities, according to whether:
\begin{enumerate}[(1)]
\item  root $A$ is labeled $5$ or $6$;
\item  root $F$ is labeled $1$ or $2$;
\item  root $E$ has the highest label on the $B-E$ line or root $E$ has the lowest label on the $B-E$ line.
\end{enumerate}
\noindent
Choosing root $A$ to have label $6$ and root $F$ to have label $1$, we see that the roots on the $B-E$ line must be labeled $2$, $3$, $4$, and $5$.   The case where root $E$ is labeled $5$ gives the $\mathcal{W}$-labeling corresponding to the reduced expression $(u,s,t,s,t,u)$, while the case where root $E$ is labeled $2$ gives the $\mathcal{W}$-labeling corresponding to the reduced expression $(u,t,s,t,s,u)$.\\ \\
The case of root $A$ being labeled $5$ and root $F$ being labeled $2$ is impossible by the following reasoning:\\ \\
if root $E$ is labeled $1$, then $T(G) = 0$, $T(F) = 2$, and $T(E) = 1$, which contradicts the assumption that $T$ is standard since root $F$ is between root $G$ and root $E$ on the $G-E$ line.  If root $E$ is labeled $6$, then $T(I) = 0$, $T(E) = 6$, and $T(A) = 5$, which again contradicts the assumption that $T$ is standard since root $E$ is between root $A$ and root $I$ on the $A-I$ line.\\ \\
Now suppose that root $A$ is labeled $6$ and root $F$ is labeled $2$.  Then root $E$ must be labeled $1$ or $5$.  If root $E$ is labeled $1$, then $T(G) = 0$, $T(F) = 2$, and $T(E) = 1$, which contradicts the assumption that $T$ is standard.  However, the case where root $E$ is labeled $5$ gives the $\mathcal{W}$-labeling corresponding to the reduced expression $(s,u,t,s,t,u)$.\\ \\
If root $A$ is labeled $5$ and root $F$ is labeled $1$, then root $E$ must be labeled either $2$ or $6$.  If root $E$ is labeled $6$, we have $T(I) = 0$, $T(E) = 6$, and $T(A) = 5$, which contradicts the assumption that $T$ is standard.  However, the case where root $E$ is labeled $2$ gives the $\mathcal{W}$-labeling corresponding to the reduced expression $(u,t,s,t,u,s)$.
\end{ex}
\section{Restricting the standard segment structure to $\Phi(w)$}
The definitions and correspondences of the previous section concern labelings of $\Phi^+$ that have $\Phi(w)$ as support.  The conceptual advantage to this approach is that whenever an inversion set intersects $\Phi(w)$, but is not completely contained in $\Phi(w)$, our labeling scheme informs us what labeling inequalities must hold.  However, computationally this is not optimal if the number of roots in $\Phi^+ \setminus \Phi(w)$ is large.  For infinite Coxeter groups, there are always infinitely many such roots, so it is desirable to have an alternative to Corollary~\ref{c:reducedtheorem}.
\begin{defn}  \label{d:restrictsegment}
Let $\mathcal{I} = (P, \mathcal{L}, B)$ be a segment structure and let $A \subseteq P$.  For each $L \in \mathcal{L}$, we set $L_A = A \, \cap \, L$ and define $$\mathcal{L}_A = \{L_A \ : \ L \in \mathcal{L} \text{ and } |L_A| \geq  2\}.$$  Define a ternary relation $B_A$ by the following condition:  $B_A(p_1,p_2,p_3)$ holds if and only if $p_1,p_2,p_3 \in A$ and $B(p_1,p_2,p_3)$ holds.  We call the triple $\mathcal{I}_A = (A, \mathcal{L}_A, B_A)$ the \emph{restriction of $\mathcal{I}$ to $A$}.
\end{defn}
\begin{lem} \label{l:subsegment}
Let $\mathcal{I} = (P, \mathcal{L}, B)$ be a segment structure and let $A \subseteq P$.  Let $\mathcal{I}_A = (A, \mathcal{L}_A, B_A)$ be the restriction of $\mathcal{I}$ to $A$.  Then $\mathcal{I}_A$ is a segment structure.
\end{lem}
\begin{proof}
First note that the lines $L \in \mathcal{L}_A$ satisfy $|L| \geq 2$.\\ \\
If $B_A(p_1,p_2,p_3)$ holds, then $B(p_1,p_2,p_3)$ holds, so $p_1$, $p_2$, and $p_3$ are pairwise distinct.   Thus, $B_A$ satisfies property $B1$.\\ \\
If $B_A(p_1,p_2,p_3)$ holds then $B(p_1,p_2,p_3)$ implies there exists $L \in \mathcal{L}$ such that $p_1,p_2,p_3 \in L$.  Since $p_1,p_2,p_3 \in A$, this implies $p_1,p_2,p_3 \in L_A$ for some $L_A \in \mathcal{L}_A$, so $B_A$ satisfies property $B2$.\\ \\
If $p_1,p_2,p_3 \in L_A$ for some $L_A \in \mathcal{L}_A$, then by property $B3$ for $B$, there exists a permutation $\sigma$ such that $B(\sigma(p_1),\sigma(p_2),\sigma(p_3))$ holds.  Since $p_1,p_2,p_3 \in A$, this implies $B_A(\sigma(p_1),\sigma(p_2),\sigma(p_3))$ holds, so $B_A$ satisfies property $B3$.\\ \\
To show that $B_A$ satisfies property $B4$, let $\sigma:\{p_1,p_2,p_3\} \rightarrow \{p_1,p_2,p_3\}$ be a permutation.  If $B_A(p_1,p_2,p_3)$ holds, then $p_1,p_2,p_3 \in A$ and by definition $$B_A(\sigma(p_1),\sigma(p_2),\sigma(p_3))$$ holds if and only if $B(p_1,p_2,p_3)$ holds, proving that $B_A$ satisfies property $B4$.
\end{proof}
\begin{defn} \label{d:standardphisegment}
Let $(W,S)$ be a Coxeter system and let $$\mathcal{W} = (\Phi^+,\text{Inv}(\Phi^+),B_W)$$ be the standard segment structure of $W$.  For any $w \in W$, we call $\mathcal{W}_{\Phi(w)}$ the \emph{standard segment structure of $\Phi(w)$}.
\end{defn}
\begin{rem}
Although the definition for restriction to a subset $A$ allows infinite lines in the same way that Definition~\ref{d:segmentstructure} does, the primary choice for restriction in this thesis is $\Phi(w)$, which is a finite set.  Thus, the lines we see in this section are all finite.
\end{rem}
\begin{lem}  \label{l:betweenconerestrict}
Let $w \in W$ and let $\mathcal{W}_{\Phi(w)} = (\Phi(w), \text{Inv}(\Phi^+)_{\Phi(w)}, B_{\Phi(w)})$ be the standard segment structure of $\Phi(w)$.  Let $\lambda,\mu,\nu \in \Phi^+$ be distinct positive roots.  Then $B_{\Phi(w)}(\lambda, \mu, \nu)$ holds if and only if $\mu$ lies in the convex cone spanned by $\lambda$ and $\nu$ and $\lambda,\mu,\nu \in \Phi(w)$.
\end{lem}
\begin{proof}
By Lemma~\ref{l:subsegment} and Definition~\ref{d:standardphisegment}, $B_{\Phi(w)}(\lambda,\mu,\nu)$ holds if and only if $\lambda,\mu,\nu \in \Phi(w)$ and $B_W(\lambda,\mu,\nu)$ holds.  By Lemma~\ref{l:betweencone}, $B_W(\lambda,\mu,\nu)$ holds if and only if $\mu$ lies in the convex cone spanned by $\lambda$ and $\nu$.  The conclusion follows.
\end{proof}
\noindent
Given any standard $\mathcal{W}$-labeling $T$ of the standard segment structure for $W$, if we restrict the labeling to the set $\Phi(w)$, where $w \in W$, then the result is a standard $\mathcal{W}_{\Phi(w)}$-labeling.  However, the converse need not be true.  The problem is that if an inversion set $\Psi$ partially intersects $\Phi(w)$, then some roots of the inversion set must be labeled $0$.  This forces the roots in $\Psi \cap \Phi(w)$ to have labels increasing towards the one canonical simple root of $\Psi$ that is in $\Phi(w)$.  However, a standard $\mathcal{W}_{\Phi(w)}$-labeling could have labels increasing in the opposite direction.  Since we wish to produce a correspondence for our restricted segment structure $\mathcal{W}_{\Phi(w)}$, we must do some bookkeeping to account for this possibility.
\begin{defn}\label{d:fullpartial}
Let $w \in W$.  Let $\Psi$ be an inversion set.  If $\Psi \cap \Phi(w) = \Psi$ then we call $\Psi \cap \Phi(w)$ a \emph{full inversion set of $w$}. If $|\Psi \cap \Phi(w)| \geq 2$ and $\Psi \cap \Phi(w) \neq \Psi$, then we call $\Psi$ a \emph{partial inversion set of $w$}.
\end{defn}
\begin{rem}
Observe that for an infinite inversion set $\Psi$, if $|\Psi \cap \Phi(w)| \geq 2$, then $\Psi \cap \Phi(w)$ is automatically a partial inversion set of $w$, since $\Phi(w)$ is finite.
\end{rem}
\begin{ex} \label{ex:fullpartial}
Let $W = W(A_3)$, let $S = \{s,t,u\}$, and let $w = \phi(t,s,t,u)$.  Then $$\Phi(w) = \{\alpha_s,\alpha_s + \alpha_t + \alpha_u, \alpha_t + \alpha_u, \alpha_u\}.$$  Recall, by Definition~\ref{d:inversionmset}, that an inversion $2$-set is an inversion set $\Psi$ such that $|\Psi| = 2$.  Observe that (by Definition~\ref{d:fullpartial}) an inversion $2$-set can not be a partial inversion set and that a partial inversion set can not be a full inversion set.  The sets $$\Psi_1 = \{\alpha_s, \alpha_s + \alpha_t + \alpha_u, \alpha_t + \alpha_u\},$$ $$\Psi_2 = \{\alpha_t, \alpha_t + \alpha_u, \alpha_u\}, $$ $$\Psi_3 = \{\alpha_s,\alpha_s + \alpha_t, \alpha_t\}, \text{ and }$$ $$\Psi_4 = \{\alpha_s + \alpha_t, \alpha_s + \alpha_t + \alpha_u, \alpha_u\}$$ are all the $3$-inversion sets in the positive root system $\Phi^+$ of $W(A_3)$.  Since $\Psi_1 \subseteq \Phi(w)$, the set $\Psi_1 \cap \Phi(w) = \Psi_1$ is a full inversion set of $w$.  Since $\Psi_2 \cap \Phi(w) = \{\alpha_t + \alpha_u, \alpha_u\} \neq \Psi_2$, we see that $\Psi_2 \cap \Phi(w)$ is a partial inversion set of $w$.  However, note that $\Psi_3 \cap \Phi(w) = \{\alpha_s\}$ is neither a partial inversion set of $w$ nor a full inversion set of $w$.  (This choice in terminology is made because $\Psi_3 \cap \Phi(w)$ fails to form a line in the restriction of $\mathcal{W}$ to $\Phi(w)$.)  Lastly, note that $\Psi_4 \cap \Phi(w)$ is a partial inversion set.
\end{ex}
\begin{lem} \label{l:thatlemma}
Let $w \in W$ and let $\Psi$ be an inversion set such that $\Psi \cap \Phi(w)$ is a partial inversion set of $w$.  Then there exists exactly one canonical simple root $\gamma$ of $\Psi$ such that $\gamma \in \Psi \cap \Phi(w)$.
\end{lem}
\begin{proof}
By Definition~\ref{d:fullpartial}, $\Psi \cap \Phi(w) \neq \emptyset$.  By Corollary~\ref{c:biconvexstructure}, there exists a canonical simple root $\gamma \in \Psi \cap \Phi(w)$.  Suppose towards a contradiction that both canonical simple roots are in $\Psi \cap \Phi(w)$.  By Lemma~\ref{l:inversionbiconvex}, $\Phi(w)$ is biconvex, so it follows that $\Psi \cap \Phi(w) = \Psi$, which contradicts the hypothesis that $\Psi \cap \Phi(w)$ is a partial inversion set.
\end{proof}
\begin{rem}
We will always use $\gamma$ as a name for the unique canonical simple root of a partial inversion set.
\end{rem}
\begin{cor} \label{c:thatcorollary}
Let $w \in W$ and let $\Psi \cap \Phi(w)$ be a partial inversion set of $w$.  Then there exist canonical simple roots $\gamma$ and $\delta$ for $\Psi$ such that $\Psi \cap \Phi(w) = [\gamma_1,\gamma_k]$ in the total order $(\Psi, <_{\Psi,\gamma})$ and such that $k < |s_\gamma s_\delta|$.
\end{cor}
\begin{proof}
By Lemma~\ref{l:thatlemma}, there exists one canonical simple root $\gamma$ of $\Psi$ such that $\gamma \in \Psi \cap \Phi(w)$.  Let $\delta$ be the canonical simple root of $\Psi$ that is not in $\Psi \cap \Phi(w)$.  Then, by Corollary~\ref{c:biconvexstructure}, we have $\Psi \cap \Phi(w) = [\gamma_1,\gamma_k]$ in the total order $(\Psi,<_{\Psi,\gamma})$ for some $k$ because $\gamma \in \Psi \cap \Phi(w)$.  If $|s_\gamma s_\delta|$ is infinite then $k < |s_\gamma s_\delta|$ since $k$ is finite.  Otherwise, $\delta \not\in \Psi \cap \Phi(w)$ and Proposition~\ref{p:localdisjoint} implies that $k < |s_\gamma s_\delta|$.
\end{proof}
\begin{defn}
Let $w \in W$.  Let $\Psi$ be an inversion set and let $\Psi \cap \Phi(w)$ be a partial inversion set.  We call the unique canonical simple root $\gamma$ of $\Psi$ that satisfies $\gamma \in \Psi \cap \Phi(w)$ \emph{the canonical simple root of $\Psi \cap \Phi(w)$}.
\end{defn}
\begin{ex}
In Example~\ref{ex:fullpartial}, the partial inversion set $$\Psi_2 \cap \Phi(w) = \{\alpha_t + \alpha_u, \alpha_u\}$$ has $\alpha_u$ as its canonical simple root.  The inversion set $\Psi_2$ has $\gamma = \alpha_u$ and $\delta = \alpha_t$ as its canonical simple roots, but $\alpha_t \not\in \Psi_2 \cap \Phi(w)$.  Note that $$\Psi \cap \Phi(w) = [\gamma_1,\gamma_2],$$ which is consistent with the assertions of Corollary~\ref{c:thatcorollary}.
\end{ex}
\begin{defn} \label{d:partialinversionlabel}
Let $w \in W$ and let $\mathcal{W}_{\Phi(w)}$ be the standard segment structure of $\Phi(w)$.  Let $\Psi \cap \Phi(w)$ be a partial inversion set of $w$, let $\gamma = \gamma_1$ be the canonical simple root of $\Psi \cap \Phi(w)$ and, as in Corollary~\ref{c:thatcorollary}, write $\Psi \cap \Phi(w) = [\gamma_1, \gamma_k]$.  We say that a $\mathcal{W}_{\Phi(w)}$-labeling $T:\Phi(w) \rightarrow \mathbb{N}_0$ \emph{satisfies the restrictions of $\Psi \cap \Phi(w)$} if $T(\gamma_1) > \cdots > T(\gamma_k)$.  We say that a $\mathcal{W}_{\Phi(w)}$-labeling $T$ \emph{satisfies the restrictions of $\Phi(w)$} if $T$ satisfies the restrictions of every partial inversion set $\Psi \cap \Phi(w)$ of $w$.  We denote the set of all sequential standard $\mathcal{W}_{\Phi(w)}$-labelings with $\text{supp}(T) = \Phi(w)$ that satisfy the restrictions of $\Phi(w)$ by $\mathrm{Lab}(\mathcal{W}_{\Phi(w)},w)$.
\end{defn}
\begin{ex}
Again we let $W = W(A_3)$ and let $w = \phi(t,s,t,u)$.  Define $T:\Phi(w) \rightarrow \mathbb{N}_0$ by $T(\alpha_s) = 1$, $T(\alpha_s + \alpha_t + \alpha_u) = 2$, $T(\alpha_t + \alpha_u) = 3$, and $T(\alpha_u) = 4$.  From Example~\ref{ex:fullpartial}, we have that $\Psi_2 \cap \Phi(w) = \{\alpha_t + \alpha_u, \alpha_u\}$ and $\Psi_4 \cap \Phi(w) = \{\alpha_s + \alpha_t + \alpha_u, \alpha_u\}$ are the only two partial inversion sets of $w$.  The canonical simple root of $\Psi_4 \cap \Phi(w)$ is $\alpha_u$, and the canonical simple root of $\Psi_2 \cap \Phi(w)$ is also $\alpha_u$.  We note that $T(\alpha_u) > T(\alpha_t+\alpha_u)$ and $T(\alpha_u) > T(\alpha_s + \alpha_t + \alpha_u)$ so that $T$ satisfies the restrictions of $\Phi(w)$.\\ \\
If instead we chose the labeling $T':\Phi(w) \rightarrow \mathbb{N}_0$ defined by $T'(\alpha_s) = 1$, $T'(\alpha_s + \alpha_t + \alpha_u) = 2$, $T'(\alpha_u) = 3$, and $T'(\alpha_t + \alpha_u) = 4$, then $T'$ satisfies the restrictions of $\Psi_4 \cap \Phi(w)$, but not those of $\Psi_2 \cap \Phi(w)$:  this is because $T'(\alpha_u) < T'(\alpha_t + \alpha_u)$.  As $T'$ does not satisfy the restrictions of all the partial inversion sets of $w$, $T'$ does not satisfy the restrictions of $\Phi(w)$.  Note that $T'$ is, however, a sequential standard $\mathcal{W}_{\Phi(w)}$-labeling.
\end{ex}
\noindent
The restrictions made on $\mathcal{W}_{\Phi(w)}$-labelings are justified by the next lemma.
\begin{lem} \label{l:uniquewlabel}
Let $T:\Phi(w) \rightarrow \mathbb{N}_0$ be a sequential standard $\mathcal{W}_{\Phi(w)}$-labeling that satisfies the restrictions of $\Phi(w)$ and has $\text{supp}(T) = \Phi(w)$.  Then there exists a unique sequential standard $\mathcal{W}$-labeling $T':\Phi^+ \rightarrow \mathbb{N}_0$ such that $\text{supp}(T') = \Phi(w)$ and $T'|_{\Phi(w)} = T$.
\end{lem}
\begin{proof}
To construct $T'$, we set $T'(\alpha) = 0$ if $\alpha \not\in \Phi(w)$; otherwise we set $T'(\alpha) = T(\alpha)$.
We have $\text{supp}(T') = \text{supp}(T) = \Phi(w)$.  Thus, $T'$ is sequential since $T$ is sequential.\\ \\
To show that $T'$ is a standard $\mathcal{W}$-labeling, we let $\lambda, \mu, \nu$ be distinct positive roots such that $\mu$ strictly lies in the convex cone spanned by $\lambda$ and $\nu$.  By Corollary~\ref{c:inversionline}, there is a unique inversion set $\Psi$ containing $\lambda$ and $\nu$ (and hence $\mu$).  If $\Psi \cap \Phi(w)$ is a partial inversion set we apply Lemma~\ref{l:thatlemma} and let $\gamma$ be the canonical simple root of $\Psi \cap \Phi(w)$.  Assume without loss of generality that $\lambda <_{\Psi,\gamma} \mu <_{\Psi,\gamma} \nu$ (the reverse case being a symmetric argument and the other cases being eliminated by Lemma~\ref{l:betweencone}).  Recall that by Definition~\ref{d:betweenness} this means that $B_W(\lambda,\mu,\nu)$ holds.  We consider the following cases:
\begin{enumerate}[(1)]
\item $\lambda,\mu,\nu \in \Phi(w)$;
\item $\lambda,\mu \in \Phi(w)$, but $\nu \not\in \Phi(w)$;
\item $\lambda \in \Phi(w)$ but $\mu,\nu \not\in \Phi(w)$;
\item $\lambda,\mu,\nu \not\in \Phi(w)$.
\end{enumerate}
\noindent
The assumed ordering of $\lambda$, $\mu$, and $\nu$ with respect to $<_{\Psi,\gamma}$ and Corollary~\ref{c:biconvexstructure} ensures that these are the only possibilities.\\ \\
For case $(1)$, Lemma~\ref{l:betweenconerestrict} implies that $B_{\Phi(w)}(\lambda,\mu,\nu)$ holds.  Since $T$ is assumed to be a standard $\mathcal{W}_{\Phi(w)}$-labeling and $T = T'$ on $\Phi(w)$, we have either $T'(\lambda) \leq T'(\mu) \leq T'(\nu)$ or $T'(\nu) \leq T'(\mu) \leq T'(\lambda)$.\\ \\
For case $(2)$, we have that $\Psi \cap \Phi(w)$ is a partial inversion set.  By Definition~\ref{d:partialinversionlabel}, we have $T(\lambda) > T(\mu)$ since $T$ satisfies the restrictions of $\Psi \cap \Phi(w)$.  Since $\lambda,\mu \in \Phi(w)$, this means $T'(\lambda) > T'(\mu)$.  We have $T'(\mu) > T'(\nu)$ since $T'(\nu) = 0$.\\ \\
For case $(3)$, we have that $T'(\lambda) > 0$ since $\lambda \in \Phi(w)$.  Thus $$T'(\lambda) > T'(\mu) \geq T'(\nu).$$
Similarly, for case $(4)$ we have $T'(\lambda) \geq T'(\mu) \geq T'(\nu)$.  Having exhausted all cases, we see that $T'$ is a standard $\mathcal{W}$-labeling by Definition~\ref{d:incidencelabel}.\\ \\
Now suppose there exists a sequential standard $\mathcal{W}$-labeling $T'':\Phi^+ \rightarrow \mathbb{N}_0$ such that $\text{supp}(T'') = \Phi(w)$ and $T''|_{\Phi(w)} = T$.  Then, if $\alpha \in \Phi(w)$, we have $T''(\alpha) = T(\alpha) = T'(\alpha)$.  If $\alpha \not\in \Phi(w)$, then since $\text{supp}(T'') = \Phi(w)$, we have $T''(\alpha) = 0 = T'(\alpha)$.  Thus $T'' = T'$ so that $T'$ is the unique $\mathcal{W}$-labeling satisfying these properties.
\end{proof}
\begin{lem} \label{l:inversionrestriction}
Let $T:\Phi^+ \rightarrow \mathbb{N}_0$ be a sequential standard $\mathcal{W}$-labeling such that $\text{supp}(T) = \Phi(w)$.  Then $T|_{\Phi(w)}$ is a sequential standard $\mathcal{W}_{\Phi(w)}$-labeling that satisfies the restrictions of $\Phi(w)$.  Furthermore, $(T|_{\Phi(w)})' = T$, where $'$ is the extension of the domain $\Phi(w)$ to $\Phi^+$ introduced in Lemma~\ref{l:uniquewlabel}.
\end{lem}
\begin{proof}
Since $\text{supp}(T) = \Phi(w)$, we have $\text{supp}(T|_{\Phi(w)}) = \Phi(w)$.  Thus $T|_{\Phi(w)}$ is sequential because $T$ is sequential.\\ \\
For every $\lambda,\mu,\nu \in \Phi(w)$ such that $B_{\Phi(w)}(\lambda,\mu,\nu)$ holds, $B_W(\lambda,\mu,\nu)$ must also hold.  Thus, if $B_{\Phi(w)}(\lambda,\mu,\nu)$ holds, then either $$T|_{\Phi(w)}(\lambda) \leq T|_{\Phi(w)}(\mu) \leq T|_{\Phi(w)}(\nu)$$
\begin{center}
or
\end{center}
$$T|_{\Phi(w)}(\nu) \leq T|_{\Phi(w)}(\mu) \leq T|_{\Phi(w)}(\lambda).$$  It then follows that $T|_{\Phi(w)}$ is a standard $\mathcal{W}_{\Phi(w)}$-labeling.\\ \\
Now let $\Psi \cap \Phi(w)$ be a partial inversion set of $w$.  By Corollary~\ref{c:thatcorollary}, there exist canonical simple roots $\gamma$ and $\delta$ of $\Psi$ such that $\gamma \in \Psi$, $\delta \not\in \Psi$ and such that we may write $\Psi \cap \Phi(w) = [\gamma_1,\gamma_k]$ for some $k < |s_\gamma s_\delta|$.  By Lemma~\ref{l:wlabel}(2), we have $$T(\gamma_1) > \cdots > T(\gamma_k) > 0$$
\begin{center}
so that
\end{center}
$$T|_{\Phi(w)}(\gamma_1) > \cdots > T|_{\Phi(w)}(\gamma_k).$$  By Definition~\ref{d:partialinversionlabel}, $T|_{\Phi(w)}$ satisfies the restrictions imposed by $\Psi \cap \Phi(w)$.\\ \\
Lastly, $(T|_{\Phi(w)})' = T$ follows from the uniqueness assertion of Lemma~\ref{l:uniquewlabel}.
\end{proof}
\begin{cor} \label{c:wlabelcorrespondence}
Let $\mathrm{Lab}(\mathcal{W},w)$ be the set of all sequential standard $\mathcal{W}$-labelings $T$ such that $\text{supp}(T) = \Phi(w)$.  Let $\mathrm{Lab}(\mathcal{W}_{\Phi(w)},w)$ be the set of all sequential standard $\mathcal{W}_{\Phi(w)}$-labelings that satisfy the restrictions of $\Phi(w)$.  Then there is a 1--1 correspondence between $\mathrm{Lab}(\mathcal{W},w)$ and $\mathrm{Lab}(\mathcal{W}_{\Phi(w)},w)$ given by
\begin{equation*}
\begin{split}
T &\longrightarrow T|_{\Phi(w)} \text{ and }\\
T'&\longleftarrow T.
\end{split}
\end{equation*}
\end{cor}
\begin{proof}
This follows from Lemmas~\ref{l:uniquewlabel} and ~\ref{l:inversionrestriction}.
\end{proof}
\begin{cor} \label{c:labelinversion}
Let $(W,S)$ be an arbitrary Coxeter system and let $w \in W$.  Then the set $\mathcal{R}(w)$ of reduced expressions for $w$ is in $1-1$ correspondence with the set $\mathrm{Lab}(\mathcal{W}_{\Phi(w)},w)$ of all sequential standard $\mathcal{W}_{\Phi(w)}$-labelings that satisfy the restrictions of $\Phi(w)$.
\end{cor}
\begin{proof}
This correspondence can be obtained by composing the correspondence given in Corollary~\ref{c:wlabelcorrespondence} with the one given in Corollary~\ref{c:reducedtheorem}.
\end{proof}
\begin{ex} \label{ex:inversionex}
For purposes of comparison, we demonstrate the correspondence of Corollary~\ref{c:labelinversion} using the same Coxeter group element as in Example~\ref{ex:mainexample}, namely $W = W(B_3)$ and $w = \phi(u,s,t,s,t,u)$.  For this example, the vertices are only the roots of $\Phi(w)$ instead of all the roots in $\Phi^+$, and we pictorially represent the inversion sets with two roots.  We represent the partial inversion sets with an arrow pointing in the direction that the labels must go.  The first figure below depicts the standard segment structure of $\Phi(w)$ and the chart to its right gives names we will use for the roots in $\Phi(w)$.  Thus, the full inversion sets of $w$ are $\{A,B\}$, $\{B,F\}$, and $\{B,C,D,E\}$.  The partial inversion sets of $w$ are $\{A,E\}$, $\{A,D,F\}$, $\{A,C\}$, $\{C,F\}$, and $\{E,F\}$.\\ \\
The four figures following the first figure give the sequential standard $\mathcal{W}_{\Phi(w)}$-labelings that satisfy the restrictions of $\Phi(w)$, which we discuss after listing the figures.\\ \\
\begin{tikzpicture}[scale=2.5] \label{Bsequentialfig}
\filldraw           (3,0) circle (0.02)
                    (0.75,0.75) circle (0.02)
                    (0,0) circle (0.02)
                    (1.67,0.44) circle (0.02)
                    (1.33,.56) circle (0.02)
                    (1.63,0.68) circle (0.02);
\draw (0,0) -- (0.75,0.75);
\draw (0.75,0.75) -- (3,0);
\draw (0,0) -- (1.63,0.68);
\draw (1.63,0.68) -- (1.67,0.44);
\draw (0.75,0.75) -- (1.63,0.68);
\draw (0,0) -- (1.63,0.68);
\draw (0,0) -- (3,0);
\draw (0,0) -- (1.67,0.44);
\draw (1.63,0.68) -- (3,0);
\draw [->] (0.87,0.87) -- (0.79,0.79);
\draw [->] (1.84,0.77) -- (1.7,0.71);
\draw [->] (1.86,0.49) -- (1.73,0.46);
\draw [->] (1.6,0.86) -- (1.62,0.74);
\draw [->] (1.91,0.65) -- (1.75,0.67);
\draw (-.1,-.05) node {$A$};
\draw (3.05,-.1) node {$B$};
\draw (1.67,0.32) node {$C$};
\draw (1.35,0.44) node {$D$};
\draw (0.63,0.76) node {$E$};
\draw (1.69,0.79) node {$F$};
\node [right=1cm,text width=8cm] at (2.45,1.2)
{
\begin{tabular}{cl} Name & Root\\A & $\alpha_u$\\B & $\alpha_s$\\
C & $\sqrt{2}\alpha_s + \alpha_t + \alpha_u$\\D & $\alpha_s + \sqrt{2} \alpha_t + \sqrt{2}\alpha_u$\\E & $\alpha_t + \alpha_u$\\F & $\sqrt{2}\alpha_s + 2\alpha_t + \alpha_u$
\end{tabular}
};

\end{tikzpicture}
\\
The labelings corresponding to the reduced expressions $(u,s,t,s,t,u)$ and $(s,u,t,s,t,u)$, respectively are:\\ \\
\begin{tikzpicture}[scale=1.8] \label{Bexample3fig}
\filldraw           (3,0) circle (0.02)
                    (0.75,0.75) circle (0.02)
                    (0,0) circle (0.02)
                    (1.67,0.44) circle (0.02)
                    (1.33,.56) circle (0.02)
                    (1.63,0.68) circle (0.02);
\draw (0,0) -- (0.75,0.75);
\draw (0.75,0.75) -- (3,0);
\draw (0,0) -- (1.63,0.68);
\draw (1.63,0.68) -- (1.67,0.44);
\draw (0.75,0.75) -- (1.63,0.68);
\draw (0,0) -- (1.63,0.68);
\draw (0,0) -- (3,0);
\draw (0,0) -- (1.67,0.44);
\draw (1.63,0.68) -- (3,0);
\draw [->] (0.87,0.87) -- (0.79,0.79);
\draw [->] (1.84,0.77) -- (1.7,0.71);
\draw [->] (1.86,0.49) -- (1.73,0.46);
\draw [->] (1.6,0.86) -- (1.62,0.74);
\draw [->] (1.91,0.65) -- (1.75,0.67);
\draw (-.1,-.05) node {$6$};
\draw (3.05,-.1) node {$2$};
\draw (1.67,0.32) node {$3$};
\draw (1.35,0.44) node {$4$};
\draw (0.63,0.76) node {$5$};
\draw (1.69,0.79) node {$1$};

\filldraw           (6.5,0) circle (0.02)
                    (4.25,0.75) circle (0.02)
                    (3.5,0) circle (0.02)
                    (5.17,0.44) circle (0.02)
                    (4.83,.56) circle (0.02)
                    (5.13,0.68) circle (0.02);
\draw (3.5,0) -- (4.25,0.75);
\draw (4.25,0.75) -- (6.5,0);
\draw (3.5,0) -- (5.13,0.68);
\draw (5.13,0.68) -- (5.17,0.44);
\draw (4.25,0.75) -- (5.13,0.68);
\draw (3.5,0) -- (5.13,0.68);
\draw (3.5,0) -- (6.5,0);
\draw (3.5,0) -- (5.17,0.44);
\draw (5.13,0.68) -- (6.5,0);
\draw [->] (4.37,0.87) -- (4.29,0.79);
\draw [->] (5.34,0.77) -- (5.2,0.71);
\draw [->] (5.36,0.49) -- (5.23,0.46);
\draw [->] (5.1,0.86) -- (5.12,0.74);
\draw [->] (5.41,0.65) -- (5.25,0.67);
\draw (3.4,-.05) node {$6$};
\draw (6.65,-.1) node {$1$};
\draw (5.17,0.32) node {$3$};
\draw (4.85,0.44) node {$4$};
\draw (4.13,0.76) node {$5$};
\draw (5.19,0.79) node {$2$};
\end{tikzpicture}
\\
The labelings corresponding to the reduced expressions $(u,t,s,t,s,u)$ and $(u,t,s,t,u,s)$, respectively are:\\ \\
\begin{tikzpicture}[scale=1.8] \label{Bexample4fig}
\filldraw           (3,0) circle (0.02)
                    (0.75,0.75) circle (0.02)
                    (0,0) circle (0.02)
                    (1.67,0.44) circle (0.02)
                    (1.33,.56) circle (0.02)
                    (1.63,0.68) circle (0.02);
\draw (0,0) -- (0.75,0.75);
\draw (0.75,0.75) -- (3,0);
\draw (0,0) -- (1.63,0.68);
\draw (1.63,0.68) -- (1.67,0.44);
\draw (0.75,0.75) -- (1.63,0.68);
\draw (0,0) -- (1.63,0.68);
\draw (0,0) -- (3,0);
\draw (0,0) -- (1.67,0.44);
\draw (1.63,0.68) -- (3,0);
\draw [->] (0.87,0.87) -- (0.79,0.79);
\draw [->] (1.84,0.77) -- (1.7,0.71);
\draw [->] (1.86,0.49) -- (1.73,0.46);
\draw [->] (1.6,0.86) -- (1.62,0.74);
\draw [->] (1.91,0.65) -- (1.75,0.67);
\draw (-.1,-.05) node {$6$};
\draw (3.05,-.1) node {$5$};
\draw (1.67,0.32) node {$4$};
\draw (1.35,0.44) node {$3$};
\draw (0.63,0.76) node {$2$};
\draw (1.69,0.79) node {$1$};

\filldraw           (6.5,0) circle (0.02)
                    (4.25,0.75) circle (0.02)
                    (3.5,0) circle (0.02)
                    (5.17,0.44) circle (0.02)
                    (4.83,.56) circle (0.02)
                    (5.13,0.68) circle (0.02);
\draw (3.5,0) -- (4.25,0.75);
\draw (4.25,0.75) -- (6.5,0);
\draw (3.5,0) -- (5.13,0.68);
\draw (5.13,0.68) -- (5.17,0.44);
\draw (4.25,0.75) -- (5.13,0.68);
\draw (3.5,0) -- (5.13,0.68);
\draw (3.5,0) -- (6.5,0);
\draw (3.5,0) -- (5.17,0.44);
\draw (5.13,0.68) -- (6.5,0);
\draw [->] (4.37,0.87) -- (4.29,0.79);
\draw [->] (5.34,0.77) -- (5.2,0.71);
\draw [->] (5.36,0.49) -- (5.23,0.46);
\draw [->] (5.1,0.86) -- (5.12,0.74);
\draw [->] (5.41,0.65) -- (5.25,0.67);
\draw (3.4,-.05) node {$5$};
\draw (6.65,-.1) node {$6$};
\draw (5.17,0.32) node {$4$};
\draw (4.85,0.44) node {$3$};
\draw (4.13,0.76) node {$2$};
\draw (5.19,0.79) node {$1$};
\end{tikzpicture}
\\
One can check that the four figures above are sequential standard $\mathcal{W}_{\Phi(w)}$-labelings satisfying the restrictions of $\Phi(w)$.  Conversely, if we follow the arrows, we see many of the features we pointed out in Example~\ref{ex:mainexample}.  Root $A$ must have a greater label than all other roots except possibly root $B$, so it must be labeled $5$ or $6$.  The arrows also give that root $F$ must have a smaller label than all other roots except possibly root $B$.  As in Example~\ref{ex:mainexample}, the remaining labels occur on the $B-E$ line.  Thus, we can again break our analysis into the eight cases according to whether root $A$ is labeled $5$ or $6$, whether root $F$ is labeled $1$ or $2$, and whether root $E$ has the highest or lowest remaining label on the $B-E$ line.\\ \\
The main difference in the reasoning is that wherever we contradicted the assumption that $T$ is standard in Example~\ref{ex:mainexample}, we now contradict the assumption that $T$ satisfies the restrictions of a particular partial inversion set.  For example, in Example~\ref{ex:mainexample}, we derived the following contradiction: \\ \\
``Now suppose that root $A$ is labeled $6$ and root $F$ is labeled $2$.  Then root $E$ must be labeled $1$ or $5$.  If root $E$ is labeled $1$, then $T(G) = 0$, $T(F) = 2$, and $T(E) = 1$ contradicts the assumption that $T$ is standard.'' \\ \\
The corresponding contradiction here is as follows:  Suppose that root $A$ is labeled $6$ and root $F$ is labeled $2$. Then root $E$ must be labeled $1$ or $5$.  If root $E$ is labeled $1$, then we contradict the assumption that $T$ satisfies the restrictions of the partial inversion set $\{E,F\}$, which requires that $T(F) < T(E)$ since root $E$ is the canonical simple root of $\{E,F\}$.\\ \\
Note that the above restriction on labelings for the partial inversion set $\{E,F\}$ is caused by the absence of root $G$ from $\Phi(w)$.  The other contradictions derived in Example~\ref{ex:mainexample} are translated similarly to the setting of $\mathcal{W}_{\Phi(w)}$-labelings.  One can also check that the four labelings given for this example are the same as those of Example~\ref{ex:mainexample}, except that the labelings are restricted to the set $\Phi(w)$.
\end{ex}
\section{Labelings and Directed Acyclic Graphs}
We establish a connection between the labelings we have obtained, and directed acyclic graphs with restrictions imposed upon them.  For this purpose, we need the basic terminology and some basic results from graph theory.
\begin{defn}
We call the pair $G = (V,E)$ a \emph{simple directed graph} if $V$ is a set and $E \subseteq V \times V$ such that for each $v \in V$, $(v,v) \not\in E$.  The elements of $V$ are called \emph{vertices} and the elements of $E$ are called \emph{directed edges}.  We sometimes denote the assertion that $(u,v) \in E$ by $u \rightarrow v$.  An \emph{oriented graph} is a simple directed graph such that for any $u,v \in V$, there is at most one directed edge connecting $u$ and $v$ (i.e. $u \rightarrow v$ or $v \rightarrow u$ or neither).
\end{defn}
\begin{rem}
Every oriented graph is a simple directed graph.  However, a simple directed graph may have both $u \rightarrow v$ and $v \rightarrow u$.  The directed graphs that arise in this thesis are all simple and directed, so we do not bother with the full spectrum and generality of graph theoretic objects.  It should also be remarked that all of the directed graphs we use have a finite vertex set $V$ (and hence a finite edge set $E$).
\end{rem}
\begin{defn}
Let $G = (V,E)$ be a directed graph.  If $v_1,\ldots,v_n \in V$, then a sequence of edges satisfying $v_1 \rightarrow v_2 \cdots \rightarrow v_n \rightarrow v_1$ is called a \emph{directed cycle}.  If $G$ contains no directed cycles, then we call $G$ a \emph{directed acyclic graph}.
\end{defn}
\begin{rem}
A directed acyclic graph is necessarily oriented since $u \rightarrow v$ and $v \rightarrow u$ gives a cycle.
\end{rem}
\begin{lem}
Let $G = (V,E)$ be a directed acyclic graph.  Then there exists a partial order $\leq$ on $V$ such that for $x,y \in V$, $x \leq y$ whenever $x \rightarrow y$.
\end{lem}
\begin{proof}
For any edge $x \rightarrow y$ we require that $x \leq y$.  Forming the reflexive and transitive closure of the required relations makes $\leq$ a partial order.  See \cite[Section 3.2.3]{handbookgraphtheory} for details.
\end{proof}
\begin{defn}
A \emph{topological sort} of a directed acyclic graph $G = (V,E)$ is a linear ordering of $V$ such that if $G$ contains an edge $u \rightarrow v$, then $u$ appears before $v$ in the linear ordering.  A \emph{linear extension} of a partially ordered set $(X,\leq)$ is a linear ordering $\leq'$ of $X$ such that $x \leq y$ implies $x \leq' y$.
\end{defn}
\begin{lem}
There exists a topological sort of any finite directed acyclic graph.
\end{lem}
\begin{proof}
See Fact 23 of \cite[Section 3.2.4]{handbookgraphtheory}.
\end{proof}
\noindent
As discussed in \cite[Section 3.2.3]{handbookgraphtheory}, there is a close connection between directed acyclic graphs and partially ordered sets.  From any directed acyclic graph we can produce a partially ordered set compatible with it, and associated to a partially ordered set is at least one directed acyclic graph (typically we represent partially ordered sets with ``Hasse diagrams'', which can be interpreted as a directed acyclic graph).  Thus it is not surprising that we can also give a linear extension (via  a topological sort) of a finite partially ordered set.
\begin{lem}
Every finite partially ordered set $(X,\leq)$ has a linear extension.
\end{lem}
\begin{proof}
See \cite[Algorithm 12.2.2]{graphtheoryapps}.
\end{proof}
\begin{defn}
We call a directed graph $G = (V,E)$ a \emph{tournament on $V$} if for every pair of distinct vertices $x,y \in V$, we have either $x \rightarrow y$ or $y \rightarrow x$ (but not both).  A tournament is \emph{transitive} if for all $x,y,z \in V$, $(x \rightarrow y$ and $y \rightarrow z$) implies $x \rightarrow z$.
\end{defn}
\begin{rem}
According to \cite[Section 3.3]{handbookgraphtheory}, tournaments arise as models of round-robin tournaments, which are competitions in which each participant competes against every other participant exactly once.  Thus an edge $u \rightarrow v$ models the situation in which $u$ defeats $v$ in the tournament.  In a transitive tournament, there is a unique overall winner, unique overall second place, and so forth.  This is captured in the next lemma.
\end{rem}
\begin{lem} \label{l:uniquesort}
Let $G = (V,E)$ be a transitive tournament.  Then $G$ is a directed acyclic graph and there is a unique topological sort of $G$.
\end{lem}
\begin{proof}
See Fact 14 of \cite[Section 3.3.2]{handbookgraphtheory}.
\end{proof}
\begin{lem} \label{l:inversiontournament}
Let $T:\Phi(w) \rightarrow \mathbb{N}_0$ be a sequential $\mathcal{W}_{\Phi(w)}$-labeling such that $\text{supp}(T) = \Phi(w)$.  Let $G_T = (V,E)$ be the directed graph with vertex set $\Phi(w)$ and edge set $$E = \{(\alpha,\beta) : T(\alpha) < T(\beta)\}.$$  The directed graph $G_T$ is a transitive tournament on $\Phi(w)$.
\end{lem}
\begin{proof}
If $\alpha, \beta \in \Phi(w)$ are distinct positive roots, then by Definition~\ref{d:incidencelabel}, either $T(\alpha) < T(\beta)$ or $T(\beta) < T(\alpha)$.  Thus, either $\alpha \rightarrow \beta$ or $\beta \rightarrow \alpha$.
\end{proof}
\begin{defn} \label{d:labeltournament}
We call the directed graph $G_T$ in Lemma~\ref{l:inversiontournament} the \emph{transitive tournament on $\Phi(w)$ determined by $T$}.
\end{defn}
\noindent
It will be convenient to regard the labelings of $\Phi(w)$ in Corollary~\ref{c:labelinversion} as inducing linear orders on $\Psi \cap \Phi(w)$ for any inversion set $\Psi$.
\begin{defn}
Let $(W,S)$ be a Coxeter system and let $w \in W$.  If $$\Psi \cap \Phi(w) = \{\gamma_1,\ldots,\gamma_k\}$$ is a partial inversion set of $w$, then we call the linear ordering $(\Psi \cap \Phi(w), \leq)$ given by $\gamma_1 > \cdots > \gamma_k$ the \emph{induced ordering on $\Psi \cap \Phi(w)$}.  Suppose $$\Psi = \{\gamma_1,\ldots,\gamma_m\}$$ is a full inversion set.  Then we call the two linear orderings $\leq_1$ and $\leq_2$ on $\Psi$ given by $\gamma_1 <_1 \cdots <_1 \gamma_m$ and $\gamma_m <_2 \cdots <_2 \gamma_1$ the \emph{dual orderings of $\Psi$}.
\end{defn}
\noindent
To phrase our correspondence in the terminology of tournaments, we need to account for the variety of restrictions that our correspondence from the previous section imposes.
\begin{defn}
Let $G = (V,E)$ be a transitive tournament.  Suppose $U \subseteq V$ and $(U,\leq)$ is a linearly ordered set.  We say that \emph{the tournament $G$ is consistent with $(U,\leq)$} if for any $u,u' \in U$, $u < u'$ implies $(u,u') \in E$.
\end{defn}
\begin{ex}
Let $V = \{A,B,C,D,E\}$ and let $G = (V,E)$ be the tournament having the edge set depicted in the graph below.\\ \\

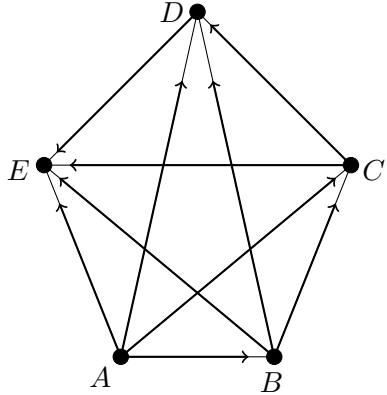
\begin{tikzpicture}[scale=3.4] \label{consistentfig}
\node [right=1cm,text width=8cm] at (0,0.6) { };
\filldraw           (2.2,0.6) circle (0.03)
                    (2.8,0.6) circle (0.03)
                    (3.1,1.35) circle (0.03)
                    (2.5,1.95) circle (0.03)
                    (1.9,1.35) circle (0.03);
\draw (2.2,0.6) -- (2.8,0.6);
\draw (2.8,0.6) -- (3.1,1.35);
\draw (3.1,1.35) -- (2.5,1.95);
\draw (2.5,1.95) -- (1.9,1.35);
\draw (1.9,1.35) -- (2.2,0.6);
\draw (2.2,0.6) -- (3.1,1.35);
\draw (3.1,1.35) -- (1.9,1.35);
\draw (1.9,1.35) -- (2.8,0.6);
\draw (2.8,0.6) -- (2.5,1.95);
\draw (2.5,1.95) -- (2.2,0.6);
\draw [->, thick] (2.2,0.6) -- (2.7,0.6);
\draw [->, thick] (3.1,1.35) -- (2.0,1.35);
\draw [->, thick] (2.8,0.6) -- (3.04,1.2);
\draw [->, thick] (3.1,1.35) -- (2.55,1.9);
\draw [->, thick] (2.5,1.95) -- (1.95,1.4);
\draw [->, thick] (2.2,0.6) -- (1.96,1.2);
\draw [->, thick] (2.2,0.6) -- (3.04,1.3);
\draw [->, thick] (2.8,0.6) -- (1.96,1.3);
\draw [->, thick] (2.2,0.6) -- (2.44,1.68);
\draw [->, thick] (2.8,0.6) -- (2.56,1.68);
\draw (2.12,0.52) node {$A$};
\draw (2.79,0.5) node {$B$};
\draw (3.19,1.33) node {$C$};
\draw (2.4,1.95) node {$D$};
\draw (1.8,1.33) node {$E$};
\end{tikzpicture}
\\
Let $U = \{A,B,C\}$ and $U' = \{C,D,E\}$.  Let $(U,\leq)$ be the linear order determined by the relations $A < B < C$ and let $(U',\leq')$ be the linear order determined by the relations $E <' D <' C$.  Then $G$ is consistent with $(U,\leq)$ since $A \rightarrow B \rightarrow C$ in $G$ and $A < B < C$.  However, $G$ is not consistent with $(U',\leq')$ since $C \rightarrow D \rightarrow E$ in $G$ and $E <' D <' C$.
\end{ex}
\begin{lem}
Let $G = (V,E)$ be a transitive tournament.  Suppose $(U, \leq)$ is a linearly ordered set and suppose that $G$ is consistent with $U$.  Let $(V,\leq')$ be the unique topological sort of $G$.   Then for any distinct $u,u' \in U$ we have $u < u'$ if and only if $u <' u'$.
\end{lem}
\begin{proof}
If $u < u'$, then since $G$ is consistent with $U$ we have $u \rightarrow u'$.  Thus, since $<'$ is given by the topological sort of $G$, we have $u <' u'$.\\ \\
Conversely, suppose $u <' u'$.  Since $G$ is a tournament, we have either $u \rightarrow u'$ or $u' \rightarrow u$, but not both.  As $<'$ is a topological sort of $G$, we must have $u \rightarrow u'$.  Now since $G$ is consistent with $U$, we have $u < u'$.
\end{proof}
\begin{defn}
Let $G = (V,E)$ be a transitive tournament with vertex set $V = \Phi(w)$ and let $\Psi$ be an inversion set.  Suppose that for any partial inversion set $\Psi \cap \Phi(w)$, the tournament $G$ is consistent with the induced ordering on $\Psi \cap \Phi(w)$.  Also suppose that for any full inversion set $\Psi$, the tournament $G$ is consistent with one of the two dual orders of $\Psi$.  Then we say that \emph{$G$ satisfies the restrictions imposed by $\Phi(w)$}.  The set of all transitive tournaments on $\Phi(w)$ that satisfy the restrictions imposed by $\Phi(w)$ is denoted by $\mathrm{Tour}(w)$.
\end{defn}
\begin{lem} \label{l:tourcorrespondence}
Let $(W,S)$ be an arbitrary Coxeter system and let $w \in W$.  Then the set $\mathrm{Lab}(\mathcal{W}_{\Phi(w)},w)$ of sequential standard labelings of $\mathcal{W}_{\Phi(w)}$ satisfying the restrictions of $\Phi(w)$ is in 1--1 correspondence with the set $\mathrm{Tour}(w)$ of transitive tournaments satisfying the restrictions imposed by $\Phi(w)$.
\end{lem}
\begin{proof}
Let $T$ be a sequential standard labeling of $\mathcal{W}_{\Phi(w)}$ satisfying the restrictions of $\Phi(w)$.  Let $\Psi \cap \Phi(w)$ be a partial inversion set.  By Lemma~\ref{l:thatlemma}, there exists a canonical simple root $\gamma \in \Psi \cap \Phi(w)$.  By Corollary~\ref{c:thatcorollary}, we have $\Psi \cap \Phi(w) = [\gamma_1,\gamma_k]$ in the total ordering $(\Psi, <_{\Psi,\gamma})$.  By Definition~\ref{d:partialinversionlabel}, $T(\gamma_1) > \cdots > T(\gamma_k)$.  In the transitive tournament $G_T$ determined by the labeling $T$, we have $\gamma_i \rightarrow \gamma_j$ for any $i > j$.  Thus, $G_T$ is consistent with the induced ordering on $\Psi \cap \Phi(w)$.\\ \\
Similarly, if $\Psi = \{\gamma_1,\ldots, \gamma_m\}$ is a full inversion set, then since $T$ is standard, either $$T(\gamma_1) < \cdots < T(\gamma_m)$$
\begin{center}
or
\end{center}
$$T(\gamma_m) < \cdots < T(\gamma_1).$$  For $G_T$, the subgraph induced by the vertices of $\Psi$ is consistent with one of those orders.  Thus we may define
$\mathcal{G}:\mathrm{Lab}(\mathcal{W},w) \rightarrow \mathrm{Tour}(w)$ by $\mathcal{G}(T) = G_T$.\\ \\
Conversely, let $G = (V,E)$ be a transitive tournament on the vertex set $V = \Phi(w)$ that satisfies the restrictions imposed by $\Phi(w)$.  By Lemma~\ref{l:uniquesort}, there is a unique topological sort $(\Phi(w),\leq)$.  We define a labeling $T$ by $T(\alpha) = k$ if $\alpha$ is the $k$-th element appearing in the linear ordering $\leq$.  This immediately implies that $T$ is sequential and $\text{supp}(T) = \Phi(w)$.\\ \\
If $\Psi \cap \Phi(w) = \{\gamma_1,\ldots,\gamma_k\}$ is a partial inversion set of $w$, then $\gamma_i \rightarrow \gamma_j$ for any $i > j$ since $G$ satisfies the restrictions imposed by $\Phi(w)$.  Since $\gamma_i \rightarrow \gamma_j$ for $i > j$ gives $\gamma_i < \gamma_j$ in the linear ordering given by the topological sort, it follows that $T(\gamma_1) > \cdots > T(\gamma_k)$ by our construction of $T$.\\ \\
If $\Psi = \{\gamma_1,\ldots,\gamma_m\}$ is a full inversion set, then the restrictions imposed by $\Phi(w)$ give that either $$\gamma_1 \rightarrow \cdots \rightarrow \gamma_m$$
\begin{center}
or
\end{center}
$$\gamma_m \rightarrow \cdots \rightarrow \gamma_1.$$  By our definition of $T$, this implies that either $$T(\gamma_1) < \cdots < T(\gamma_m)$$
\begin{center}
or
\end{center}
$$T(\gamma_m) < \cdots < T(\gamma_1),$$ so that $T$ is a standard $\mathcal{W}_{\Phi(w)}$-labeling.  Thus, we may define $$\mathcal{F}:\mathrm{Tour}(w) \rightarrow \mathrm{Lab}(\mathcal{W},w)$$ by $\mathcal{F}(G) = T_G$, where $T_G$ is determined by the unique topological sort of $G$.\\ \\
The labeling $T$ can be reconstructed by knowing whether $T(\alpha) < T(\beta)$ or $T(\alpha) > T(\beta)$ for each pair $\alpha,\beta \in \Phi(w)$.  Thus we have $T_{G_T} = T$, since $G_T$ faithfully encodes the order of the labels.  Similarly, $G_{T_G} = G$ so that $\mathcal{F}$ and $\mathcal{G}$ are inverse to each other.
\end{proof}
\begin{cor} \label{c:tourcorrespondence}
Let $(W,S)$ be an arbitrary Coxeter system and let $w \in W$.  Then the set $\mathrm{Tour}(w)$ of all transitive tournaments on $\Phi(w)$ satisfying the restrictions induced by $\Phi(w)$ are in $1-1$ correspondence with the set $\mathcal{R}(w)$ of reduced expressions for $w$.
\end{cor}
\begin{proof}
This follows by combining the correspondences of Lemma~\ref{l:tourcorrespondence} and Corollary~\ref{c:labelinversion}.
\end{proof}
\noindent
We now combine the correspondences obtained so far into a single main theorem.
\begin{theorem} \label{t:maintheorem}
Let $w \in W$.  Then the following sets are all in natural bijection:
\begin{enumerate}[(1)]
\item the set $\mathcal{R}(w)$ of reduced expressions for $w$;
\item the set $\mathrm{Lab}(w)$ of all sequential standard labelings with $\Phi(w)$ as support;
\item the set $\mathrm{Lab}(\mathcal{W},w)$ of all sequential standard $\mathcal{W}$-labelings with $\Phi(w)$ as support;
\item the set $\mathrm{Lab}(\mathcal{W}_{\Phi(w)},w)$ of all sequential standard $\mathcal{W}_{\Phi(w)}$-labelings that satisfy the restrictions of $\Phi(w)$;
\item the set $\mathrm{Tour}(w)$ of all transitive tournaments satisfying restrictions imposed by $\Phi(w)$.
\end{enumerate}
\end{theorem}
\begin{proof}
Proposition~\ref{p:correspondence1} establishes a bijection between $(1)$ and $(2)$.  Corollary~\ref{c:reducedtheorem} establishes a bijection between $(1)$ and $(3)$.  Corollary~\ref{c:wlabelcorrespondence} establishes a bijection between $(1)$ and $(4)$.  Corollary~\ref{c:tourcorrespondence} establishes a bijection between $(1)$ and $(5)$.
\end{proof}
\begin{ex}
To illustrate Theorem~\ref{t:maintheorem}, we give a ``side by side'' view of the corresponding element of the sets $(1)-(5)$ for a single reduced expression of a Coxeter group element.  As in Example~\ref{ex:inversionex}, we let $W = W(B_3)$ and $w = \phi(u,s,t,s,t,u)$.  Let $\textbf{x} = (u,s,t,s,t,u)$.  Recall that
\begin{equation*}
\begin{split}
\Phi^+ = \{\alpha_s,\alpha_t,\alpha_u,\sqrt{2}\alpha_s + \alpha_t, \alpha_t + \alpha_u, \alpha_s + \sqrt{2} \alpha_t,  \sqrt{2}\alpha_s + \alpha_t + \alpha_u,& \\ \alpha_s + \sqrt{2}\alpha_t + \sqrt{2} \alpha_u, \sqrt{2}\alpha_s + 2\alpha_t + \alpha_u\}.
\end{split}
\end{equation*}
Also recall that $$\Phi(w) = \{\alpha_s,\alpha_u, \alpha_t + \alpha_u,\sqrt{2}\alpha_s + \alpha_t + \alpha_u, \alpha_s + \sqrt{2}\alpha_t + \sqrt{2}\alpha_u, \sqrt{2}\alpha_s + 2\alpha_t + \alpha_u\}.$$
Then the corresponding elements for $(1)-(5)$ are given by:
\begin{enumerate}[(1)]
\item the reduced expression $\textbf{x} = (u,s,t,s,t,u)$;
\item the labeling $T_{\textbf{x}}:\Phi^+ \rightarrow \{1,2,3,4,5,6\}$ given by $$T_{\textbf{x}}(\alpha_t) = T_{\textbf{x}}(\alpha_s + \sqrt{2}\alpha_t) = T_{\textbf{x}}(\sqrt{2}\alpha_s + \alpha_t) = 0,$$  $$T_{\textbf{x}}(\sqrt{2}\alpha_s + 2\alpha_t + \alpha_u) = 1, \,T_{\textbf{x}}(\alpha_s) = 2,\, T_{\textbf{x}}(\sqrt{2}\alpha_s + \alpha_t + \alpha_u) = 3,$$  $$T_{\textbf{x}}(\alpha_s + \sqrt{2}\alpha_u + \sqrt{2} \alpha_u) = 4, \,T_{\textbf{x}}(\alpha_t + \alpha_u) = 5, \, T_{\textbf{x}}(\alpha_u) = 6;$$
\item the sequential standard $\mathcal{W}$-labeling below, using the root chart of Example~\ref{ex:mainexample}:\\ \\
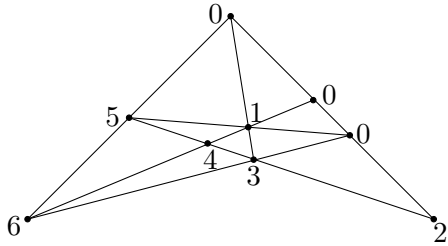
\begin{tikzpicture}[scale=1.8] \label{Bstandseqfig}
\filldraw           (3,0) circle (0.02)
                    (0.75,0.75) circle (0.02)
                    (1.5,1.5) circle (0.02)
                    (0,0) circle (0.02)
                    (2.11,0.88) circle (0.02)
                    (2.38,0.62) circle (0.02)
                    (1.67,0.44) circle (0.02)
                    (1.33,.56) circle (0.02)
                    (1.63,0.68) circle (0.02);
\draw (0,0) -- (1.5,1.5);
\draw (1.5,1.5) -- (3,0);
\draw (0.75,0.75) -- (3,0);
\draw (0,0) -- (2.11,0.88);
\draw (1.5,1.5) -- (1.67,0.44);
\draw (0.75,0.75) -- (2.38,0.62);
\draw (0,0) -- (2.38,0.62);
\draw (-.1,-.05) node {$6$};
\draw (3.05,-.1) node {$2$};
\draw (1.67,0.32) node {$3$};
\draw (1.35,0.44) node {$4$};
\draw (0.63,0.76) node {$5$};
\draw (1.69,0.79) node {$1$};
\draw (2.48,0.63) node {$0$};
\draw (2.23,0.92) node {$0$};
\draw (1.39,1.52) node {$0$};
\end{tikzpicture}
\item the sequential standard $\mathcal{W}_{\Phi(w)}$-labeling below that satisfies the restrictions of $\Phi(w)$:\\ \\
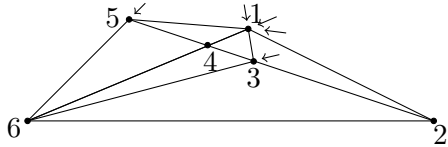
\begin{tikzpicture}[scale=1.8] \label{Brestrictedfig}
\filldraw           (3,0) circle (0.02)
                    (0.75,0.75) circle (0.02)
                    (0,0) circle (0.02)
                    (1.67,0.44) circle (0.02)
                    (1.33,.56) circle (0.02)
                    (1.63,0.68) circle (0.02);
\draw (0,0) -- (0.75,0.75);
\draw (0.75,0.75) -- (3,0);
\draw (0,0) -- (1.63,0.68);
\draw (1.63,0.68) -- (1.67,0.44);
\draw (0.75,0.75) -- (1.63,0.68);
\draw (0,0) -- (1.63,0.68);
\draw (0,0) -- (3,0);
\draw (0,0) -- (1.67,0.44);
\draw (1.63,0.68) -- (3,0);
\draw [->] (0.87,0.87) -- (0.79,0.79);
\draw [->] (1.84,0.77) -- (1.7,0.71);
\draw [->] (1.86,0.49) -- (1.73,0.46);
\draw [->] (1.6,0.86) -- (1.62,0.74);
\draw [->] (1.91,0.65) -- (1.75,0.67);
\draw (-.1,-.05) node {$6$};
\draw (3.05,-.1) node {$2$};
\draw (1.67,0.32) node {$3$};
\draw (1.35,0.44) node {$4$};
\draw (0.63,0.76) node {$5$};
\draw (1.69,0.79) node {$1$};
\end{tikzpicture}
\item the transitive tournament satisfying the restrictions imposed by $\Phi(w)$: \\ \\
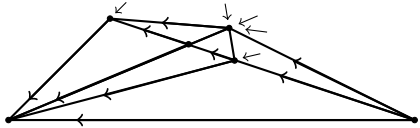
\begin{tikzpicture}[scale=1.8] \label{Btournamentfig}
\filldraw           (3,0) circle (0.02)
                    (0.75,0.75) circle (0.02)
                    (0,0) circle (0.02)
                    (1.67,0.44) circle (0.02)
                    (1.33,.56) circle (0.02)
                    (1.63,0.68) circle (0.02);
\draw [thick] (0,0) -- (0.75,0.75);
\draw [thick] (0.75,0.75) -- (3,0);
\draw [thick] (0,0) -- (1.63,0.68);
\draw [thick] (1.63,0.68) -- (1.67,0.44);
\draw [thick] (0.75,0.75) -- (1.63,0.68);
\draw [thick] (0,0) -- (1.63,0.68);
\draw [thick] (0,0) -- (3,0);
\draw [thick] (0,0) -- (1.67,0.44);
\draw [thick] (1.63,0.68) -- (3,0);
\draw [->] (0.87,0.87) -- (0.79,0.79);
\draw [->] (1.84,0.77) -- (1.7,0.71);
\draw [->] (1.86,0.49) -- (1.73,0.46);
\draw [->] (1.6,0.86) -- (1.62,0.74);
\draw [->] (1.91,0.65) -- (1.75,0.67);
\draw [->, thick] (3,0) -- (0.5,0);
\draw [->, thick] (0.75,0.75) -- (0.15,0.15);
\draw [->, thick] (1.33,0.56) -- (0.35,0.15);
\draw [->, thick] (1.67,0.44) -- (0.7,0.18);
\draw [->, thick] (3,0) -- (2,0.33);
\draw [->, thick] (1.67,0.44) -- (1.49,0.5);
\draw [->, thick] (1.33,0.56) -- (0.99, 0.67);
\draw [->, thick] (3,0) -- (2.1,0.44);
\draw [->, thick] (1.63,0.68) -- (1.13, 0.72);
\end{tikzpicture}
\end{enumerate}
\noindent
For the pictorial representation of the tournament in $(5)$, we make the convention that we need not display arrows implied by transitivity.  Furthermore, the arrows ``outside the lines'' give the restrictions on $\Phi(w)$ associated with the partial inversion sets of $\Phi(w)$.  The arrows ``along the lines'' give the choice of direction made for each line of the standard segment structure $\mathcal{W}_{\Phi(w)}$ of $\Phi(w)$.  Observe that the arrows point in the directions of the labels in increasing order from the previous two examples.
\end{ex}
\chapter{Reduced Expression Combinatorics}
\label{redexpchapt}
The goal of this chapter is to apply the constructions of the previous chapters to answer questions involving the combinatorics of reduced expressions.  The main theorem of this chapter is Theorem~\ref{t:deletiontheorem}, which gives a formula for the length of an element of $W$ obtained by deleting a generator from a reduced expression for another element of $W$.  In Section $5.2$, we collect facts about the relationship between braid moves and root sequences which will be applied in Section $5.3$.  In Section $5.3$, we introduce the freely braided elements of Green and Losonczy, and give a characterization of the elements $w$ that are freely braided in terms of statistics derived from $w$.
\section{Deletion}
Segment structures do not come equipped with a metric.  Though we can define one based on the betweenness relation, we only need such a metric for the standard segment structure $\mathcal{W}$ of $\Phi^+$.
\begin{defn} \label{d:distance}
Let $\Psi$ be an inversion set and let $\gamma$ and $\delta$ be the canonical simple roots of $\Psi$.  Let $\overline{\gamma}$ and $\overline{\delta}$ be the associated $\gamma$- and $\delta$-sequences.  If $\alpha = \gamma_i$ and $\beta = \gamma_j$, then we define the \emph{distance between $\alpha$ and $\beta$} to be $|i - j|$.  Similarly, if $\alpha = \delta_i$ and $\beta = \delta_j$ then we define the distance between $\alpha$ and $\beta$ to be $|i - j|$.  If neither of the above two conditions hold then the distance between $\alpha$ and $\beta$ is undefined.
We denote the distance between $\alpha$ and $\beta$ by $d(\alpha,\beta)$.
\end{defn}
\begin{rem}
We only use this metric in the context of the restriction of the standard segment structure $\mathcal{W} = (\Phi^+,\text{Inv}(\Phi^+),B_W)$ to $$\mathcal{W}_{\Phi(w)} = (\Phi(w), \text{Inv}(\Phi^+)_{\Phi(w)}, B_{\Phi(w)})$$ for a prescribed $w \in W$.  In this context, if $\alpha, \beta \in \Psi \cap \Phi(w)$ for some inversion set $\Psi$, then Corollary~\ref{c:biconvexstructure} implies $\alpha$ and $\beta$ lie in the same local root sequence.  It follows by Corollary~\ref{c:inversionline} that if $\alpha, \beta \in \Phi(w)$, then $d(\alpha,\beta)$ is defined.
\end{rem}
\begin{defn}  \label{d:endpointdistance}
Let $\mathcal{W}_{\Phi(w)} = (\Phi(w), \text{Inv}(W)_{\Phi(w)}, B_{\Phi(w)})$ be the standard segment structure of $\Phi(w)$.  Let $L \in \text{Inv}(W)_{\Phi(w)}$ be a line in $\mathcal{W}_{\Phi(w)}$ and let $\theta \in L$.  If $L = [\gamma_1,\gamma_k]$, then we define the \emph{minimum distance of $\theta = \gamma_i$ to the endpoints of $L$} to be $\text{min}(d(\gamma_i,\gamma_1),d(\gamma_i,\gamma_k)) = \text{min}(i - 1,k - i)$, which we denote by $|\theta|_L$.  If $\theta \not\in L$, we set $|\theta|_L = 0$.
\end{defn}
\begin{ex} \label{ex:distanceex}
Let $W = W(B_3)$ and $w = \phi(u,s,t,s,t,u)$.  Below, we represent the standard segment structure of $\Phi(w)$ in a way similar to Example~\ref{ex:inversionex} so that roots in the same inversion set (or partial inversion set) appear to be collinear.  Throughout our explanation, we refer to the roots by their name given in the chart below. (Note that this naming is not a ``labeling'' in the sense of Definition~\ref{d:labeling}.)\\
\\
Let $L$ be the full inversion set $\Psi \cap \Phi(w) = \{B,C,D,E\}$.  Then root $B$ and root $E$ are the canonical simple roots of $\Psi$.  Thus, if $E = \gamma = \gamma_1$, we have that $E = \gamma_1$, $D = \gamma_2$, $C = \gamma_3$, and $B = \gamma_4$.  To find $|C|_\Psi$,  we note that $d(C,E) = 2$ and $d(C,B) = 1$.  Thus $|C|_\Psi = \text{min}(1,2) = 1$.  On the other hand, if $\Psi'$ is the full inversion set $\Psi' \cap \Phi(w) = \{A,B\}$, then we have $|C|_{\Psi'} = 0$, since $C \not\in \Psi'$.  In the partial inversion $\Psi'' \cap \Phi(w) = \{C,F\}$, we have $|C|_{\Psi''} = 0$ because $d(C,C) = 0$ and $C$ is an endpoint of $\Psi''$.
\end{ex}
\noindent
One advantage to ordering the positive roots of a dihedral subsystem according to how many alternating reflections are applied to a root is that we can use statistics from the ordering to determine whether a reflection applied to a root within the dihedral subsystem is positive or negative.
\begin{lem} \label{l:localreflect}
Let $\gamma$, $\delta$ be the canonical simple roots for a local Coxeter system and let $i \geq 1, j \in \mathbb{Z}$.  Then we have $s_{\gamma_i} (\gamma_j) = \delta_{j - 2i + 1}$.  Similarly, $s_{\delta_i} (\delta_j) = \gamma_{j - 2i + 1}$.
\end{lem}
\begin{proof}
For $i = 1$, this is just the backward recurrence $s_\gamma (\gamma_j) = \delta_{j-1}$ of (\ref{e:backrecurrence}).  For $i > 1$, we first show that $s_{\gamma_i} = [(s_\gamma s_\delta)]_{2i-1}$ and $s_{\delta_i} = [(s_\delta s_\gamma)]_{2i-1}$ by induction:
\begin{equation*}
\begin{split}
s_{\gamma_i} &=
    s_{s_\gamma \delta_{i-1}}\\
    &= s_\gamma s_{\delta_{i-1}} s_\gamma\\
    &= s_\gamma [(s_\delta s_\gamma)]_{2i-3} s_\gamma\\
    &= [(s_\gamma s_\delta)]_{2i - 1}.
\end{split}
\end{equation*}
Now, since $\gamma_j = s_\gamma (\delta_{j-1})$ and since the last factor of $s_{\gamma_i}$ is $s_\gamma$, we have
\begin{equation*}
\begin{split}
s_{\gamma_i}(\gamma_j) &= [(s_\gamma s_\delta)]_{2i - 1} \ s_\gamma (\delta_{j-1}) \\
&= [(s_\gamma s_\delta)]_{2i - 3} \ s_\delta s_\gamma s_\gamma (\delta_{j-1})\\
&= s_{\gamma_{i-1}} (\gamma_{j-2})\\
&= \delta_{(j - 2) - 2i + 3}\\
&= \delta_{j - 2i + 1}.
\end{split}
\end{equation*}
The third equality applies the backward recurrence of (\ref{e:backrecurrence}) to get the equality $s_{\delta}(\delta_{j-1}) = \gamma_{j-2}$.  The fourth equality follows by induction.  Interchanging the roles of $\gamma$ and $\delta$ gives the last statement of the lemma.
\end{proof}
\begin{cor} \label{c:localreflect}
Let $\gamma, \delta$ be canonical simple roots for a local Coxeter system and let $i \geq 1$.  Let $m = |s_\gamma s_\delta|$ and let $j$ be such that $i < j \leq m$ ($j < m$ if $m = \infty$).  Then $s_{\gamma_i}(\gamma_j)$ is negative if $j < 2i$, otherwise $s_{\gamma_i}(\gamma_j)$ is positive.
\end{cor}
\begin{proof}
By Lemma~\ref{l:localreflect}, $s_{\gamma_i}(\gamma_j) = \delta_{j - 2i + 1}$.  If $j < 2i$, then since $j \leq m$, we have $-(m - 1) \leq j - 2i + 1 \leq 0$ \ ($j - 2i + 1 \leq 0$ if $m = \infty$), which implies that $s_{\gamma_i}(\gamma_j)$ is negative by Lemma~\ref{l:sequencenegative}.  If $j \geq 2i$, then since $j \leq m$ and $i \geq 1$, we have $1 \leq j - 2i + 1 \leq m$, which implies that $s_{\gamma_i}(\gamma_j)$ is positive by Lemma~\ref{l:sequencepositive}.
\end{proof}
\begin{lem} \label{l:gammachoice}
Let $w \in W$ and let $\Psi$ be an inversion set intersecting $\Phi(w)$ so that $|\Psi \cap \Phi(w)| \geq 2$.  Let $\textbf{x}$ be a reduced expression for $w$.  Then there exists a unique canonical simple root $\gamma$ of $\Psi$ that is the last root of $\Psi \cap \Phi(w)$ to occur in the root sequence $\overline{\theta}(\textbf{x})$.
\end{lem}
\begin{proof}
If $\Psi \cap \Phi(w)$ is a partial inversion set, then by Lemma~\ref{l:thatlemma}, there is only one canonical simple root of $\Psi$ in $\Phi(w)$.  By Lemma~\ref{l:wlabel}(2), $T_{\textbf{x}}(\gamma)$ has the highest label of the roots in $\Psi \cap \Phi(w)$.\\ \\
If $\Psi \cap \Phi(w)$ is a full inversion set, then there are two distinct canonical simple roots in $\Phi(w)$.  By  Lemmas~\ref{l:wlabel} and \ref{l:sequencerelation}, the root labeled highest by $T_{\textbf{x}}$ among all the roots of $\Psi \cap \Phi(w)$ is one of the canonical simple roots of $\Psi$.  We let $\gamma$ be that highest labeled root and note that the root sequence result follows from Definition~\ref{d:standardencoding}.
\end{proof}
\noindent
Recall that $D_j(\textbf{x})$ denotes the expression obtained from $\textbf{x}$ by deleting the $j$-th generator.  By Lemma~\ref{l:reflectdelete}, if $\textbf{x}$ is a reduced expression for $w$ and $\theta_j$ is the $j$-th entry of the root sequence $\overline{\theta}(\textbf{x})$, then $D_j(\textbf{x})$ is an expression for $ws_{\theta_j}$.  Thus, the length of the Coxeter group element $\phi(D_j(\textbf{x}))$ is given by the length of $ws_{\theta_j}$.  Lastly, recall that by Corollary~\ref{c:deletion}, the length of $ws_{\theta_j}$ is $\ell(w) - 2d - 1$ where $d$ is the number of $\theta_k$ in $\overline{\theta}(\textbf{x})$ satisfying $1 \leq k < j$ and $s_{\theta_j}(\theta_k) \in \Phi^-$.  The content of the next theorem is that the number $d$ can be expressed in terms of a statistic associated to $\theta_k$ in the standard segment structure of $\Phi(w)$.
\begin{theorem} \label{t:deletiontheorem}
Let $w \in W$, let $\theta \in \Phi(w)$, and let $$\mathcal{W}_{\Phi(w)} = (\Phi(w), \text{Inv}(W)_{\Phi(w)}, B_{\Phi(w)})$$ be the standard segment structure of $\Phi(w)$.  Define $$D = \sum_{L \in \text{Inv}(W)_{\Phi(w)}} |\theta|_L,$$ where $|\theta|_L$ is as it is in Definition~\ref{d:endpointdistance}.  Then
\begin{equation*}
\ell(ws_\theta) = \ell(w) - 1 - 2D.
\end{equation*}
\end{theorem}
\begin{proof}
Let $\textbf{x}$ be a reduced expression for $w$ with root sequence $\overline{\theta}(\textbf{x})$ and let $n$ be the index of the root sequence satisfying $\theta_n = \theta$.  Since $\theta$ is fixed throughout the proof, so is the associated index $n$.  Let $$\Theta = \{\theta_k \in \overline{\theta}(\textbf{x}) \, : \, 1 \leq k < n \text{ and } s_{\theta_n} (\theta_k) \in \Phi^-\}$$ and set $d = |\Theta|$.  By Corollary~\ref{c:deletion}, it suffices to show that $D = d$.\\ \\
For each $L \in \text{Inv}(W)_{\Phi(w)}$, define $\Theta_L = \Theta \cap L$.\\ \\
By Corollary~\ref{c:inversionline}, for any $k < n$ there exists a unique inversion set $\Psi$ containing $\theta_n$ and $\theta_k$.  By Definition~\ref{d:standardphisegment} each $L \in \text{Inv}(W)_{\Phi(w)}$ is of the form $\Psi \cap \Phi(w)$ for some inversion set $\Psi$.  Thus there exists a unique inversion set $L \in \text{Inv}(W)_{\Phi(w)}$ containing both $\theta_n$ and $\theta_k$ by Corollary~\ref{c:inversionline}.  By Lemma~\ref{l:oneroot} if $L, L' \in \text{Inv}(W)_{\Phi(w)}$ are distinct lines that contain $\theta_n$, then they have no other points in common.  In other words, $$(L \setminus \{\theta_n\}) \cap (L' \setminus \{\theta_n\}) = \emptyset.$$  Since $\theta_n \not\in \Theta$, $\Theta$ is the disjoint union of sets of the form $\Theta \cap L$, where $L$ contains $\theta_n$.  Thus we have $$|\Theta| = \sum |\Theta \cap L|,$$
where the sum is over all $L \in \text{Inv}(W)_{\Phi(w)}$ that contain $\theta = \theta_n$.\\ \\
Since $|\theta_n|_L = 0$ for any $L \in \text{Inv}(W)_{\Phi(w)}$ such that $\theta_n \not\in L$, it now suffices to show that $|\theta_n|_L = |\Theta \cap L|$ for all $L \in \text{Inv}(W)_{\Phi(w)}$ containing $\theta_n$.\\ \\
Now let $L$ be an arbitrary line $L$ of $\text{Inv}(W)_{\Phi(w)}$  containing $\theta_n$ and let $\theta_l$ be a root occurring before $\theta_n$ in $\overline{\theta}(\textbf{x})$ so that $1 \leq l < n$.  By Corollary~\ref{c:inversionline}, there exists a unique inversion set $\Psi$ containing $\theta_n$ and $\theta_l$.  By Lemma~\ref{l:gammachoice}, we may let $\gamma$ denote the unique canonical simple root of $\Psi$ that occurs latest in the root sequence $\overline{\theta}(\textbf{x})$.\\ \\
By Corollary~\ref{c:biconvexstructure}, since $L = \Psi \cap \Phi(w)$ and $|L| \geq 2$, we have $L = [\gamma_1,\gamma_k]$ for some $k \leq m = |s_\gamma s_\delta|$.  Thus $\theta_n = \gamma_i$ for some $i$ satisfying $1 \leq i \leq k$.  We have chosen $\gamma$ so that the only roots in $[\gamma_1,\gamma_k]$ occurring before $\theta_n$ in the root sequence $\overline{\theta}(\textbf{x})$ are the roots $\gamma_{i+1},\ldots,\gamma_k$.  Thus, $\theta_l = \gamma_j$ where $i < j$.\\ \\
By Corollary~\ref{c:localreflect}, since $i < j$, $s_{\gamma_i}(\gamma_j)$ is negative if and only $j < 2i$.  There are two cases:  $k - i \leq i - 1$ and $k - i > i - 1$.\\ \\
First suppose $k - i \leq i - 1$, so that $k \leq 2i - 1$.  Then, for any $j$ satisfying $i + 1 \leq j \leq k$, $k \leq 2i - 1$ implies $j < 2i$.  Thus $s_{\gamma_i}(\gamma_j)$ is negative.  There are $k - i$ such $j$ in this interval, so that, in this case, $|\theta_n|_L = \text{min}(k - i, i - 1)$.\\ \\
Next suppose $i - 1 < k - i$ so that $2i - 1 < k$.  Then, for $j$ satisfying $i + 1 \leq j \leq 2i - 1 < k$, $s_{\gamma_i}(\gamma_j)$ is negative.  For $j$ satisfying $2i \leq j \leq k$, $s_{\gamma_i}(\gamma_j)$ is positive.  Thus the number of $s_{\gamma_i}(\gamma_j)$ sent negative, where $i < j \leq k$, is $i - 1$, which is $\text{min}(k - i, i - 1)$ in this case as well.\\ \\
Thus, in both cases the number of roots in $L$ occurring before $\theta_n$ in $\overline{\theta}(\textbf{x})$ is $\text{min}(k - i, i - 1) = |\theta_n|_L$, as desired.
\end{proof}
\begin{ex}
Again we use the example $W = W(B_3)$, $\textbf{x} = (u,s,t,s,t,u)$, and $w = \phi(\textbf{x})$.  Using the naming scheme of Example~\ref{ex:distanceex}, we have that the root sequence $\overline{\theta}(\textbf{x}) = (F,B,C,D,E,A)$.  Thus, deletion of the root $C$ in $\overline{\theta}(\textbf{x})$ (given by right multiplying by $s_{\sqrt{2}\alpha_s + \alpha_t + \alpha_u}$) corresponds to the expression $\textbf{x}' = (u,s,s,t,u)$.  We have that $|C|_{\Psi} = 0$ except for $\Psi = \{B,C,D,E\}$.  In that case $|C|_\Psi = 1$, so that $D = 1$ in Theorem~\ref{t:deletiontheorem}.  Theorem~\ref{t:deletiontheorem} now predicts that $\ell(\textbf{x}') = 6 - 2(1) - 1 = 3$, which is easily verified by inspection.\\ \\
If instead we delete the root $D$ from $\overline{\theta}(\textbf{x})$, we get the expression $$\textbf{x}' = (u,s,t,t,u).$$  On the lines $L = \{B,C,D,E\}$ and $L' = \{A,D,F\}$, we have $|D|_L = 1$ and $D_{L'} = 1$, respectively.  Thus, $D = 2$ in Theorem~\ref{t:deletiontheorem}.  It follows that $\ell(\phi(\textbf{x}')) = 6 - 2(2) - 1 = 1$, which is also easily verified by inspection.  Note that the deletion of a generator associated to any other root results in a reduced expression.
\end{ex}
\noindent
In \cite{shortbraid}, Fan explores the connection between three types of special elements of Coxeter groups: fully commutative elements, short-braid avoiding elements, and fully covering elements.  The fully commutative elements are the $w \in W$ such that only relations of the form $st = ts$ need to be applied to transform one reduced expression for $w$ into another.  The short-braid avoiding elements are those whose reduced expressions avoid substrings of the form $sts$, where $s,t \in S$ are noncommuting generators.  If the deletion of any generator from any reduced expression for $w$ results in a reduced expression (for some $w' \in W$), then the element is called fully covering.  Fan (\cite[Theorem 1]{shortbraid}) shows that in a finite Weyl group, the short-braid avoiding elements are precisely the fully covering elements.  In \cite{hagiwara}, Hagiwara, et al$.$ show that the fully commutative elements of a simply-laced Coxeter group are precisely the fully covering elements so long as the group has finitely many fully commutative elements.  Since fully commutative elements are precisely the short-braid avoiding elements in simply-laced Coxeter groups, this extends Fan's result.  In Proposition~\ref{p:fullycovering}, we give a characterization of the fully covering elements in terms of the geometry of $\mathcal{W}_{\Phi(w)}$, which applies in the setting of arbitrary Coxeter groups.
\begin{defn} \label{d:bruhatdef}
Let $(W,S)$ be a Coxeter system and let $u,v \in W$.  If there exists a reflection $t \in T$ such that $ut = w$, then we write $u \rightarrow w$.  We denote the transitive closure of $\rightarrow$ by $\prec$, which we call the \emph{Bruhat--Chevalley order of $W$}.
\end{defn}
\begin{defn} \label{d:fullycovering}
We say that \emph{$w$ covers $v$ relative to $\prec$} if $v \prec u \prec w$ implies that either $u = v$ or $u = w$.  We say that $w$ is \emph{fully covering} if $$\ell(w) = |\{u \in W \, : \, w \text{ covers } u\}|.$$
\end{defn}
\noindent
In the next proposition, we record some well-known facts about the Bruhat--Chevalley order.  In Corollary~\ref{c:charcovering}, we give a trivial (but useful) translation of Definition~\ref{d:fullycovering} into a generator deletion property.
\begin{prop} \label{p:bruhatcharacterization}
Let $(W,S)$ be a Coxeter system and let $\prec$ denote the Bruhat--Chevalley order of $W$.  Then:
\begin{enumerate}[(1)]
\item Suppose $w = s_1\cdots s_k$ is a reduced expression for $w \in W$.  Then for any $u \in W$, $u \prec w$ if and only if $u = s_{i_1} \cdots s_{i_k}$ for some choice $i_1,\ldots,i_k$ of indices.
\item  For any $u \in W$, $w$ covers $u$ if and only if $\ell(u) = \ell(w) - 1$.
\end{enumerate}
\end{prop}
\begin{proof}
For fact $(1)$, see \cite[Theorem 5.10]{humph}.  For fact $(2)$, see \cite[Proposition 5.11]{humph}.
\end{proof}
\begin{cor} \label{c:charcovering}
Let $w \in W$.  Let $s_1 \cdots s_k$ be a reduced expression for $w$.  Then $w$ is fully covering if and only if the expression $s_1 \cdots \widehat{s_i} \cdots s_k$ is reduced for every $i$ satisfying $1 \leq i \leq k$.
\end{cor}
\begin{proof}
Let $w$ be fully covering so that $w$ covers $\ell(w)$ elements in the Bruhat--Chevalley order.  By Proposition~\ref{p:bruhatcharacterization}, $w$ covers $u$ if and only if the equation $\ell(u) = \ell(w) - 1$ holds and $u$ can be realized as a subexpression of $s_1 \cdots s_k$.  The only subexpressions having $k-1$ generators are those formed by deletion of a generator.  Their length is $\ell(w) - 1$ if and only if each such expression is reduced.
\end{proof}
\begin{ex} \label{ex:fullycovering}
Let $W = W(A_3)$ with $S = \{s,t,u\}$, $m_{s,t} = 3$, $m_{t,u} = 3$, and $m_{s,u} = 2$.  Then the element $w = \phi(t,s,u,t)$ is fully covering since $(s,u,t)$, $(t,u,t)$, $(t,s,t)$, and $(t,s,u)$ are all reduced expressions in type $A_3$.\\ \\
For an example of an element that is not fully covering, consider the Coxeter group $W = W(\widetilde{A_2})$.  This is the Coxeter group with generating set $S = \{s,t,u\}$, $m_{s,t} = 3$, $m_{t,u} = 3$, and $m_{s,u} = 3$.  Thus, $W$ is simply-laced, but it is known (see \cite[Theorem 4.1]{onfc}) that there are infinitely many fully commutative elements in type $\widetilde{A_2}$.  Thus Hagiwara et al.'s characterization (\cite[Theorem 2.9]{hagiwara}) of fully covering elements does not apply.  Let $w = \phi(t,u,s,t,u)$.  Then $w$ is fully commutative since no Coxeter relations can be applied.  However, $(t,u,t,u)$ is a non-reduced subexpression, so $w$ is not fully covering.
\end{ex}
\begin{lem} \label{l:fullycovering}
Let $(W,S)$ be an arbitrary Coxeter group and let $$\mathcal{W}_{\Phi(w)} = (\Phi(w), \text{Inv}(W)_{\Phi(w)}, B_{\Phi(w)})$$ be the standard segment structure of $\Phi(w)$.  Then $w \in W$ is fully covering if and only if every point of the standard segment structure is an endpoint of the standard segment structure $\mathcal{W}_{\Phi(w)}$ of $\Phi(w)$.
\end{lem}
\begin{proof}
If $w$ is fully covering then given a reduced expression $\textbf{x}$ for $w$, $D_j(\textbf{x})$ is reduced for every $j$ satisfying $1 \leq j \leq \ell(w)$.  Thus, by Theorem~\ref{t:deletiontheorem}, we have that for every $\mu \in \Phi(w)$, $|\mu|_L = 0$ for every line $L$ containing $\mu$.  Thus $\mu$ is an endpoint of $(\Phi(w),\text{Inv}(W)_{\Phi(w)}, B_{\Phi(w)})$.\\ \\
Conversely, if every point of $\Phi(w)$ is an endpoint, then by Theorem~\ref{t:deletiontheorem}, $\ell(ws_\mu) = \ell(w) - 1$ for every $\mu \in \Phi(w)$.  Thus, the number of roots of $\Phi(w)$ for which this happens is $\ell(w)$.
\end{proof}
\noindent
Recall that by Definition~\ref{d:segmentstructure} that every line of a segment structure contains at least two points.
\begin{prop} \label{p:fullycovering}
Let $w \in W$ and let $$\mathcal{W}_{\Phi(w)} = (\Phi(w), \text{Inv}(W)_{\Phi(w)}, B_{\Phi(w)})$$ be the standard segment structure of $\Phi(w)$.  Then $w \in W$ is fully covering if and only if $|L| = 2$ for every line $L \in \text{Inv}(W)_{\Phi(w)}$.
\end{prop}
\begin{proof}
Let $\lambda$ be a positive root in $\Phi(w)$ and suppose $|L| = 2$ for every line $L \in \text{Inv}(W)_{\Phi(w)}$.  By Definition~\ref{d:segmentstructure}, $\lambda$ is an endpoint of the standard segment structure $\mathcal{W}_{\Phi(w)}$ if $\lambda$ is an endpoint for each line $L \in \text{Inv}(W)_{\Phi(w)}$ that contains $\lambda$.  However, if $\lambda \in L$, then $\lambda$ is an endpoint of $L$ since $|L| = 2$ by Remark~\ref{rem:lines}.\\ \\
Conversely, if $w \in W$ is fully covering, then by Lemma~\ref{l:fullycovering}, every root $\lambda \in \Phi(w)$ is an endpoint of the standard segment structure $\mathcal{W}_{\Phi(w)}$.  However, if there exists a line $L \in \text{Inv}(W)_{\Phi(w)}$ such that $|L| > 2$, then property (B3) of Definition~\ref{d:segmentstructure} implies the existence of an intermediate point $\lambda \in L$.  This contradicts that $\lambda$ is an endpoint of the standard segment structure $\mathcal{W}_{\Phi(w)}$.
\end{proof}
\begin{ex}
Let $W = W(\widetilde{A_2})$ and consider the element $w = \phi(t,u,s,t,u)$ given in Example~\ref{ex:fullycovering}.  By calculation, $$\Phi(w) = \{\alpha_u, \alpha_t + \alpha_u, \alpha_s + \alpha_t + 2\alpha_u, \alpha_s + 2\alpha_t + 2\alpha_u, 2\alpha_s + 2\alpha_t + 3\alpha_u\}.$$  The set $$L = \{\alpha_u, (\alpha_s + \alpha_t) + 2\alpha_u, 2(\alpha_s + \alpha_t) + 3\alpha_u\}$$ is a partial inversion set:  one can check (using Definitions~\ref{d:preorder} and \ref{d:simplelocal}) that the inversion set containing $\alpha_u$ and $\alpha_s + \alpha_t$ has $\gamma = \alpha_u$ and $\delta = \alpha_s + \alpha_t$ as canonical simple roots so that $L$ is a partial inversion set.  Since $|L| = 3$, we have that $w$ is not fully covering by Proposition~\ref{p:fullycovering}.
\end{ex}
\section{Braid moves}
\begin{defn} \label{braidmovedef}
Let $\textbf{a},\textbf{b} \in S^*$ and $s,t \in S$.  Let $\textbf{x} = \textbf{a}(s,t)_{m_{s,t}}\textbf{b}$ and $\textbf{x}' = \textbf{a}(t,s)_{m_{s,t}}\textbf{b}$.  We call the substitution $\textbf{a}(s,t)_{m_{s,t}}\textbf{b} = \textbf{a}(t,s)_{m_{s,t}}\textbf{b}$ an \emph{$m$-braid move} and we say that $\textbf{x}$ and $\textbf{x}'$ are \emph{braid move related}.  We denote the braid move transforming $\textbf{x}$ into $\textbf{x}'$ by $\textbf{x} \rightarrow \textbf{x}'$.  The reflexive and transitive closure of the braid move relation is an equivalence relation called \emph{braid equivalence}.
\end{defn}
\noindent
A well-known result of Matsumoto \cite{matsumoto} and Tits \cite{tits} states that any reduced expression for $w$ can be transformed into any other by applying a finite sequence of braid moves.  Thus, the set $\mathcal{R}(w)$ of reduced expressions for $w$ form a single equivalence class with respect to braid equivalence.
\begin{theorem}[\textbf{Matsumoto, Tits}] \label{t:matsumototits}
Let $w \in W$.  Let $\textbf{x}$ and $\textbf{y}$ be reduced expressions for $w$.  Then $\textbf{x}$ and $\textbf{y}$ are braid equivalent.
\end{theorem}
\begin{proof}
For a proof, see \cite[Theorem 3.3.1(ii)]{bb}.
\end{proof}
\begin{defn}
Let $\overline{\theta} = (\theta_1,\ldots,\theta_k)$ be a sequence of roots.  Then we denote the sequence $(\theta_k,\ldots,\theta_1)$ by $\overline{\theta}^R$ and call $\overline{\theta}^R$ the \emph{reversal of $\overline{\theta}$}.
\end{defn}
\noindent
Recall that if $\textbf{x}$ factors as $\textbf{x} = \textbf{x}_1 \cdots \textbf{x}_n$ then we call $\overline{\theta}(\textbf{x}) = \overline{\theta_1} \cdots \overline{\theta}_n$ the decomposition of $\overline{\theta}(\textbf{x})$ respecting $\textbf{x}_1 \cdots \textbf{x}_n$ if $\ell(\textbf{x}_i) = \ell(\overline{\theta_i})$ for each $i$ satisfying $1 \leq i \leq n$.
\begin{lem} \label{l:revroots}
Let $m = m_{s,t}$ be finite and let $\textbf{x} = \textbf{a}(s,t)_m\textbf{b}$ be a reduced expression for some $w \in W$.  Let $\overline{\theta}(\textbf{x}) = \overline{\theta_1}\,\overline{\theta_2}\,\overline{\theta_3}$ be the decomposition of $\overline{\theta}(\textbf{x})$ respecting $\textbf{a}(s,t)_m\textbf{b}$.  Let $\textbf{x}' = \textbf{a}(t,s)_m\textbf{b}$.  Then $\overline{\theta}(\textbf{x}') = \overline{\theta_1}\,\overline{\theta_2}^R\,\overline{\theta_3}$ is the decomposition of $\overline{\theta}(\textbf{x}')$ respecting $\textbf{a}(t,s)_m\textbf{b}$.
\end{lem}
\begin{proof}
Let $v^{-1} = \phi(\textbf{b})$ and let $\overline{\alpha} = \overline{\theta}((s,t)_m) = (\theta_1,\ldots,\theta_m)$ be the root sequence for $(s,t)_m$.  The first $m$ entries of the associated $\overline{\alpha_s}$- and $\overline{\alpha_t}$-sequences give the root sequences for $(s,t)_m$ and $(t,s)_m$ respectively, so Lemma~\ref{l:sequencerelation} implies that $\overline{\beta} = \overline{\theta}((t,s)_m) = (\theta_m,\ldots,\theta_1) = \overline{\alpha}^R$ is the root sequence for $(t,s)_m$.  Since $\phi((s,t)_m) = \phi((t,s)_m)$,  Lemma~\ref{l:subrs} implies that $\overline{\theta}(\textbf{x}') = \overline{\theta_1}\, \overline{\theta_2}'\, \overline{\theta_3}$, where $\overline{\theta_2}' = v[\overline{\theta}((t,s)_m)] = v[\overline{\beta}]$.  By Lemma~\ref{l:basicrs}, $\overline{\theta_2} = v[\overline{\theta}((s,t)_m) = v[\overline{\alpha}]$.  Thus $\overline{\theta_2}' = \overline{\theta_2}^R$.
\end{proof}
\begin{lem} \label{l:moveinversion}
Let $m = m_{s,t}$ be finite and let $\textbf{x} = \textbf{a}(s,t)_m\textbf{b}$ be a reduced expression for some $w \in W$.  Let $\overline{\theta}(\textbf{x}) = \overline{\theta_1}\,\overline{\theta_2}\,\overline{\theta_3}$ be the decomposition of $\overline{\theta}(\textbf{x})$ respecting $\textbf{x} = \textbf{a}(s,t)_m\textbf{b}$.  If $\Psi$ is the set of roots in the sequence $\overline{\theta_2}$, then $\Psi$ is an inversion $m$-set.
\end{lem}
\begin{proof}
Let $v^{-1} = \phi(\textbf{b})$.  If $\textbf{u} = (s,t)_m$, then the roots in $\overline{\theta}(\textbf{u})$ form an inversion $m$-set $\Psi'$.  Let $\Psi = v(\Psi')$ be the set of roots in $\overline{\theta_2}$.  Then, since all roots in $v(\Psi')$ are in $\Phi(w)$ (and thus positive), Lemma~\ref{l:preservesignspan} implies that $\Psi$ is an inversion $m$-set.
\end{proof}
\begin{defn}
Let $\textbf{x} = \textbf{a}(s,t)_m\textbf{b}$ and $\textbf{x}' = \textbf{a}(t,s)_m\textbf{b}$ be reduced expressions for $w$ that are braid move related.  Let $\overline{\theta_1}\, \overline{\theta_2}\, \overline{\theta_3}$ be the decomposition of $\overline{\theta}(\textbf{x})$ respecting $\textbf{x} = \textbf{a}(s,t)_m\textbf{b}$.  We call the set of roots $\Psi$ in the sequence $\overline{\theta_2}$, the \emph{inversion set of the move $\textbf{x} \rightarrow \textbf{x}'$}.
\end{defn}
\begin{ex}
Let $W = W(A_3)$ with $S = \{s,t,u\}$, and let $\textbf{x} = (u, s, t, s, u)$.  Let $\textbf{a} = \textbf{b} = (u)$ so that $\textbf{x} = \textbf{a} (s,t,s) \textbf{b}$.  Then $\overline{\theta_1} = (\alpha_s + \alpha_t)$, $$\overline{\theta_2} = (\alpha_t + \alpha_u, \alpha_s + \alpha_t + \alpha_u, \alpha_s),$$ and $\overline{\theta_3} = (\alpha_u)$ determine the decomposition of $\overline{\theta}(\textbf{z})$ respecting $\textbf{a} (s,t,s) \textbf{b}$. Then $$\Psi = \{\alpha_t + \alpha_u, \alpha_s + \alpha_t + \alpha_u, \alpha_s\}$$ is the inversion set of the move $(u,s,t,s,u) \rightarrow (u,t,s,t,u)$.
\end{ex}
\begin{lem} \label{l:rootform}
Let $\textbf{x}$ be a reduced expression for $w$ of the form $$\textbf{x} = \textbf{a}\textbf{b}\textbf{c},$$
and let $\overline{\theta}(\textbf{x}) = \overline{\theta_1}\,\overline{\theta_2}\,\overline{\theta_3}$ be the decomposition of $\overline{\theta}(\textbf{x})$ respecting $\textbf{a}\textbf{b}\textbf{c}$.  Suppose that the roots of $\overline{\theta_2}$ form an inversion $m$-set.  Then, there exist distinct $s,t \in S$ such that $\textbf{b} = (s,t)_m$ with $m_{s,t} = m$.
\end{lem}
\begin{proof}
Let $v^{-1} = \phi(\textbf{c})$.  Let $s$ be the last entry of $\textbf{b}$ and $t$ be the $(m-1)$-st entry of $\textbf{b}$.  Then $\alpha_s$ and $s(\alpha_t)$ are in the root sequence $\overline{\theta}(\textbf{b}) = (\theta_1',\ldots,\theta_m')$.  Note that $s(\alpha_t)$ is a positive root in the local Coxeter system $(W,S)_{\{\alpha_s,\alpha_t\}}$.  Since $\alpha_s$, $\alpha_t$ are simple roots, $\Delta_{\{\alpha_s,\alpha_t\}} = \{\alpha_s,\alpha_t\}$ by Lemma~\ref{l:simplerootlocal}.  By hypothesis and Lemma~\ref{l:subrs}, $\{v(\theta_1'),\ldots,v(\theta_m')\}$ is an inversion $m$-set.  Thus there exist $\alpha,\beta \in \Phi$ such that $$\Phi^+_{\{\alpha,\beta\}} = \{v(\theta_1'),\ldots,v(\theta_m')\}$$ with canonical simple roots $\gamma = v(\theta_i')$ and $\delta = v(\theta_j')$, where $1 \leq i,j \leq m$.  Since $\textbf{b}$ is a reduced expression, $v^{-1}(\gamma)$ and $v^{-1}(\delta)$ are positive roots.  By Lemma~\ref{l:preservesignspan} and Definition~\ref{d:inversionmset}, $\{\theta_1',\ldots,\theta_m'\}$ is an inversion $m$-set.  In particular, we must have $\{\theta_1',\ldots,\theta_m'\} = \Phi^+_{\{\alpha_s,\alpha_t\}}$ by Lemma~\ref{l:closure}.\\ \\
Suppose $u \in S$ is an entry occurring in $\textbf{b}$ and suppose $u \neq s,t$ .  Then, for some $\theta_k'$ in the root sequence $\overline{\theta}(\textbf{b})$, the $\alpha_u$ coefficient is nonzero by the reflection formula (\ref{e:reflection}). This contradicts the fact that $$\text{span}(\{\theta_1',\ldots,\theta_m'\}) = \text{span}(\{\alpha_s,\alpha_t\})$$ is a two-dimensional subspace of $V$ by Lemma~\ref{l:inversion2d}.  Thus, every entry in $\textbf{b}$ is either $s$ or $t$, and since $\textbf{b}$ is reduced we have $\textbf{b} = (s,t)_m$ or $\textbf{b} = (t,s)_m$.\\ \\
Since $\left|\Phi^+_{\{\alpha_s, \alpha_t\}}\right| = \left|st\right| = m_{s,t}$, we have $m = m_{s,t}$.
\end{proof}
\noindent
Lemma~\ref{l:revroots} asserts that the effect of applying a braid move upon the root sequence is that the corresponding roots are reversed in the root sequence of the resulting expression.  Lemma~\ref{l:moveinversion} asserts that associated to any braid move is an inversion set.  Thus the roots of an inversion set are the roots being reversed upon applying a braid move.  Lemma~\ref{l:rootform} asserts that when roots of an inversion set appear consecutively in the root sequence of an expression, then there exists a braid in the corresponding portion of the expression.  In light of this, the next definition is natural in the context of braid moves.
\begin{defn}  \label{d:contracted}
Let $\textbf{x}$ be a reduced word for some $w \in W$ with root sequence $\overline{\theta}(\textbf{x})$.  We say that the inversion $k$-set $\Psi$ is \emph{contracted} in $\textbf{x}$ if the roots in $\Psi$ occur as a consecutive subsequence of $\overline{\theta}(\textbf{x})$.  If $\Psi$ is an inversion $k$-set and there exists a reduced expression $\textbf{y}$ for $w$ such that $\Psi$ is contracted in $\textbf{y}$ then we say that $\Psi$ is \emph{contractible in $w$}.  Otherwise we say that $\Psi$ is non-contractible.
\end{defn}
\begin{rem}
Though we do not pursue this, Lemmas~\ref{l:revroots}, \ref{l:moveinversion}, and \ref{l:rootform} combine in a natural way with the reduced expression correspondences of Chapter $4$.  In particular, when an inversion set is contracted, the associated labelings of $\Phi^+$ or $\Phi(w)$ given by Theorem~\ref{t:maintheorem} will be consecutive and applying a braid move reverses the labelings.
\end{rem}
\noindent
Recall that two ``braid move related'' expressions differ by a single braid move.
\begin{lem}
Let $\textbf{x}$ and $\textbf{x}'$ be reduced expressions for $w$ that are braid move related.  Let $\Psi$ be the inversion set of the braid move $\textbf{x} \rightarrow \textbf{x}'$.  If $\Psi' \neq \Psi$ is an inversion set, then the roots of $\Psi' \cap \Phi(w)$ appear in the same order in $\overline{\theta}(\textbf{x}')$ as they do in $\overline{\theta}(\textbf{x})$.
\end{lem}
\begin{proof}
If $|\Psi' \cap \Phi(w)| < 2$, the statement is vacuously true.  Let $$\mu, \nu \in \Psi' \cap \Phi(w)$$ and suppose that $\mu$ and $\nu$ appear in a different order in $\overline{\theta}(\textbf{x}')$ than they do in $\overline{\theta}(\textbf{x})$.  By Lemma~\ref{l:revroots}, $\mu$ and $\nu$ can only appear in a different order in $\overline{\theta}(\textbf{x}')$ as they do in $\overline{\theta}(\textbf{x})$ if $\mu,\nu \in \Psi$.  However, by Lemma~\ref{l:oneroot}, this would imply that $\Psi' = \Psi$.
\end{proof}
\begin{lem} \label{l:detect2set}
Let $\textbf{x}$ be a reduced expression for some $w \in W$.  Let $$\overline{\theta}(\textbf{x}) = (\theta_1,\ldots,\theta_{\ell(w)})$$ and suppose that $B(\theta_i,\theta_{i+1}) = 0$ for some $i$ such that $1 \leq i \leq \ell(w) - 1$.  Then $\Psi = \{\theta_i,\theta_{i+1}\}$ is an inversion $2$-set.
\end{lem}
\begin{proof}
By Corollary~\ref{c:inversionline}, there is a unique inversion set $\Psi$ that contains $\theta_i$ and $\theta_{i+1}$. By calculating $\theta_{i+1}$ and $\theta_i$ in the root sequence $\overline{\theta}(\textbf{x})$, we have that $\theta_{i+1} = v(\alpha_s)$ and $\theta_i = vs(\alpha_t)$ for some $s,t \in S$.  Thus $$B(\theta_i, \theta_{i+1}) = B(vs(\alpha_t), v(\alpha_s)) = B(s(\alpha_t), \alpha_s) = -B(\alpha_t,\alpha_s)$$ since $B$ is invariant under multiplication in $W$.  By hypothesis, $B(\alpha_s,\alpha_t) = 0$ so that $m_{st} = 2$.  By Lemma~\ref{l:moveinversion}, we have that $\Psi = \{\theta_i,\theta_{i+1}\}$ is an inversion $2$-set since the factor $st = (st)_2$ occurs as the associated factor in $\textbf{x}$.
\end{proof}
\begin{lem} \label{l:orthoroots}
Let $\Psi$ be an inversion $m$-set, where $m \geq 2$.  Suppose that $\lambda \not\in \Psi$ and that there are distinct roots $\alpha, \beta \in \Psi$ such that $B(\lambda,\alpha) = 0$ and $B(\lambda,\beta) = 0$.  Then, for any root $\mu$ in $\Psi$, we have $B(\lambda,\mu) = 0$.
\end{lem}
\begin{proof}
We can express any $\mu \in \Psi$ as $\mu = a \alpha + b \beta$ since $$\text{span}(\{\alpha,\beta\}) = \text{span}(\Psi)$$ by Lemma~\ref{l:inversion2d}.  Thus $$B(\lambda,\mu) = B(\lambda, a\alpha + b\beta) = aB(\lambda,\alpha) + bB(\lambda, \beta) = 0,$$ since $B$ is bilinear.
\end{proof}
\begin{lem} \label{l:orthointersections}
Let $\Psi$ be an inversion $m$-set with $m > 2$ and $\Psi'$ be an inversion $m'$-set with $m' > 2$, and assume that $|\Psi \cap \Psi'| = 1$.  Then there exist distinct roots $\alpha \in \Psi$ and $\beta \in \Psi'$, neither of which is in $\Psi \cap \Psi'$, such that $B(\alpha,\beta) \neq 0$. If $\alpha$ and $\beta$ occur together in an inversion $m''$-set $\Psi''$, then $m'' > 2$.
\end{lem}
\begin{proof}
Let $\gamma$ be the root in $\Psi \cap \Psi'$.  Choose $\alpha \in \Psi$ distinct from $\gamma$ such that $B(\alpha,\gamma) \neq 0$.  Since $m > 2$, Lemma~\ref{l:orthoroots} implies such an $\alpha$ exists.  Suppose towards a contradiction that $B(\alpha, \beta) = 0$ for all roots $\beta \in \Psi'$ such that $\beta \neq \gamma$.  Since $m' > 2$, there are two roots in $\Psi'$ orthogonal to $\alpha$, so $\alpha$ is orthogonal to every root in $\Psi'$.  In particular, since $\gamma \in \Psi'$, we have $B(\gamma, \alpha) = 0$, a contradiction.  Thus we may choose $\beta$ such that $B(\alpha,\beta) \neq 0$, so that if $\alpha$ and $\beta$ occur together in an inversion $m''$-set, we have $m'' > 2$.
\end{proof}
\begin{defn} \label{d:commutation}
Let $\textbf{x} = \textbf{a}(s,t)\textbf{b}$ and suppose $m_{s,t} = 2$.  Then we say $\textbf{x}' = \textbf{a}(t,s)\textbf{b}$ is a \emph{$2$-braid neighbor of $\textbf{x}$}.  The reflexive and transitive closure of the $2$-braid neighbor relation is an equivalence relation that we call \emph{commutation equivalence}.  The equivalence classes are called \emph{commutation classes}.  The set of commutation classes for the reduced expressions for $w$ is denoted $\mathcal{C}(w)$.  We denote commutation equivalent expressions $\textbf{x}$ and $\textbf{x}'$ by $\textbf{x} \sim_C \textbf{x}'$.
\end{defn}
\begin{rem}
We can partition the reduced expressions for $w$ into commutation classes.  See \cite[Section 1.1]{onfc} for details.
\end{rem}
\begin{lem} \label{l:cancelsim}
Let $\textbf{x} = \textbf{a}\textbf{b}$ and let $\textbf{x}' = \textbf{a}\textbf{b}'$.  Suppose $\textbf{b} \sim_C \textbf{b}'$.  Then $\textbf{x} \sim_C \textbf{x}'$.
\end{lem}
\begin{proof}
The sequence of 2-braid moves transforming $\textbf{b}$ into $\textbf{b}'$ can be applied to transform $\textbf{x}$ into $\textbf{x}'$ via 2-braid moves.
\end{proof}
\begin{defn} \label{d:longinversion}
Let $\Psi$ be an inversion set.  If $|\Psi| = 2$ then we call $\Psi$ a \emph{short inversion set}.  If $|\Psi| > 2$, we call $\Psi$ a \emph{long inversion set}.  The set of all contractible long inversion sets is denoted $\text{CInv}(w)$.
\end{defn}
\begin{rem}
By Proposition~\ref{p:localdisjoint}, if $|\Psi| = 2$, then $m = |s_\gamma s_\delta| = 2$.  Since $B(\gamma,\delta) = -\cos(\pi/|s_\gamma s_\delta|)$, the two roots in a short inversion set are necessarily orthogonal relative to $B$.  In a long inversion set there exists at least one pair of nonorthogonal roots.  The notion of ``long inversion set'' generalizes the notion of ``inversion triple'', as defined by Green and Losonczy in \cite{fbI}, from the context of simply-laced Coxeter groups to that of arbitrary Coxeter groups.  That is, long inversion sets are inversion triples (and vice versa) whenever $W$ is simply-laced, but long inversion sets need not be inversion triples in a Coxeter group that is not simply laced.
\end{rem}
\begin{lem}  \label{l:inversioncontracts}
Let $w \in W$ and let $\alpha, \beta \in \Phi(w)$.  Suppose there exist reduced expressions $\textbf{x}$ and $\textbf{x}'$ such that $\alpha$ occurs before $\beta$ in $\overline{\theta}(\textbf{x})$, but $\beta$ occurs before $\alpha$ in $\overline{\theta}(\textbf{x}')$.  Then there exists a contractible inversion set $\Psi$ of $w$ that contains $\alpha$ and $\beta$.
\end{lem}
\begin{proof}
By Theorem~\ref{t:matsumototits}, there exists a sequence of braid moves transforming $\textbf{x}$ into $\textbf{x}'$.  By Lemma~\ref{l:revroots} and Lemma~\ref{l:rootform}, each braid move reverses the order in the root sequence of the roots in an inversion $m$-set $\Psi$.  If such a $\Psi$ contains both $\alpha$ and $\beta$, then it is a contractible inversion set of $w$ containing both $\alpha$ and $\beta$.  Otherwise, Lemma~\ref{l:revroots} implies that $\alpha$ and $\beta$ remain in the same order.  Since $\alpha$ and $\beta$ occur in a different order in $\overline{\theta}(\textbf{x}')$ than they do in $\overline{\theta}(\textbf{x})$, there must exist an inversion set $\Psi$ containing both $\alpha$ and $\beta$.
\end{proof}
\begin{cor} \label{c:2setlemma}
Let $w \in W$ and let $\alpha, \beta \in \Phi(w)$.  Suppose $\alpha$ and $\beta$ do not occur together in a contractible long inversion set and suppose there exist reduced expressions $\textbf{x}$ and $\textbf{x}'$ for $w$ where $T_{\textbf{x}}(\alpha) > T_{\textbf{x}}(\beta)$ and $T_{\textbf{x}'}(\alpha) < T_{\textbf{x}'}(\beta)$.  Then there exists a contractible inversion $2$-set $\Psi$ containing $\alpha$ and $\beta$ so that $\alpha \perp \beta$.
\end{cor}
\begin{proof}
By Lemma~\ref{l:inversioncontracts}, there exists an inversion $\Psi$ containing $\alpha$ and $\beta$.  Since $\alpha$ and $\beta$ are not contained in a long inversion set, $\Psi$ is an inversion $2$-set by Definition~\ref{d:longinversion}.
\end{proof}
\section{Freely braided elements}
The elements whose contractible long inversion sets are pairwise disjoint are the freely braided elements.  They were introduced by Green and Losonczy in \cite{fbI} for simply-laced Coxeter groups.  Our definition is for arbitrary Coxeter groups, but in many places our exposition closely follows that of \cite{fbI}.
\begin{defn} \label{d:freelybraided}
Let $w \in W$.  We say that $w$ is \emph{freely braided} if for every pair of contractible long inversion sets $\Omega,\Omega' \in \text{CInv}(w)$, we have $\Omega \cap \Omega' = \emptyset$.
\end{defn}
\begin{lem} \label{l:contractibleorthogonal}
Let $\Psi = \{\gamma_1,\ldots,\gamma_m\}$ be a contractible long inversion set.  Let $\textbf{x}$ be reduced expression for a freely braided element $w \in W$ and let $\alpha \in \Phi(w)$ be such that $\alpha \not\in \Psi$.  Then the following implications hold:
\begin{enumerate}[(1)]
\item If $T_{\textbf{x}}(\gamma_1) < T_{\textbf{x}}(\alpha) < T_{\textbf{x}}(\gamma_2)$, then $\alpha \perp \gamma_1$.
\item  If $T_{\textbf{x}}(\gamma_{m-1}) < T_{\textbf{x}}(\alpha) < T_{\textbf{x}}(\gamma_m)$, then $\alpha \perp \gamma_m$.
\item If $m > 3$ and $T_{\textbf{x}}(\gamma_i) < T_{\textbf{x}}(\alpha) < T_{\textbf{x}}(\gamma_{i+1})$ for some $i$ satisfying the inequality $1 < i < m-1$, then $\alpha \perp \mu$ for all $\mu \in \Psi$.
\end{enumerate}
\end{lem}
\begin{proof}
Recall, by Lemmas~\ref{l:standard} and \ref{l:wlabel}(3), that $T_{\textbf{x}}$ is monotone with respect to the indices of elements $\gamma_i$ of a local root sequence.  That is, we either have $T_{\textbf{x}}(\gamma_1) < \cdots < T_{\textbf{x}}(\gamma_m)$ or $T_{\textbf{x}}(\gamma_m) < \cdots < T_{\textbf{x}}(\gamma_1)$.\\ \\
We first prove statement $(1)$, so suppose that $T_{\textbf{x}}(\gamma_1) < T_{\textbf{x}}(\alpha) < T_{\textbf{x}}(\gamma_2)$.  The monotonicity of $T_{\textbf{x}}$ with respect to $\Psi$ implies that $T_{\textbf{x}}(\alpha) < T_{\textbf{x}}(\gamma_i)$ for $i \geq 2$.  Since $\Psi$ is contractible, there exists a reduced expression $\textbf{x}'$ for $w$ in which $\Psi$ is contracted.  Since $\Psi$ is contracted in $\textbf{x}'$
we either have $T_{\textbf{x}'}(\alpha) < T_{\textbf{x}'}(\gamma_i)$ for all $1 \leq i \leq m$ or  $T_{\textbf{x}'}(\gamma_i) < T_{\textbf{x}'}(\alpha)$ for all $1 \leq i \leq m$.\\ \\
In the case that $T_{\textbf{x}'}(\alpha) < T_{\textbf{x}'}(\gamma_i)$ for all $1 \leq i \leq m$, we have $T_{\textbf{x}}(\gamma_1) < T_{\textbf{x}}(\alpha)$ and $T_{\textbf{x}'}(\alpha) < T_{\textbf{x}'}(\gamma_1)$. By Corollary~\ref{c:2setlemma}, we have $\alpha \perp \gamma_1$.\\ \\
In the case that $T_{\textbf{x}'}(\gamma_i) < T_{\textbf{x}'}(\alpha)$ for all $1 \leq i \leq m$, we have in particular that $T_{\textbf{x}'}(\gamma_{m-1}) < T_{\textbf{x}'}(\alpha)$ and $T_{\textbf{x}'}(\gamma_m) < T_{\textbf{x}'}(\alpha)$.  Since $$T_{\textbf{x}}(\alpha) < T_{\textbf{x}}(\gamma_{m-1}) < T_{\textbf{x}}(\gamma_m),$$ Corollary~\ref{c:2setlemma} implies that $\alpha \perp \gamma_{m-1}$ and that $\alpha \perp \gamma_m$.  Thus, in this case, $\alpha \perp \gamma_1$ by Lemma~\ref{l:orthoroots} and statement $(1)$ follows.\\ \\
The argument for statement $(2)$ can be obtained by interchanging $\gamma_1$ and $\gamma_m$ in the proof of statement $(1)$, by interchanging $\gamma_2$ and $\gamma_{m-1}$ in the proof of statement $(1)$, and by reversing the inequalities in the proof of statement $(1)$.\\ \\
Turning to statement $(3)$, suppose that $m > 3$ and $$T_{\textbf{x}}(\gamma_i) < T_{\textbf{x}}(\alpha) < T_{\textbf{x}}(\gamma_{i+1})$$ for some $i$ satisfying $1 < i < m-1$.  The monotonicity of $T_{\textbf{x}}$ with respect to $\Psi$ implies that $T_{\textbf{x}}(\gamma_j) < T_{\textbf{x}}(\alpha)$ for $1 \leq j \leq i$ and $T_{\textbf{x}}(\alpha) < T_{\textbf{x}}(\gamma_j)$ for $i + 1 \leq j \leq m$.  Since $\Psi$ is contractible, there
exists a reduced expression $\textbf{x}'$ for $w$ such that $\Psi$ is contracted in $\textbf{x}'$.  Since $\Psi$ is contracted in $\textbf{x}'$, we either have $T_{\textbf{x}'}(\alpha) < T_{\textbf{x}'}(\gamma_i)$ for all $1 \leq i \leq m$, or $T_{\textbf{x}'}(\gamma_i) < T_{\textbf{x}'}(\alpha)$ for all $1 \leq i \leq m$.\\ \\
In the case that $T_{\textbf{x}'}(\alpha) < T_{\textbf{x}'}(\gamma_i)$ for all $1 \leq i \leq m$, we have in particular that $T_{\textbf{x}'}(\alpha) < T_{\textbf{x}'}(\gamma_1)$ and $T_{\textbf{x}'}(\alpha) < T_{\textbf{x}'}(\gamma_2)$.  As $T_{\textbf{x}}(\alpha) > T_{\textbf{x}}(\gamma_1)$ and $T_{\textbf{x}}(\alpha) > T_{\textbf{x}}(\gamma_2)$, Corollary~\ref{c:2setlemma} implies that $\alpha \perp \gamma_1$ and $\alpha \perp \gamma_2$.  Thus, $\alpha \perp \mu$ for all $\mu \in \Psi$ by Lemma~\ref{l:orthoroots}.\\ \\
In the case that $T_{\textbf{x}'}(\alpha) > T_{\textbf{x}'}(\gamma_i)$ for all $1 \leq i \leq m$, we have in particular that $T_{\textbf{x}'}(\alpha) > T_{\textbf{x}'}(\gamma_m)$ and $T_{\textbf{x}'}(\alpha) > T_{\textbf{x}'}(\gamma_{m-1})$.  Since $T_{\textbf{x}}(\alpha) < T_{\textbf{x}}(\gamma_{m-1})$ and $T_{\textbf{x}}(\alpha) < T_{\textbf{x}}(\gamma_m)$, Corollary~\ref{c:2setlemma} implies that $\alpha \perp \gamma_{m-1}$ and that $\alpha \perp \gamma_m$.  By Lemma~\ref{l:orthoroots}, $\alpha \perp \mu$ for all $\mu \in \Psi$, proving statement $(3)$.
\end{proof}
\noindent
Recall that sequence multiplication is given by concatenation, so that $\overline{\theta_1} (\alpha)$ refers to the concatenation of the sequence $\overline{\theta_1}$ with the length $1$ sequence $(\alpha)$.
\begin{lem} \label{l:moveroots}
Let $w \in W$ and let $\textbf{x}$ be a reduced expression for $w$ that factors as $\textbf{x} = \textbf{x}_1\textbf{x}_2 (s) \textbf{x}_3 \textbf{x}_4$ and let $\overline{\theta}(\textbf{x}) = \overline{\theta_1} \,\overline{\theta_2} (\alpha) \overline{\theta_3}\, \overline{\theta_4}$ be the decomposition of $\overline{\theta}(\textbf{x})$ respecting $\textbf{x}_1\textbf{x}_2 (s) \textbf{x}_3 \textbf{x}_4$.  Suppose that $\alpha \perp \theta$ for every $\theta \in \overline{\theta_2}$.  Then $\textbf{x}' = \textbf{x}_1 (s) \textbf{x}_2 \textbf{x}_3 \textbf{x}_4$ is a reduced expression for $w$ such that $\textbf{x} \sim_C \textbf{x}'$ and $\overline{\theta}(\textbf{x}') = \overline{\theta_1} (\alpha) \overline{\theta_2}\, \overline{\theta_3}\,\overline{\theta_4}$ is the decomposition of $\overline{\theta}(\textbf{x}')$ respecting $\textbf{x}_1 (s) \textbf{x}_2 \textbf{x}_3 \textbf{x}_4$.\\ \\
Similarly, if $\alpha \perp \theta$ for every $\theta \in \overline{\theta_3}$, then $\textbf{x}' = \textbf{x}_1 \textbf{x}_2 \textbf{x}_3 (s) \textbf{x}_4$ is a reduced expression for $w$ such that $\textbf{x} \sim_C \textbf{x}'$ and $\overline{\theta}(\textbf{x}') = \overline{\theta_1}\, \overline{\theta_2}\, \overline{\theta_3} (\alpha) \overline{\theta_4}$ is the decomposition of $\overline{\theta}(\textbf{x}')$ respecting $\textbf{x}_1 \textbf{x}_2 \textbf{x}_3 (s) \textbf{x}_4$.
\end{lem}
\begin{proof}
Let $\overline{\theta_2} = (\alpha_1,\ldots,\alpha_k)$.  Since $\alpha \perp \alpha_k$, Lemma~\ref{l:detect2set} implies that $\Psi = \{\alpha,\alpha_k\}$ is an inversion $2$-set that is contracted in $\overline{\theta}(\textbf{x})$.  Thus, by Lemma~\ref{l:rootform}, $\textbf{x}_2 = \textbf{x}_2' (t)$, where $m_{s,t} = 2$.  Let $\textbf{x}' = \textbf{x}_1 \textbf{x}_2 (s,t) \textbf{x}_3 \textbf{x}_4$ be the reduced expression formed by applying the braid move $(s,t) = (t,s)$.  Furthermore, by Lemma~\ref{l:revroots}, $$\overline{\theta}(\textbf{x}') = \overline{\theta_1}(\alpha_1,\ldots,\alpha_{k-1},\alpha,\alpha_k) \overline{\theta_3}\,\overline{\theta_4}.$$  Also, since $m_{s,t} = 2$, we have $\textbf{x}' \sim_C \textbf{x}$.\\ \\
Since $\alpha \perp \alpha_1,\ldots,\alpha_{k-1}$ we may proceed inductively to obtain a reduced expression $\textbf{y}$ such that $\overline{\theta}(\textbf{y}) = \overline{\theta_1}\,(\alpha)\overline{\theta_2}\,\overline{\theta_3}\,\overline{\theta_4}$.  Furthermore, $\textbf{y} \sim_C \textbf{x}$.\\ \\
The statement for the situation in which $\alpha \perp \theta$ for every $\theta$ in $\overline{\theta_3}$ is proven similarly.
\end{proof}
\begin{lem} \label{l:maintaincontract}
Let $w \in W$ be freely braided.  Let $\textbf{x}$ be a reduced expression for $w$ that factors as $\textbf{x} = \textbf{x}_1 (s) \textbf{x}_2 (t) \textbf{x}_3$ and let $\overline{\theta}(\textbf{x}) = \overline{\theta_1}(\alpha) \overline{\theta_2}(\beta)\overline{\theta_3}$ be the decomposition of $\overline{\theta}(\textbf{x})$ respecting $\textbf{x}_1 (s) \textbf{x}_2 (t) \textbf{x}_3$.  Suppose that $\Psi$ is a contractible long inversion set containing both $\alpha$ and $\beta$ and that $\Psi' \neq \Psi$ is a long inversion set that is contracted in $\textbf{x}$.  Then every root in $\Psi'$ occurs consecutively in exactly one of the subsequences $\overline{\theta_1}$, $\overline{\theta_2}$, or $\overline{\theta_3}$.
\end{lem}
\begin{proof}
Suppose that $\Psi'$ has roots in more than one of the subsequences.  Since $\Psi'$ is contracted in $\textbf{x}$,
the roots of $\Psi'$ occur consecutively in $\overline{\theta}(\textbf{x})$.  It follows that $\Psi'$ contains $\alpha$ or $\Psi'$ contains $\beta$.  Thus $\Psi \cap \Psi' \neq \emptyset$, contradicting the assumption that $w$ is freely braided.
\end{proof}
\begin{defn}
Let $\textbf{x}$ be a reduced expression for a freely braided element $w \in W$.  We say that $\textbf{x}$ is a \emph{contracted reduced expression} for $w$ if every long contractible inversion set of $w$ is contracted in $\overline{\theta}(\textbf{x})$.
\end{defn}
\begin{lem}  \label{l:cre}
Let $w \in W$ be a freely braided element.  Then there exists a contracted reduced expression $\textbf{x}$ for $w$.  Furthermore, every reduced expression $\textbf{x}$ for $w$ is commutation equivalent to a contracted reduced expression.
\end{lem}
\begin{proof}
Let $\textbf{y}$ be a reduced expression such that the number of contracted long inversion sets in $\textbf{y}$ is $n < N(w)$.  Let $\Psi = \{\gamma_1,\ldots,\gamma_m\}$ be a contractible long inversion set that is not contracted in $\textbf{y}$.  We suppose without loss of generality that the order in which the roots of $\Psi$ appear in $\overline{\theta}(\textbf{y})$ is $\gamma_1,\ldots,\gamma_m$.\\ \\
Since $m \geq 3$, $\gamma_2$ is a root that is neither $\gamma_1$ nor $\gamma_m$.  Factor $\textbf{y}$ as $$\textbf{y} = \textbf{y}_1 (s) \textbf{y}_2 (t) \textbf{y}_3$$ so that the decomposition of $\overline{\theta}(\textbf{y})$ respecting $\textbf{y}_1 (s) \textbf{y}_2 (t) \textbf{y}_3$ is given by $$\overline{\theta}(\textbf{y}) = \overline{\theta_1} (\gamma_2) \overline{\theta_2} (\gamma_3) \overline{\theta_3}.$$
By Lemma~\ref{l:contractibleorthogonal}, every root $\lambda$ between $\gamma_3$ and $\gamma_2$ in $\overline{\theta}(\textbf{y})$ is orthogonal to $\gamma_3$.  Thus, by Lemma~\ref{l:moveroots}, the expression $\textbf{y}' = \textbf{y}_1 (s,t) \textbf{y}_2 \textbf{y}_3$ is a reduced expression for $w$ such that $\textbf{y} \sim_C \textbf{y}'$ and $\overline{\theta}(\textbf{y}') = \overline{\theta_1} (\gamma_2) (\gamma_3) \overline{\theta_2} \, \overline{\theta_3}$ is the decomposition of $\overline{\theta}(\textbf{y}')$ respecting $\textbf{y}_1 (s,t) \textbf{y}_2 \textbf{y}_3$.\\ \\
Let $\Psi' \neq \Psi$ be an inversion set that is contracted in $\textbf{y}$.  By Lemma~\ref{l:maintaincontract}, $\Psi'$ has all of its roots in exactly one of the subsequences $\overline{\theta_1}$, $\overline{\theta_2}$, or $\overline{\theta_3}$.  Thus, $\Psi'$ remains contracted in $\textbf{y}'$.\\ \\
We repeat this procedure of shifting $\gamma_{i+1}$ to the left until it is consecutive with $\gamma_i$ for all $2 \leq i \leq m-1$.  This results in a reduced expression $\textbf{y}''$ in which $\gamma_2,\ldots,\gamma_m$ are consecutive in $\overline{\theta}(\textbf{y}'')$, and by the same logic as above, all inversion sets that are contracted in $\textbf{y}$ remain contracted in $\textbf{y}''$.  Furthermore, $\textbf{y}'' \sim_C \textbf{y}$.\\ \\
With $\gamma_2,\ldots,\gamma_m$ consecutive in $\overline{\theta}(\textbf{y}'')$, we apply Lemma~\ref{l:contractibleorthogonal} to get that every root $\lambda$ between $\gamma_1$ and $\gamma_2$ in $\overline{\theta}(\textbf{y}'')$ is orthogonal to $\gamma_1$.  Thus we may apply
Lemma~\ref{l:moveroots} to shift $\gamma_1$ to the right to obtain a reduced expression $\textbf{y}'''$ in which $\gamma_1$ is consecutive with $\gamma_2$, and where the contracted long inversion sets of $\textbf{y}$ are contracted in $\textbf{y}'''$.  Furthermore, $\textbf{y}''' \sim_C \textbf{y}$.\\ \\
The result is that $\textbf{y}'''$ has $n + 1$ long inversion sets that are contracted.  By induction, there exists a reduced expression $\textbf{x}$ such that there are $N(w)$ contracted long inversion sets.\\ \\
Since the only moves required to transform a reduced expression into a contracted reduced expression are commutation moves, the last statement of the lemma follows.
\end{proof}
\begin{defn} \label{d:commutationorder}
Let $\textbf{x}$ be a reduced expression for a Coxeter group element $w$ and let $\overline{\theta}(\textbf{x})$ be the associated root sequence.  Let $$R = \{(\theta_i,\theta_j) \; : \; i < j \text{ and } \theta_i \not\perp \theta_j \}.$$  Then the reflexive and transitive closure of $R$ is a partial order on $\Phi(w)$ that we denote by $\leq_{\overline{\theta}(\textbf{x})}$.
\end{defn}
\begin{rem}
Given distinct reduced expressions $\textbf{x}$ and $\textbf{x}'$ for $w$, the partial orders $\leq_{\overline{\theta}(\textbf{x})}$ and $\leq_{\overline{\theta}(\textbf{x}')}$ are both binary relations on $\Phi(w)$.  Thus, when we speak of equality of these partial orders in the sequel, we mean that $\leq_{\overline{\theta}(\textbf{x})}$ and $\leq_{\overline{\theta}(\textbf{x}')}$, viewed as subsets of $\Phi(w) \times \Phi(w)$, are equal.
\end{rem}
\noindent
The statement and proof of \cite[Proposition 3.1.5]{fbI}, which is formed in the context of simply-laced Coxeter group generalizes exactly.  We reproduce the proof and statement here for convenience.
\begin{prop}[\textbf{Green and Losonczy}]  \label{p:propheap}
Let $w \in W$.  Let $\textbf{x}$ and $\textbf{x}'$ be reduced expressions for $w$.  Then $\leq_{\overline{\theta}(\textbf{x})}$ and $\leq_{\overline{\theta}(\textbf{x}')}$ are equal as partially ordered sets if and only if $\textbf{x} \sim_C \textbf{x}'$.
\end{prop}
\begin{proof}
Suppose $\textbf{x}$ and $\textbf{x}'$ are reduced expressions that differ by a single $2$-braid move.  Then, by Lemma~\ref{l:rootform} and Definition~\ref{d:commutationorder}, $\alpha \leq_{\overline{\theta}(\textbf{x})} \beta$ if and only if $\alpha \leq_{\overline{\theta}(\textbf{x}')} \beta$.  Thus if $\textbf{x}$ and $\textbf{x}'$ are commutation equivalent, then $\leq_{\overline{\theta}(\textbf{x})}$ and $\leq_{\overline{\theta}(\textbf{x}')}$ are equal.\\ \\
Conversely, suppose $\leq_{\overline{\theta}(\textbf{x})}$ and $\leq_{\overline{\theta}(\textbf{x}')}$ are equal as partially ordered sets.  Write $\overline{\theta}(\textbf{x}) = (\theta_1,\ldots,\theta_n)$ and $\overline{\theta}(\textbf{x}') = (\theta_1',\ldots,\theta_n')$.  Let $\pi \in S_n$ be the permutation of the indices satisfying $\theta_i = \theta_{\pi(i)}'$.  Suppose $\theta_1$ and $\theta_i$ are nonorthogonal.  Then, since the partial orders are equal, $\pi(1) < \pi(i)$ so that the roots occurring before $\theta_1$ in $\overline{\theta}(\textbf{x}')$ are all orthogonal to $\theta_1$.\\ \\
Let $(\theta_1',\ldots, \theta_k', \theta_{\pi(1)}')$ be the initial subsequence of $\overline{\theta}(\textbf{x}')$.  By Lemma~\ref{l:rootform}, we can form an expression $\textbf{x}''$ by applying a sequence of 2-braid moves starting with $\textbf{x}'$ such that $\overline{\theta}(\textbf{x}'') = (\theta_1'',\ldots,\theta_n'')$ satisfies $\theta_1'' = \theta_1$.  Since the sequence consisted of $2$-braid moves, we have $\textbf{x}'' \sim_C \textbf{x}'$.\\ \\
Now we factor $\textbf{x} = (s_\alpha) \textbf{y}$ and $\textbf{x}'' = (s_\beta) \textbf{y}''$.  Let $v = \phi(\textbf{y})$ and $v'' = \phi(\textbf{y}'')$.  Since the root sequences $\overline{\theta}(\textbf{y})$ and $\overline{\theta}(\textbf{y}'')$ have the same entries, $v = v''$ by Proposition~\ref{p:basiccoxeter} parts $(7)$ and $(8)$.  Similarly, we have $s_\alpha v = s_\beta v''$ so that $s_\alpha = s_\beta$.  By induction, since the partial orders $\leq_{\overline{\theta}(\textbf{y})}$ and $\leq_{\overline{\theta}(\textbf{y}'')}$ are equal, we have $\textbf{y} \sim_C \textbf{y}''$.  By Lemma~\ref{l:cancelsim}, we now have $\textbf{x} \sim_C \textbf{x}''$.  It follows that $\textbf{x} \sim_C \textbf{x}'' \sim_C \textbf{x}'$.
\end{proof}
\noindent
Following \cite{fbI}, we denote the number of contractible long inversion sets of $w$ by $N(w)$.  One result of \cite{fbI} and \cite{fbII} we wish to generalize is that an element $w \in W$ is freely braided if and only if the number of commutation classes of $w$ is $2^{N(w)}$.  Towards this end we introduce an injective map from the set of reduced expressions for $w$ to the set of 0-1 states indexed by the contractible long inversion sets of $w$.
\begin{defn}
Let $w \in W$ and let $\textbf{x}$ be a fixed reduced expression for $w$ with standard encoding $T_{\textbf{x}}$.  Let $\mathcal{R}(w)$ denote the set of reduced expressions for $w$.  We define a map $F_{\textbf{x}}:\mathcal{R}(w) \rightarrow \{0,1\}^{\text{CInv}(w)}$ by
\begin{equation*}
F_{\textbf{x}}(\textbf{x}')\,(\Omega) = \begin{cases} 0 &\text{ if }\Omega \text{ is in the same relative ordering in }
T_{\textbf{x}} \text{ as in } T_{\textbf{x}'}\\
                                      1 &\text{ otherwise.}
\end{cases}
\end{equation*}
\end{defn}
\begin{lem} \label{l:statewd}
Let $w \in W$ and $\textbf{x}$ be a fixed reduced expression for $w$.  Let $\textbf{y}$ and $\textbf{y}'$ be reduced expressions for $w$.  If $\textbf{y} \sim_C \textbf{y}'$, then $F_{
\textbf{x}}(\textbf{y}) = F_{\textbf{x}}(\textbf{y}')$.
\end{lem}
\begin{proof}
Suppose $\textbf{y}$ and $\textbf{y}'$ differ by a $2$-braid move.  By Lemma~\ref{l:revroots}, there exist $\alpha, \beta \in \Phi(w)$, and a label $k$ such that $T_{\textbf{y}}(\alpha) = k$, $T_{\textbf{y}}(\beta) = k + 1$, $T_{\textbf{y}'}(\alpha) = k+1$, and $T_{\textbf{y}'}(\beta) = k$.  Since $\textbf{y}$ and $\textbf{y}'$ differ by only a $2$-braid move, $T_{\textbf{y}}(\gamma) = T_{\textbf{y}'}(\gamma)$ whenever $\gamma$ is neither $\alpha$ nor $\beta$. By Lemma~\ref{l:rootform}, $\{\alpha, \beta\}$ is an inversion $2$-set.  Thus if $\Omega \in \text{CInv}(w)$, then $\Omega \cap \{\alpha, \beta\} \neq \{\alpha,\beta\}$ by Lemma~\ref{l:oneroot}.  It follows that $\Omega$ is in the same relative ordering in $T_{\textbf{y}}$ as in $T_{\textbf{y}'}
$.  Hence, $F_{\textbf{x}}(\textbf{y}) \, (\Omega) = F_{\textbf{x}}(\textbf{y}') \, (\Omega)$ for any $\Omega \in \text{CInv}(w)$.  Repeatedly applying $2$-braid moves gives the same result, so $\textbf{y} \sim_C \textbf{y}'$ implies $F_{\textbf{x}}(\textbf{y}) = F_{\textbf{x}}(\textbf{y}')$.
\end{proof}
\noindent
We now prove the converse of the previous assertion.
\begin{lem} \label{l:stateinjective}
Let $w \in W$ and $\textbf{x}$ be a fixed reduced expression for $w$.  Let $\textbf{y}$ and $\textbf{y}'$ be reduced expressions for $w$.  If $F_{\textbf{x}}(\textbf{y}) = F_{\textbf{x}}(\textbf{y}')$, then $\textbf{y} \sim_C \textbf{y}'$.
\end{lem}
\begin{proof}
Let $\alpha$ and $\beta$ be distinct nonorthogonal roots of $\Phi(w)$.  Without loss of generality, suppose that $\alpha <_{\overline{\theta}(\textbf{y})} \beta$.  Then, since $\alpha$ and $\beta$ lie in some inversion set $\Psi$ by Corollary~\ref{c:inversionline}, we have $F_{\textbf{x}}(\textbf{y})(\Psi) = F_{\textbf{x}}(\textbf{y}')(\Psi)$.  Thus, $\alpha$ and $\beta$ occur in the same relative order in $\overline{\theta}(\textbf{y})$ as they do in $\overline{\theta}(\textbf{y}')$.  It follows that $\leq_{\overline{\theta}(\textbf{y})}$ and $\leq_{\overline{\theta}(\textbf{y}')}$ are the same partial order.  By Proposition~\ref{p:propheap}, we have $\textbf{y} \sim_C \textbf{y}'$.
\end{proof}
\begin{cor} \label{c:Finjective}
The map $F_{\textbf{x}}$ induces a an injective mapping $$F_{\textbf{x}}':\mathcal{C}(w) \rightarrow \{0,1\}^{\text{CInv(w)}}.$$
\end{cor}
\begin{proof}
By Lemma~\ref{l:statewd}, the map $F_{\textbf{x}}$ is well-defined on the commutation classes of $w$.  By Lemma~\ref{l:stateinjective}, $F_{\textbf{x}}$ is injective.
\end{proof}
\begin{theorem} \label{t:fbcharacterization}
Let $w \in W$.  Then $w$ is freely braided if and only if the number of commutation classes of $w$ is $2^{N(w)}$.
\end{theorem}
\begin{proof}
Let $w$ be freely braided and $\textbf{x}$ be a reduced expression for $w$.
Since $w$ is freely braided, Lemma~\ref{l:cre} implies that there is a contracted reduced expression $\textbf{y}$ for $w$.  By Lemma~\ref{l:revroots}, applying a braid move to a contracted long inversion set $\Psi$ results in a contracted reduced expression $\textbf{y}'$ for $w$ such that the roots in $\overline{\theta}(\textbf{y}')$ occur in the reverse order of those in $\overline{\theta}(\textbf{y})$.  Thus, for each contracted long inversion set, we may specify that it be in either order and find a sequence of braid moves that transforms $\textbf{y}$ into a contracted reduced expression with each contracted inversion set in the prescribed order.  Thus, $F_{\textbf{x}}$ is surjective.  By Corollary~\ref{c:Finjective}, we have that $|\mathcal{C}(w)| = 2^{N(w)}$.\\ \\
Conversely, if $w$ is not freely braided, then there exist distinct contractible inversion sets $\Psi$ and $\Psi'$ that intersect in a single root $\delta$.  By Lemma~\ref{l:orthointersections}, there exist $\alpha,\beta \in \Phi^+$, neither of which is $\delta$, such that $\alpha \in \Psi$, $\beta \in \Psi'$, and $B(\alpha, \beta) \neq 0$.  By Corollary~\ref{c:inversionline}, there is an inversion set $\Psi''$ containing $\alpha$ and $\beta$, but not $\delta$.\\ \\
If $\Psi''$ is not contractible, then by Lemma~\ref{l:revroots}, $\alpha$ is before $\beta$ in every root sequence or vice versa.  Suppose without loss of generality that $T_{\textbf{x}}(\alpha) > T_{\textbf{x}}(\beta)$ for all reduced expressions $\textbf{x}$ for $w$.  If there were $2^{N(w)}$ commutation classes for $w$, then there would be a commutation class representative $\textbf{y}$ such that $\Psi$ was in the relative ordering $T_{\textbf{y}}(\alpha) < T_{\textbf{y}}(\delta)$ and $\Psi'$ was in the relative ordering where $T_{\textbf{y}}(\delta) < T_{\textbf{y}}(\beta)$.  This would then imply that $T_{\textbf{y}}(\alpha) < T_{\textbf{y}}(\delta) < T_{\textbf{y}}(\beta) < T_{\textbf{y}}(\alpha)$, a contradiction.\\ \\
If instead, the set $\Psi''$ is a contractible inversion set of $\Phi(w)$, and there are $2^{N(w)}$ commutation classes, then there is a labeling $T$ of $\Phi(w)$ such that the relative ordering of $\Psi$ implies $T(\alpha) < T(\delta)$, that of $\Psi'$ implies $T(\delta) < T(\beta)$ and that of $\Psi''$ implies $T(\beta) < T(\alpha)$.  However, this relative ordering can not exist, for it implies $T(\alpha) < T(\alpha)$.
\end{proof}

\bibliographystyle{plain}

\end{document}